\documentclass[11pt,leqno]{amsart}

\usepackage[all,cmtip]{xy}
\usepackage{a4,latexsym}
\usepackage[utf8]{inputenc}
\usepackage{amsfonts}
\usepackage[hidelinks]{hyperref}
\usepackage{amsxtra,amsmath,amsfonts,amscd, amssymb, mathrsfs, amsthm}
\usepackage{color, mathtools}
\usepackage{bbm}
\usepackage{enumitem}
\usepackage{xcolor}
\usepackage{graphicx}
\usepackage{tikz-cd}
\usepackage{braket}
\usepackage[all]{xypic}
\usepackage[toc,page,title,titletoc,header]{appendix}
\usepackage[capitalize]{cleveref}
\usepackage{bbm}

\numberwithin{paragraph}{section}

\setlist[enumerate]{label=\it{(\roman*)},
	ref=\it{(\roman*)}}

\newlist{enumerate1}{enumerate}{1}
\setlist[enumerate1]{leftmargin=*,label=\it(\arabic*),ref=\it{(\arabic*)}}

\newlist{enumeratea}{enumerate}{1}
\setlist[enumeratea]{leftmargin=*,label=\it(\alph*),ref=\it{(\alph*)}}

\newcommand{\C}{{\mathbb C}}
\newcommand{\R}{{\mathbb R}}
\renewcommand{\P}{{\mathbb P}}
\newcommand{\Q}{{\mathbb Q}}
\newcommand{\Z}{{\mathbb Z}}
\newcommand{\N}{{\mathbb N}}

\renewcommand{\Im}{\mathrm{Im}}

\newcommand{\Frac}{\mathrm{Frac}}
\newcommand{\Char}{\mathrm{Char}}

\newcommand{\metr}{{\|\hspace{1ex}\|}}
\newcommand{\val}{{|\hspace{1ex}|}}

\newcommand{\Spec}{\mathrm{Spec}}

\newcommand{\Supp}{\mathrm{Supp}}
\newcommand{\CH}{\mathrm{CH}}
\newcommand{\Div}{\mathrm{Div}}

\newcommand{\Pic}{\mathrm{Pic}}

\newcommand{\relint}{\mathrm{rel\text{-}int}}
\newcommand{\relnef}{\mathrm{rel\text{-}nef}}
\newcommand{\relsnef}{\mathrm{rel\text{-}snef}}
\newcommand{\arsnef}{\mathrm{ar\text{-}snef}}

\newcommand{\nef}{\mathrm{nef}}
\newcommand{\snef}{\mathrm{snef}}
\newcommand{\integrable}{\mathrm{int}}

\newcommand{\vol}{\mathrm{vol}}
\newcommand{\YZ}{\mathrm{YZ}}

\newcommand{\gm}{\mathrm{gm}}

\newcommand{\mo}{\mathrm{mo}}
\newcommand{\cpt}{\mathrm{cpt}}

\newcommand{\OO}{\mathcal{O}}

\newcommand{\ord}{\mathrm{ord}}
\newcommand{\Gal}{\mathrm{Gal}}

\newcommand{\id}{\operatorname{id}\nolimits}

\renewcommand{\Im}{\operatorname{Im}\nolimits}

\renewcommand{\dim}{\operatorname{dim}\nolimits}

\newcommand{\an}{\mathrm{an}}

\makeatletter
\newcommand\tint{\mathop{\mathpalette\tb@int{t}}\!\int}
\newcommand\bint{\mathop{\mathpalette\tb@int{b}}\!\int}
\newcommand\tb@int[2]{%
	\sbox\z@{$\m@th#1\int$}%
	\if#2t%
	\rlap{\hbox to\wd\z@{%
			\hfil
			\vrule width .45em height \dimexpr\ht\z@+1.4pt\relax depth -\dimexpr\ht\z@+1pt\relax
			\kern.05em 
	}}
	\else
	\rlap{\hbox to\wd\z@{%
			\vrule width .45em height -\dimexpr\dp\z@+1pt\relax depth \dimexpr\dp\z@+1.4pt\relax
			\hfil
	}}
	\fi
}
\makeatother

\newtheorem{theorem}{Theorem}[subsection]

\newtheorem{corollary}[theorem]{Corollary}

\newtheorem{lemma}[theorem]{Lemma}
\newtheorem{lemma*}{Lemma}
\newtheorem{proposition}[theorem]{Proposition}

\theoremstyle{definition}
\newtheorem{definition}[theorem]{Definition}
\newtheorem{proposition&definition}[theorem]{Proposition\&Definition}
\newtheorem{lemma&definition}[theorem]{Lemma\&Definition}
\newtheorem{theorem&definition}[theorem]{Theorem\&Definition}
\newtheorem{example}[theorem]{Example}

\newtheorem{example*}{Example}
\newtheorem{remark}[theorem]{Remark}

\newtheorem{question*}{Question}

\newtheorem{art}[theorem]{}

\newtheorem{mainthm}{Theorem}

\begin{document}
	
	\title[Equidistribution for quasi-projective varieties over function fields]{Equidistribution for quasi-projective varieties over function fields}

	\author[D.~Biswas]{Debam Biswas}	\address{D. Biswas, Yau Mathematical Sciences Center, Tsinghua University, Beijing 100084, China}
	\email{debambiswas@tsinghua.edu.cn}
	
	\author[Y.~Cai]{Yulin Cai}
    \address{Y. Cai, Hangzhou International Innovation Institute, Beihang University, Hangzhou 311115, China}
	\email{ylcai5388339@gmail.com}

	\begin{abstract}
We prove equidistribution of small points and subvarieties on quasi-projective varieties over function fields, confirming \cite[Conjecture~5.4.1]{yuan2021adelic} for compactified relatively nef metrized line bundles in the function field case. The argument relies on the generic curves to reduce higher-dimensional bases to curves. We also establish the equidistribution theorem for arithmetically big line bundles using positive intersection.

	\end{abstract}
	
	
	\maketitle
	
	\setcounter{tocdepth}{1}
	
	\tableofcontents

\section{Introduction}

Let $U$ be a $d$-dimensional algebraic variety over a field $K$ that has a height function $h$, and $\{x_m\}_{m\in I} \subset U(\overline{K})$ a net of algebraic points. For each place $v$ of $K$, we let $K_v$ be the completion of $K$ at $v$, let $\C_v$ be the completion of an algebraic closure of $K_v$, and fix an embedding $\overline{K}\hookrightarrow \C_v$. Via this embedding, we can view $U(\overline{K})$ as a subset of $U^\an_v$, the Berkovich analytification of $U_v\coloneq U\times_{\Spec(K)}\Spec(K_v)$ (see \cite{berkovich1990spectral}). The Galois group $\Gal(\overline{K}/K)$ acts on $U$, and we denote the orbit of $x\in U(\overline{K})$ by $O(x)$. The (arithmetic) equidistribution theorem at $v$ (for \emph{small points}) is formulated as follows: under a certain condition on the heights $h(x_m)$, there exists a measure $\mu_v$ on $U_v^\an$ such that for any $f_v\in C_c(U_v^\an)$, we have
\[\lim\limits_{m\in I}\frac{1}{|O(x_m)|}\sum\limits_{x_m^{\sigma}\in O(x_m)}f_v(x_m^{\sigma})= \int_{U_v^\an}f_v\, \mu_v,\] 
where $C_c(U_v^\an)$ is the space of compactly supported real-valued continuous functions on $U_v^\an$, and $O(x_m)$ is the Galois orbit of $x_m$. 

In Arakelov theory, the field $K$ can be a number field, a function field (not necessarily of one variable), or more generally a field equipped with an adelic curve structure $S$ in the sense of \cite[Section~3.1]{chen2020arakelov}. The height function $h=h_{\overline{L}}$ is given by some (global) \emph{metrized line bundle} $\overline{L}=(L,\metr)$ on $U$, and the measure $\mu_v$ depends only on the localization $\overline{L}_v=(L_v,\metr_v)$ of $\overline{L}$ at $v$. For example, when $U=X$ is projective over a number field and $\overline{L}$ is a \emph{semipositive} metrized line bundle on $X$ in the sense of \cite{zhang1995small}, then $\mu_v$ is the \emph{Monge--Amp\'ere measure} $c_1(\overline{L}_v)^d$ associated to $\overline{L}_v$, constructed in \cite{bedford1976dirchlet} in the archimedean case and in \cite{chambert2012formes} in the non-archimedean case. 

To formulate the hypothesis on $h(x_m)$ in the equidistribution theorem, we need more assumptions on $\overline{L}$. For example, if $U=X$ is projective over a number field $K$, and $\overline{L}$ is semipositive in the sense of \cite{zhang1995small}, we denote by $(\overline{L}^{d+1}\mid X)_S$ the \emph{arithmetic intersection number} of $\overline{L}$, and by $\deg_{L}(X)$ the \emph{degree} of $X$ with respect to $L$. Then one may assume that $\{x_m\}_{m\in I}$ is \emph{small} in the sense that
\begin{align*}\label{eq:introduction1}
	\lim\limits_{m\in I}h_{\overline{L}}(x_m)=\frac{(\overline{L}^{d+1}\mid X)_S}{(d+1)\deg_{{L}}(X)}.
\end{align*}
See \cite[Theorem~1.2]{ballay2024approximation} for a more general (smallness) assumption on $(x_m)_{m\in I}$.  

For the number field case, the pioneering work of Szpiro--Ullmo--Zhang \cite{szpiro1997equirepartition}, Ullmo \cite{ullmo1998positivite} and Zhang \cite{zhang1998equidistribution} established the equidistribution theorem in the case where $U=X$ is projective, $\overline{L}=(L,\metr)$ is induced by an ample model and $v$ is archimedean, which played a key role in the proofs of the Bogomolov conjecture in \cite{ullmo1998positivite,zhang1998equidistribution}. For a non-archimedean place $v$, Chambert-Loir \cite[Th\'eor\`eme~3.1]{chambert2006measures} proved the equidistribution theorem when $U=C$ is a projective curve and $\metr_v$ is induced by an ample line bundle on some model of $C_v$. Yuan \cite{yuan2008big} proved the equidistribution theorem at any place $v$ only assuming that $L$ is ample and $\metr$ is semipositive in the sense of \cite{zhang1995small}.

For the function field case, Gubler proved a tropical version of the equidistribution theorem over closed subvarieties of abelian varieties in \cite[Theorem~5.5]{gubler2007the}, which was used to prove the Bogomolov conjecture for totally degenerate abelian varieties over function fields. Afterwards, the equidistribution theorem over projective varieties was shown in \cite[Theorem~1.1]{gubler2008equidistribution} when $\overline{L}$ is a semipositive admissible metrized line bundle in the sense of \cite[3.1]{gubler2008equidistribution}. The theorem was also proved independently by Faber \cite[Theorem~1.1]{faber2009equidistribution} when $K$ is the function field of a projective curve.

In \cite{moriwaki2000arithmetic}, Moriwaki considered Arakelov theory over a finitely generated field $K$ over $\Q$. He proved the equidistribution theorem if $X$ is projective over $K$ and $\overline{L}$ satisfies certain ``nef'' conditions; see \cite[Theorem~6.1]{moriwaki2000arithmetic}. His theorem was generalized by Chen--Moriwaki \cite[Theorem~8.11.2]{chen2020arakelov} to the adelic curve case. 

All these results are for projective varieties. Recently, Yuan--Zhang developed a theory of adelic line bundles over quasi-projective varieties and proved the equidistribution theorem \cite[Theorem~5.4.3]{yuan2021adelic} in the case where $K$ is a number field or a function field of a projective curve, $U$ is a quasi-projective variety and $\overline{L}$ is \emph{nef} in the sense of \cite[Definition~2.6.2]{yuan2021adelic}.

In this paper, we consider the equidistribution theorem when $K$ is the function field of a higher-dimensional projective variety, $U$ is a quasi-projective variety and $\overline{L}$ is the limit of a sequence of relatively nef model line bundles on some model $\mathcal{X}$ of $X$; see \cref{def:boundarytopologyglobal} for the precise definition. Notice that such an $\overline{L}$ is relatively nef in the sense of \cite[7.2]{cai2024abstract}. We fix an adelic curve $S = (K=k(T), \Omega,\mathcal{A},\nu)$ given by an integral, normal projective variety $T$ over a field $k$ and an ample class $\mathbf{c}\in \Pic(T)$; see \cref{adelic curve defined by projective variety}. Let $U$ be a $d$-dimensional quasi-projective variety over $K$. Then there is a symmetric multilinear map on the set of \emph{integrable} $\Q$-line bundles
\[\underbrace{\widehat{\Pic}_{S,\Q}({U})_\integrable^\YZ\times\cdots\times\widehat{\Pic}_{S,\Q}({U})_\integrable^\YZ}_{(d+1)\text{-times}}\to \R, \ \ (\overline{L_0}, \dots, \overline{L_d})\mapsto (\overline{L_0}\cdots\overline{L_d}\mid U)_S\] 
induced by the arithmetic intersection numbers. It can be extended to the cone $\widehat{\Pic}_{S,\Q}(U)_\relnef^{\nef,\YZ}$; see \cref{extension of YZ intersection number} and \cref{rmk:arithmetic intersection number for subvarieties}. In particular, for any $\overline{L}\in \widehat{\Pic}_{S,\Q}(U)_\relnef^{\nef,\YZ}$, the $0$-dimensional case of this map induces a height function $h_{\overline{L}}\colon U(\overline{K})\to \R$.

To state the main result, we need to extend the notions from the projective case to the quasi-projective case. For each $v\in\Omega$ (which corresponds to a non-archimedean valuation $\val_v$ on $K=k(T)$), and any $\overline{L}\in\widehat{\Pic}_{S,\Q}(U)_{\relnef}^{\nef,\YZ}$, there exists a regular Borel measure $c_1(\overline{L}_v)^{d}$ on the Berkovich analytification $U^\an_v$ of $U_v=U\times_{\Spec(K)}\Spec(K_v)$; see \cite[3.6.7]{yuan2021adelic}.

Let $(x_m)_{m\in I}\subset U(\overline{K})$ be a net of algebraic points. We say that $(x_m)_{m\in I}$ is \emph{generic} if for any proper closed subset $Y\subset U$, there exists $m_0\in I$ such that $x_m\not\in Y$ for all $m\geq m_0$. We say that $(x_m)_{m\in I}$ is \emph{small} (\emph{with respect to $\overline{L}$}) if 
\[\lim\limits_{m\in I}h_{\overline{L}}(x_m)=\frac{(\overline{L}^{d+1}\mid U)_S}{(d+1)\deg_{\widetilde{L}}(U)},\]
where $\widetilde{L}$ is the (geometric) compactified $\Q$-line bundle determined by $\overline{L}$; see \cref{def:boundarytopologyglobal}, and $\deg_{\widetilde{L}}(U)$ is the auto-intersection number of $\widetilde{L}$ on $U$; see \cref{geometric intersection number}. 

We prove the following \emph{equidistribution theorem}, which gives a positive answer to \cite[Conjecture~5.4.1]{yuan2021adelic} for the function field case:

\begin{mainthm}[\cref{thm:equidistribution over function fields}] \label{intro:equidistribution over function fields}
	Let $\overline{L}\in \widehat{\Pic}_{S,\Q}(U)_{\relnef}^{\nef,\YZ}$ such that $\deg_{\widetilde{L}}(U)>0$, where $\widetilde{L}$ is the image of $\overline{L}$ in $\widetilde{\Pic}_\Q(U)_\integrable$. Let $(x_m)_{m\in I}$ be a generic net of points in $U(\overline{K})$ which is small with respect to $\overline{L}$. Then for any $v\in \Omega$, the Galois orbit of $(x_m)_{m\in I}$ is equidistributed in $U_v^\an$ with respect to the measure $c_1(\overline{L}_v)^d$ on $U_v^\an$, i.e. for any $f_v\in C_c(U_v^\an)$, we have
	\[\lim\limits_{m\in I}\frac{1}{|O(x_m)|}\sum\limits_{x_m^{\sigma}\in O(x_m)}f_v(x_m^{\sigma})= \frac{1}{\deg_{\widetilde{L}}(U)}\int_{U_v^\an}f_v\, c_1(\overline{L}_v)^d.\] 
\end{mainthm}

The key step of our proof is the technique of generic curves developed in \cite{gubler2008equidistribution}. Assume that $\dim(T)\geq 2$. Let $T_{\mathbf{c}}$ be the generic curve associated to $\mathbf{c}$ in \cref{generic curves}. In the language of \cite{chen2020arakelov} and \cite{chen2021arithmetic}, the technique is formulated as follows. 
\begin{mainthm}[] \label{introduction prop:B is covered by a curve as adelic curves}
	Let $S_{{\mathbf{c}}}$ be the adelic curve given by $T_{\mathbf{c}}$. Then we have a covering of adelic curves
	\[S_{{\mathbf{c}}}\to S.\]
\end{mainthm}
In fact, in our application, we need to consider the extension of the covering $S_{{\mathbf{c}}}\to S$; see \cref{prop:B is covered by a curve as adelic curves}.

Moreover, combining the technique of generic curves and the ideas in the proof of \cite[Theorem~1.1]{faber2009equidistribution}, we are able to prove a general \emph{equidistribution theorem for small subvarieties}; see \cref{thm:equidistribution for subvarieties over function fields}.

In the last section, we consider the equidistribution theorem for (\emph{arithmetically}) \emph{big} line bundles (see \cref{arithmetic volume}) on $U$. Under the assumption that $\Char(k)=0$, we establish a Fujita approximation theorem (\cref{thm:fujita approximation}) over the function field, which yields an intrinsic definition of the equidistribution measure via positive intersection. Consequently, we prove the following equidistribution theorem: 

\begin{mainthm}[\cref{theorem:equidsitributionforfinitenergy}]
	\label{theorem:equidsitributionforfinitenergy intro}
	Let $\overline{L}=(L,\metr)\in \widehat{\Pic}_{S,\Q}(U)^{\YZ}_{\cpt}$. Assume that $\overline{L}$ is big.  
	Let $(x_m)_{m\in I}$ be a generic net of points which is small with respect to $\overline{L}$. Then  $\widehat{\vol}(\overline{L}) = \widehat{\vol}(\overline{L}_{\mathbf c})$. Moreover, for any $v\in\Omega$ and any $f_v\in C_c(U_v^\an),$ 
	\begin{align*}
		\lim_{m\in I} \frac{\nu(v)}{\#O(x_m)}\sum_{x^\sigma_m\in O(x_m)}f_v(x^\sigma_m)=\frac{\langle \overline{L}^d\rangle\cdot \overline{\OO_U(f_v)}}{\vol(\widetilde{L})},
	\end{align*}
	where $\langle \overline{L}^d\rangle\cdot \overline{\OO_U(f_v)}$ is well-defined by \cref{measure given by positive intersection} and $\widetilde{L}$ is the image of $\overline L$ in $\widetilde{\Pic}_\Q(U)_\cpt$.
\end{mainthm}

We remark that when the characteristic is arbitrary, the measure can still be defined by push-forward and the same equidistribution result remains valid.

\subsection*{Organization of the paper}
In \cref{sec:adelic curves}, we recall adelic curves, coverings, and study the behavior of adelic curves arising from projective varieties under finite field extensions and base field extensions.
\cref{sec:compactified} develops compactified $S$-metrized $\YZ$-line bundles and arithmetic intersection numbers.
In \cref{sec:generic curves}, we introduce the generic curve technique and prove base change invariance for intersection numbers.
\cref{sec:equidistribution nef} establishes the equidistribution theorem for small points and subvarieties in the relatively nef case.
Finally, \cref{sec:equidistribution big} proves a Fujita approximation theorem and the equidistribution theorem for arithmetically big line bundles.

\subsection*{Notation}

For a topological space $X$, we denote by $C_c(X)$ the space of real continuous functions on $X$ with compact supports.

For a field $K$, we denote by $\overline{K}$ an algebraic closure of $K$ with an embedding $K\rightarrow \overline{K}$.

For any scheme $X$, we denote by $\Div(X)$ (resp. $\Pic(X)$) the group of Cartier divisors (resp. isomorphism classes of line bundles) on $X$, and by $\Div_\Q(X)\coloneq \Div(X)\otimes_\Z\Q$ (resp. $\Pic_\Q(X)\coloneq \Pic(X)\otimes_\Z\Q$) the space of $\Q$-Cartier divisors (resp. isomorphism classes of $\Q$-line bundles) on $X$. Given a Cartier divisor $D\in\Div(X)$, we denote by $\OO_X(D)$ the corresponding line bundle associated to $D$, and by $\Supp(D)$ the support of $D$. For $n\in\N$, we denote by $\CH_n(X)$ the Chow group of $n$-dimensional cycles on $X$, and for an $n$-dimensional cycle $\alpha$ on $X$, we denote by $\Supp(\alpha)$ its support. We denote by $X^{(n)}$ the set of all codimension $n$ points of $X$. We denote by $\mathrm{Irr}(X)$ the set of irreducible components of $X$.

An \emph{algebraic variety} $U$ over a field $k$ is defined as a geometrically integral separated scheme of finite type over $k$, and we denote by $k(U)$ its function field. A \emph{projective $k$-model} of an algebraic variety $U$ over $k$ is a projective variety $X$ over $k$ together with an open immersion $U\hookrightarrow X$. 

For a measure space $(\Omega, \mathcal{A}, \nu)$, if $\mathcal{A}$ is discrete, for simplicity, we write $\nu(\omega)\coloneq\nu(\{\omega\})$ for $\omega\in\Omega$. 

Let $S=(K,\Omega,\mathcal{A},\nu)$ be an adelic curve, see \cref{def:adelic curve}. We denote by $\mathscr{L}^1(\Omega,\mathcal{A},\nu)$ the space of integrable functions $\phi:\Omega\to\R$. For any $\omega\in \Omega$, we denote $\val_\omega$ the absolute value corresponding to $\omega$ and $K_\omega$ the completion of $K$ with respect to $\val_\omega$.

\subsection*{Acknowledgements}

The authors are grateful to Walter Gubler for helpful suggestions on the proof of \cref{lemma:lift closed subset}, and to Huayi Chen, Ruoyi Guo, Chunhui Liu and Xinyi Yuan for useful discussions.
\section{Adelic curves and coverings}
\label{sec:adelic curves}
\subsection{Adelic curves}

\begin{definition}[\cite{chen2020arakelov}~\S~3.1] 	\label{def:adelic curve}
	Let $K$ be a field and $M_K$ the set of all absolute values on $K$. An \emph{adelic structure on $K$} is a measure space $(\Omega,\mathcal{A},\nu)$ equipped with a map $\phi: \Omega\rightarrow M_K, \ \ \omega\mapsto |\cdot|_\omega$ such that for any $a\in K^\times$, the function $\log|a|_{\omega}: \Omega\rightarrow \R$ is $\mathcal{A}$-measurable, integrable with respect to $\nu$. The data $(K, (\Omega, \mathcal{A},\nu),\phi)$ (or simply $(K,\Omega, \mathcal{A}, \nu)$) is called an \emph{adelic curve}.
	Furthermore, an adelic curve $(K, (\Omega, \mathcal{A},\nu),\phi)$ is said to be \emph{proper} if the \emph{product formula} holds: for any $a\in K^\times$,
	\begin{equation*} 		\label{product formula}
		\int_\Omega \log|a|_\omega\, \nu(d\omega)=0.
	\end{equation*}
\end{definition}

In this paper, we consider the Arakelov theory over an adelic curve from a projective variety, the measure is given by the (algebraic) intersection number. For the convenience of the readers, we recall some properties of the (algebraic) intersection number.

\begin{art}\label{def:intersection number}
	Let $T$ be an integral $\delta$-dimensional scheme of finite type over a field $k$. For any $D_1,\dots, D_{n}\in \Div(T)$ and an $n$-dimensional cycle $\alpha$ on $T$, if $\dim(\Supp(D_1)\cap\cdots\cap \Supp(D_n) \cap \Supp(\alpha)) = 0$, we can define the \emph{intersection number} 
	$\deg_k(D_1\cdots D_n\cdot \alpha)$ (or $\deg(D_1\cdots D_n\cdot \alpha)$ if the base field is clear from the context), see \cite[Definition~2.4.2]{fulton1998intersection}.  When $T$ is projective over $k$, by \cite[Corollary~2.4.1]{fulton1998intersection}, moving lemma (e.g.\ \cite[Lemma~1.3.7]{chen2020arakelov}) or \cite[Proposition~2.2.3~(e)]{fulton1998intersection}, $\deg_k(D_1\cdots D_n\cdot \alpha)$ depends only on the line bundles $\OO_T(D_1),\dots, \OO_T(D_n)$ and the equivalence class $[\alpha]\in \CH_{n}(T)$, so $\deg_k(D_1\cdots D_n\cdot [\alpha])$ is well-defined for any $D_1,\dots, D_{n}\in \Div(T)$ and $[\alpha]\in  \CH_{n}(T)$. We have the following properties.
	\begin{enumerate1}
		\item\label{example:projection formula} (Projection formula, see \cite[Example~2.4.3]{fulton1998intersection}) Let $f\colon T'\to T$ be a finite, dominant morphism of projective schemes over $k$. Then
		\[\deg_k(f^*D_1\cdots f^*D_n\cdot [\alpha]) = \deg_k(D_1\cdots D_n\cdot f_*[\alpha]),\]
		where $f_*\colon \CH_n(T)\to \CH_n(T')$ is the push-forward map defined in \cite[Theorem~1.1.4]{fulton1998intersection}. 
		\item \label{intersection number base change} (Base change) Let $k_1/k$ be a field extension, and $T_{k_1}\coloneq T\times_{\Spec(k)}\Spec(k_1)$ with canonical map $\psi\colon T_{k_1}\to T$. Then we have the pull-back $\psi^*\colon \CH_n(T)\to \CH_n(T_{k_1})$, and 
		\[\deg_{k_1}(\psi^*D_1\cdots \psi^*D_n\cdot \psi^*[\alpha])= \deg_k(D_1\cdots D_n\cdot [\alpha]).\]
	\end{enumerate1}
\end{art}

\begin{example} \label{adelic curve defined by projective variety}
	Let $k$ be a field, and $T$ an integral $\delta$-dimensional {scheme} projective over $k$ such that $T$ is regular in codimension $1$. Let $K=k(T)$, and $D_1,\dots, D_{\delta-1}$ ample Cartier divisors on $T$. 
	Let $r\in\R_{>0}$. We can define an adelic curve $S_{T;D_1,\dots,D_{\delta-1};r}$ (write $S_{T;D_1,\dots,D_{\delta-1}}$ when $r=1$) as follows. Set $\Omega \coloneq T^{(1)}$, the set of all codimension $1$ points of $T$, equipped with the discrete $\sigma$-algebra $\mathcal{A}$. For any $\omega\in\Omega$, we denote by $\ord_\omega(\cdot)$ the corresponding valuation on $K$ and by $\val_\omega$ the absolute value on $K$ given by $\val_\omega\coloneq \exp(-\ord_\omega(\cdot))$. Let $\nu$ be the measure on $(\Omega,\mathcal{A})$ such that for any $\omega\in\Omega$,
	\[\nu(\omega)\coloneq r\cdot\deg(D_1\cdots D_{\delta-1}\cdot[\omega]),\]
	where $[\omega]$ is the image of $\overline{\{\omega\}}$ in $\CH_{\delta-1}(T)$.  
	
	Notice that if $\varphi\colon T'\to T$ is the normalization in $K$, then we have
	\[S_{T';\varphi^*D_1,\dots,\varphi^*D_{\delta-1};r}=S_{T;D_1,\dots,D_{\delta-1};r},\]
	so we will always assume that $T$ is normal.
\end{example}

\subsection{Coverings}

Following Chen--Moriwaki \cite[\S~2.2]{chen2021arithmetic}, a covering of adelic curves generalizes the canonical projection $\pi_{K'/K}\colon S_{K'}\to S$ for an algebraic field extension $K'/K$ (see \cref{def:extension of adelic curve} below).

\begin{definition}	\label{def:morphismofadeliccurves}
	Let $S=(K, \Omega, \mathcal{A},\nu)$ and $S'=(K', \Omega', \mathcal{A}',\nu')$ be two adelic curves. A \emph{morphism} $\alpha: S'\rightarrow S$ is a triple $(\alpha^\#,\alpha_\#,\gamma_\alpha)$, where 
	\begin{enumeratea}[resume,leftmargin=*,label=\it(\alph*),ref=\it{(\alph*)}]
		\item $\alpha^\#: K\rightarrow K'$ is a field homomorphism,
		\item $\alpha_\#: (\Omega', \mathcal{A}')\rightarrow (\Omega,\mathcal{A})$ is a measurable map,
		\item $\gamma_\alpha\in\mathscr{L}^1(\Omega',\mathcal{A}',\nu')_{> 0}$, i.e. an integrable function $\Omega'\to\R_{>0}$,
	\end{enumeratea}
	such that 
	\begin{enumerate}[leftmargin=*,label=\it(E),ref=\it{(E)}]
		\item \label{morphism:extension of value} for any $a\in K$ and $\omega'\in \Omega'$, one has that $|\alpha^\#(a)|_{\omega'}^{\gamma_\alpha(\omega')} = |a|_{\alpha_\#(\omega')}$.
	\end{enumerate}
	We say $\alpha$ is \emph{strong} if furthermore, for any $f\in \mathscr{L}^1(\Omega,\mathcal{A},\nu)$, we have 
	\begin{align}\label{eq:strong morphism condition}
\int_{\Omega'}f(\alpha_\#(\omega'))\gamma_\alpha(\omega')^{-1} \, \nu'(d\omega') = \int_\Omega f(\omega)\, \nu(d\omega),
	\end{align}
	in particular, $f\circ\alpha_\#\in   \mathscr{L}^1(\Omega',\mathcal{A}',\gamma_\alpha^{-1}\nu')$.
	
    A \emph{covering} of $S$ is a strong morphism $\alpha\colon S'\to S$ such that there is an $\R$-linear map
    \[I_\alpha\colon \mathscr{L}^1(\Omega',\mathcal{A}',\gamma_\alpha^{-1}\nu')\to \mathscr{L}^1(\Omega,\mathcal{A},\nu)\]
    verifying
    \begin{enumerate}[leftmargin=*,label=\it(C\arabic*),ref=\it{(C\arabic*)}]
    	\item\label{covering1} $\int_{\Omega}I_\alpha(g)(\omega)\, \nu(d\omega) = \int_{\Omega'}g(\omega')\gamma_\alpha(\omega')^{-1}
    	\, \nu'(d\omega')$ for any $g\in \mathscr{L}^1(\Omega',\mathcal{A}',\gamma_\alpha^{-1}\nu')$;
    	\item\label{covering2} $I_\alpha(f\circ \alpha_\#)=f$ for any $f\in \mathscr{L}^1(\Omega,\mathcal{A},\nu)$.
    \end{enumerate}
    Obviously,  we have a trivial covering $\id_S$ of $S$ given by the identity maps.
    
    Let $\alpha=(\alpha^\#,\alpha_\#,\gamma_\alpha)\colon S'\to S$ and $\beta=(\beta^\#,\beta_\#,\gamma_\beta)\colon S''=(K'',\Omega'',\mathcal{A}'',\nu'')\to S'$ be morphisms. Then the \emph{composition}
    \[\beta\circ\alpha\coloneq (\beta^\#\circ\alpha^\#, \alpha_\#\circ\beta_\#, \gamma_\beta\cdot (\gamma_\alpha\circ\beta_\#))\]
    is a morphism $S''\to S$. Moreover, if $\alpha$, $\beta$ are coverings with $\R$-linear maps  $I_\alpha$, $I_\beta$, respectively, then the \emph{composition} $\beta\circ\alpha$ is still a covering with $\R$-linear map 
    \begin{align*}
    	\mathscr{L}^1(\Omega'',\mathcal{A}'',\gamma_\beta^{-1}\cdot (\gamma_\alpha\circ\beta_\#)^{-1}\nu'')\to& \mathscr{L}^1(\Omega,\mathcal{A},\nu),\\
    	g\mapsto& I_\alpha(\gamma_\alpha \cdot I_\beta(g\cdot (\gamma_\alpha\circ \beta_\#)^{-1})).
    \end{align*}
    
    We say that $S$ and $S'$ are \emph{isomorphic}, denoted by $S\simeq S'$, if there are coverings $\alpha\colon S\to S'$ and $\alpha'\colon S'\to S$ such that $\alpha'\circ \alpha=\id_{S}$ and $\alpha\circ \alpha'=\id_{S'}$. 
\end{definition}
\begin{remark}
	Compared to \cite[Definition~2.1.2]{chen2021arithmetic}, although \cref{def:morphismofadeliccurves} for morphisms seems more flexible, these two definitions are essentially the same after we replace values $\val_{\omega'}$ by $\val_{\omega'}^{\gamma_\alpha(\omega')}$ and replace the measure $\nu'$ by $\gamma_\alpha^{-1}\nu'$. Hence, the results for coverings in \cite{chen2021arithmetic} still hold when we use \cref{def:morphismofadeliccurves} as the  definition of covering. 
\end{remark}
\begin{remark}\label{rmk:strong between discrete sigma}
	Let $\alpha=(\alpha^\#, \alpha_\#,\gamma_\alpha)\colon S' = (K', \Omega', \mathcal{A}', \nu') \to S=(K, \Omega, \mathcal{A}, \nu)$ be a morphism of adelic curves. Assume that $\mathcal{A}, \mathcal{A}'$ are discrete. Then $\alpha$ is strong if and only if for any $\omega\in \Omega$, 
	\begin{align}\label{eq:strong for discrete measurable space}
\nu(\omega) = \sum\limits_{x\in\alpha_\#^{-1}(\omega)}\gamma_\alpha(x)^{-1}\nu'(x).
	\end{align}
Indeed, if $\alpha$ is strong, for any $\omega\in\Omega$, we take $f$ in \eqref{eq:strong morphism condition} to be the characteristic function $\mathbbm{1}_\omega$ of $\{\omega\}$ which is integrable, then \eqref{eq:strong morphism condition} becomes \eqref{eq:strong for discrete measurable space}. Conversely, for any $f\in\mathscr{L}^1(\Omega,\mathcal{A},\nu)_{\geq 0}$, we have
\begin{align*}
\int_{\Omega'}f(\alpha_\#(\omega'))\gamma_\alpha(\omega')^{-1}\, \nu'(d\omega') =& \sum\limits_{\omega'\in\Omega'}f(\alpha_\#(\omega'))\gamma_\alpha(\omega')^{-1}\nu'(\omega')\\
=&\sum\limits_{\omega\in\Omega}f(\omega)\sum\limits_{\omega'\in\alpha^{-1}_\#(\omega)}\gamma_\alpha(\omega')^{-1}\nu'(\omega')\\
=&\sum\limits_{\omega\in\Omega}f(\omega)\nu(\omega)\\
=&\int_\Omega f(\omega)\, \nu(d\omega).
\end{align*}
This implies that $\alpha$ is strong since each $f\in\mathscr{L}^1(\Omega, \mathcal{A}, \nu)$ can be written as $f=\max\{0,f\}-\max\{0,-f\}$.
\end{remark}
\begin{remark}
	In the definition of isomorphism of adelic curves, although we don't have any requirement on $I_{\alpha}, I_{\alpha'}$, we will see from the proof of \cref{lemma:characterization of covering and isomorphism}~\ref{isomorphism for discrete measurable} that we can choose suitable $I_\alpha, I_{\alpha'}$ which are inverse to each other.
\end{remark}

\begin{example} \label{example:morphism given by morphism of schemes}
	Keeping the notation in \cref{adelic curve defined by projective variety}. 
	Let $T'$ be another integral $\delta$-dimensional {scheme} projective over $k$ such that $T'$ is regular in codimension $1$ with function field $K'\coloneq k(T')$. Let $D'_1,\dots, D'_{\delta-1}$ be ample Cartier divisors on $T'$, $r'\in\R_{>0}$, and $S_{T';D'_1,\dots, D'_{\delta-1};r'}=(K',\Omega',\mathcal{A},\nu')$ the corresponding adelic curve. Let $\varphi\colon T'\to T$  be a finite, dominant morphism over $k$. Then $\varphi$ naturally induces a morphism of adelic curves
	\[\beta_\varphi=(\beta_\varphi^\#,\beta_{\varphi,\#},\gamma_\varphi)\colon S_{T';D'_1,\dots, D'_{\delta-1};r'}\to S_{T;D_1,\dots, D_{\delta-1};r}\]
	defined as follows: $\beta_\varphi^\#\colon K\to K'$ is the morphism of fields defined by $\varphi$, $\beta_{\varphi,\#}\colon \Omega'\to\Omega$ is the restriction of $\varphi$ to $\Omega'$ (notice that $\varphi(\Omega')\subset \Omega$ since $\varphi$ is finite and dominant) and $\gamma_\varphi(\omega')$ is the inverse of the ramification index of $\varphi$ at $\omega'\in\Omega'$.
	Notice that $\beta_\varphi$ is functorial in $\varphi$. 
	
	Obviously, $\beta_\varphi$ is not necessarily strong. We will see from \cref{prop:adelic curve of field extension}~\ref{adelic curve:finite field extension} that if $r'=r\cdot[K':K]^{-1}$, and $D'_j=\varphi^*D_j$ for $j=1,\dots, \delta-1$, then $\beta_\varphi$ is a covering (with an $\R$-linear map $I_{\varphi}\colon \mathscr{L}^1(\Omega',\mathcal{A}',\gamma_\varphi^{-1}\nu')\to \mathscr{L}^1(\Omega,\mathcal{A},\nu)$).
\end{example}

\begin{lemma}\label{lemma:characterization of covering and isomorphism}
	Let $\alpha=(\alpha^\#, \alpha_\#,\gamma_\alpha)\colon S' = (K', \Omega', \mathcal{A}', \nu') \to S=(K, \Omega, \mathcal{A}, \nu)$ be a morphism of adelic curves.
	\begin{enumerate1}
		\item\label{strong for discrete measurable} Assume that $\mathcal{A}, \mathcal{A}'$ are discrete and  $\alpha_\#^{-1}(\omega)$ is finite for any $\omega\in \Omega$. Then $\alpha$ is a covering if and only if $\alpha$ is strong.  
		\item \label{isomorphism for discrete measurable} The morphism $\alpha$ is an isomorphism if and only if $\alpha$ is strong, $\alpha^\#$ is an isomorphism of fields, and $\alpha_\#$ is an isomorphism of measurable spaces (i.e. a bijection between $\Omega$ and $\Omega'$ which induces a bijection between $\mathcal{A}$ and $\mathcal{A}'$).
	\end{enumerate1}
\end{lemma}
\begin{proof}
	\ref{strong for discrete measurable} By definition, if $\alpha$ is a covering, then $\alpha$ is strong. Conversely, if $\alpha$ is strong, we define $I_\alpha\colon \mathscr{L}^1(\Omega', \mathcal{A}', \gamma_\alpha^{-1}\nu')\to \mathscr{L}^1(\Omega, \mathcal{A},\nu)$ as follows: 
	for any $g\in \mathscr{L}^1(\Omega', \mathcal{A}',\gamma_\alpha^{-1}\nu')$, we set
	\[I_\alpha(g)(\omega)\coloneq \frac{1}{\nu(\omega)}\sum\limits_{x\in \alpha^{-1}_\#(\omega)}\gamma_\alpha(x)^{-1}g(x)\nu'(x).\]
	It remains to check \cref{def:morphismofadeliccurves}~\ref{covering1} and \ref{covering2}. For any $g\in \mathscr{L}^1(\Omega',\mathcal{A}',\gamma_\alpha^{-1}\nu')_{\geq0}$, we have
	\begin{align*}
		\int_{\Omega}I_\alpha(g)(\omega)\, \nu(d\omega) =& \sum\limits_{\omega\in\Omega}I_\alpha(g)(\omega)\cdot\nu(\omega)\\
		=& \sum\limits_{\omega\in\Omega}\sum\limits_{x\in\alpha_\#^{-1}(\omega)}\gamma_\alpha(x)^{-1}g(x)\cdot\nu'(x)\\
	=& \sum\limits_{x\in\Omega'}\gamma_\alpha(x)^{-1}g(x)\cdot\nu'(x)\\
=& \int_{\Omega'}g(x)\gamma_\alpha(x)^{-1}\, \nu'(dx).
	\end{align*} 
	Since $I_\alpha$ is linear and each $g\in\mathscr{L}^1(\Omega', \mathcal{A}', \gamma_\alpha^{-1}\nu')$ can be written as $g=\max\{0,g\}-\max\{0,-g\}$, we know that \ref{covering1} holds. For \ref{covering2}, let $f\in\mathscr{L}^1(\Omega,\mathcal{A},\nu)$, then
	\begin{align*}
		I_\alpha(f\circ\alpha_\#)(\omega) =& \frac{1}{\nu(\omega)}\sum\limits_{x\in\alpha_\#^{-1}(\omega)}\gamma_\alpha(x)^{-1}(f\circ\alpha_\#)(x)\nu'(x)\\
		=&\frac{1}{\nu(\omega)}\left(\sum\limits_{x\in\alpha_\#^{-1}(\omega)}\gamma_\alpha(x)^{-1}\nu'(x)\right)f(\omega)\\
		=&f(\omega),
	\end{align*}
	where the last equality follows from the fact that $\alpha$ is strong and \cref{rmk:strong between discrete sigma}.
	
	\ref{isomorphism for discrete measurable} The "only if" part is obvious by \cref{def:morphismofadeliccurves}. Conversely, since $\alpha_\#$ is bijective, for any $g\colon \Omega\to \R$, we set $I_\alpha(g)\coloneq g\circ\alpha_\#^{-1}$. Notice that $I_\alpha$ induces a bijection between measurable functions. By the assumption that $\alpha$ is strong, it is easy to see that $\alpha$ is a covering with $\R$-linear map $I_\alpha$. The inverse of $\alpha$ is given by $\left({\alpha^{\#}}^{-1}, \alpha_\#^{-1}, \frac{1}{\gamma_\alpha\circ\alpha_\#^{-1}}\right)$ with $\R$-linear map $I_\alpha^{-1}$.
\end{proof}

\begin{art} \label{def:extension of adelic curve}
	Let $S=(K,\Omega,\mathcal{A},\nu)$ be an adelic curve, and $K'/K$ an algebraic field extension. Chen-Moriwaki \cite[\S 3.4.1]{chen2020arakelov} defined an adelic curve $S_{K'}=(K',\Omega_{K'}, \mathcal{A}_{K'}, \nu_{K'})$, called the \emph{extension of $S$ over $K'/K$}. Moreover, we have a covering $(\pi_{K'/K}, I_{K'/K})\colon S_{K'}\to S_K$ in the sense of \cref{def:morphismofadeliccurves} which is compatible with algebraic extensions, see \cite[\S 2.2]{chen2021arithmetic}. We recall the construction.
	\begin{enumerate1}
		\item  When $K'/K$ is a finite field extension, $S_{K'}=(K',\Omega_{K'}, \mathcal{A}_{K'}, \nu_{K'})$ and $(\pi_{K'/K}, I_{K'/K})$ are given as follows.
		\begin{enumerate}[leftmargin=*,label=\it(1.\alph*),ref=\it{(1.\alph*)}]
			\item For each $\omega\in\Omega$, let $M_{K',\omega}$ be the set of all absolute values on $K'$ which extend $\val_\omega$ on $K$ (notice that $M_{K',\omega}$ is finite), we define $\Omega_{K'}$ as the disjoint union  $$\coprod\limits_{\omega\in \Omega}M_{K',\omega}.$$
			\item We set $\pi_{K'/K}^\#\colon K\hookrightarrow K'$ the natural embedding, and $\pi_{K'/K,\#}\colon \Omega_{K'}\to \Omega$ the natural projection which sends elements of $M_{K',\omega}$ to $\omega$ for any $\omega\in\Omega$. Set $\gamma_{\pi_{K'/K}}\equiv 1$ on $\Omega_{K'}$.
			\item The $\sigma$-algebra $\mathcal{A}_{K'}$ on $\Omega_{K'}$ is the smallest $\sigma$-algebra such that  $\pi_{K'/K,\#}^{-1}(A)\in\mathcal{A}_{K'}$ for all $A\in \mathcal{A}$ and the functions $(x\in\Omega_{K'})\mapsto|a|_x$ are $\mathcal{A}_{K'}$-measurable for all $a\in K'$.
			\item \label{extension of adelic curve 1d} For any function $f\colon\Omega_{K'}\to\R$ and $\omega\in\Omega$, we set \[I_{K'/K}(f)(\omega)\coloneq\sum\limits_{x\in M_{K',\omega}}\frac{[K_x': K_\omega]_s}{[K': K]_s}f(x);\]
			where $[K':K]_s$ (resp. $[K'_x: K_\omega]_s$) denotes the separable degree of the extension $K'/K$ (reps. $K'_x/K_\omega$).
			\item \label{extension of adelic curve 1e} For any $A'\in \mathcal{A}_{K'}$, we set
			\[\nu_{K'}(A')\coloneq \int_{\Omega}I_{K'/K}(\mathbbm{1}_{A'})\, \nu,\]
			where $\mathbbm{1}_{A'}$ is the characteristic function of $A'$ of $\Omega_{K'}$.
		\end{enumerate}
		\item In general, we define $S_{K'}=(K',\Omega_{K'}, \mathcal{A}_{K'}, \nu_{K'})$ and $(\pi_{K'/K}, I_{K'/K})$ in the following steps. Denote by $\mathscr{E}_{K'/K}$ the set of all finite extensions of $K$ which is contained in $K'$.
		\begin{enumerate}[leftmargin=*,label=\it(2.\alph*),ref=\it{(2.\alph*)}]
			\item \label{general extension omega} We set
			\[(\Omega_{K'},\mathcal{A}_{K'})\coloneq \varprojlim_{L\in\mathscr{E}_{K'/K}}(\Omega_{L},\mathcal{A}_{L})\]
			the projective limit in the category of measurable spaces (in our case, the projective limit exists, see \cite[\S 3.4.2]{chen2020arakelov}).
			For any $L\in\mathscr{E}_{K'/K}$, denote by $\pi_{K'/L}^\#\colon L\to K'$ the natural embedding, and by $\pi_{K'/L,\#}\colon \Omega_{K'}\to \Omega_L$ the corresponding map induced by restriction. More precisely, as the case of finite field extensions, for each $\omega\in\Omega$, we set $M_{K',\omega}$ the set of all absolute values on $K'$ which extend $\val_\omega$ on $K$, then $\Omega_{K'}=\coprod\limits_{\omega\in \Omega}M_{K',\omega}$ (the maps $\pi_{K'/L,\#}$ can be defined naturally in this way too), and $\mathcal{A}_{K'}$ is the smallest $\sigma$-algebra which makes all natural maps $\pi_{K'/L,\#}$ measurable.  We set $\gamma_{\pi_{K'/K}}\equiv 1$ on $\Omega_{K'}$.
			\item \label{extension of adelic curve 2b} For any $\omega\in\Omega$, 
			We endow $M_{K',\omega}=\varprojlim_{L\in\mathscr{E}_{K'/K}}M_{L,\omega}$ with the projective limit topology (where $M_{L,\omega}$ is equipped with discrete topology), and set $\mathcal{A}_{K',\omega}$ the Borel algebra on $M_{K',\omega}$. Then there is a unique Borel probability measure $\mathbbm{P}_{K',\omega}$ on $(M_{K',\omega}, \mathcal{A}_{K',\omega})$ such that for any $L\in\mathscr{E}_{K'/K}$ and $A\in\mathcal{A}_{L}$, we have
			\[\mathbbm{P}_{K',\omega}(\pi_{K'/L,\#}^{-1}(A))=\sum\limits_{x\in A}\frac{[L_x:K_\omega]_s}{[L:K]_s},\]
			see \cite[Proposition~3.4.6]{chen2020arakelov}.
			\item\label{extension of adelic curve 2c} For any $A\in\mathcal{A}_{K'}$, the function $(\omega\in\Omega)\mapsto \int_{M_{K',\omega}}\mathbbm{1}_A(x)\, \mathbbm{P}_{K',\omega}(dx)$ is $\mathcal{A}$-measurable (see \cite[Proposition~3.4.8]{chen2020arakelov}), we set
			\[\nu_{K'}(A)\coloneq \int_{\Omega}\left(\int_{M_{K',\omega}}\mathbbm{1}_A(x)\, \mathbbm{P}_{K',\omega}(dx)\right)\, \nu(d\omega).\]
			\item For any $f\in\mathscr{L}^1(\Omega_{K'},\mathcal{A}_{K'},\nu_{K'})$, 
			and $\omega\in\Omega$, we set
			\[I_{K'/K}(f)(\omega)\coloneq \int_{M_{K',\omega}}f(x)\, \mathbbm{P}_{K',\omega}(dx)\]
			which is in $\mathscr{L}^1(\Omega,\mathcal{A}, \nu)$, see \cite[Proposition~3.4.9]{chen2020arakelov}.
		\end{enumerate}
	\end{enumerate1} 
\end{art}
\begin{remark}\label{rmk:morphism of extensions}
	Let $S_i=(K_i,\Omega_i,\mathcal{A}_i,\nu_i), i=1,2,$ be two adelic curves, and let $\alpha=(\alpha^\#,\alpha_\#,\gamma_\alpha)\colon S_2\to S_1$ be a morphism of adelic curves. For an algebraic field extension $K_1'/K_1$, let $K_2'$ be the composite of $K_1'$ and $K_2$ in $\overline{K_2}$. Then we have a morphism $\alpha'\colon S_{2,K_2'}\to S_{1,K_1'}$ induced by $\alpha$ such that 
	\[\xymatrix{S_{2,K_2'}\ar[d]_{\pi_{K'_2/K_2}}\ar[r]^{\alpha'}& S_{1,K_1'}\ar[d]^{\pi_{K'_1/K_1}}\\
	S_{2}\ar[r]^{\alpha}& S_1}\]
	commutes. 
	Indeed, it suffices to consider finite extension case. For simplicity, we write $S_{i,K_i'}'=(K_i',\Omega_i',\mathcal{A}_i',\nu_i'), i=1,2.$ We denote by ${\alpha'}^\#$ the natural homomorphism $K_1'\to K_2'$. Set $\gamma_{\alpha'}\coloneq\gamma_\alpha\circ\pi_{K_2'/K_2,\#}$. For any $x_2\in\Omega_2'$, $\val_{x_2}^{\gamma_{\alpha'}(x_2)}$ is an absolute value extending $\val_{\alpha_\#(\pi_{K_2'/K_2,\#}(x_2))}$. After restricting $\val_{x_2}^{\gamma_{\alpha'}(x_2)}$ to $K_1'$, we get an absolute value $\val_{x_1}$. Hence we have a map $\alpha_\#'\colon\Omega_2'\to \Omega_1'$ satisfying \cref{def:morphismofadeliccurves}~\ref{morphism:extension of value} (with respect to ${\alpha'}^\#$). It remains to show that $\alpha_\#$ is a measurable map. Notice that $\mathcal{A}_1'$ is generated by $\pi_{K_1'/K_1,\#}^{-1}(A), A\in\mathcal{A}_1$, and $|a|_x$, $a\in K_1'$; the same holds for $\mathcal{A}_2'$. For  $A\in\mathcal{A}_1$, we have ${\alpha_\#'}^{-1}(\pi_{K_1'/K_1,\#}^{-1}(A)) = \pi_{K_2'/K_2,\#}^{-1}(\alpha_\#^{-1}(A))\in \mathcal{A}_2'$. On the other hand, for any $a\in K_1'$, set $f(x_1)\coloneq |a|_{x_1}$, $x_1\in \Omega_1'$, and $g(x_2)\coloneq |\alpha^\#(a)|_{x_2}^{\gamma_{\alpha'}(x_2)}, x_2\in\Omega_2'$, then $g=f\circ\alpha_\#'$, and for any Borel subset $B$ of $\R$, ${\alpha'_\#}^{-1}(f^{-1}(B)) = g^{-1}(B)\in \mathcal{A}_2$. From our discussion, we know that $\alpha'_\#$ is a measurable map. Hence $\alpha'=({\alpha'}^\#, {\alpha'}_\#, \gamma_{\alpha'})$ is a morphism. 
	
	If $\alpha$ is an isomorphism in the sense of \cref{def:morphismofadeliccurves}, then $\alpha'$ is an isomorphism. However, it is unclear that $\alpha'$ is a covering when $\alpha$ is a covering (for some $I_\alpha$). 
\end{remark}




\begin{proposition}\label{prop:adelic curve of field extension}
	Let $T$ be a $\delta$-dimensional normal integral {scheme} projective over a field $k$, $D_1,\dots, D_{\delta-1}$ ample Cartier divisors on $T$, and $r\in\R_{>0}$. Let $S\coloneq S_{T;D_1,\dots, D_{\delta-1};r}=(K,\Omega,\mathcal{A},\nu)$ be the corresponding adelic curve given \cref{adelic curve defined by projective variety}.
	\begin{enumerate1}
		\item \label{adelic curve:finite field extension} Let $K'/K$ be a finite extension of fields, and $T'$ the normalization of $T$ in $K'$. Write $\varphi\colon T'\to T$ for the canonical morphism. 
Then we have a canonical isomorphism (canonical for $K'/K$)
\begin{align}\label{eq:adelic curve of field extension1}
	S_{T';\varphi^*D_1,\dots,\varphi^*D_{\delta-1}; r[K':K]^{-1}} \simeq S_{K'}
\end{align}
(here $S_{K'}$ is the extension of S over K'/K defined in \cref{def:extension of adelic curve}).
		\item \label{adelic curve:base field extension} Let $k_1/k$ be a {field extension which is Galois or purely inseparable}, and $T_1$ the normalization of some irreducible component of $T_{k_1}\coloneq T\times_{\Spec(k)}\Spec(k_1)$. Let $K_1$ be the function field of $T_1$ and $\psi\colon T_1\to T$ the canonical morphism. Then $K_1$ is the composite field of $K$ and $k_1$ in $\overline{K}$ (after choosing suitable embeddings), and we have a canonical isomorphism (canonical for $k_1/k$)
		\begin{align}\label{eq:adelic curve of field extension2}
S_{T_1; \psi^*D_1,\dots, \psi^*D_{\delta-1}; rm}\simeq S_{K_1},
		\end{align}
where we view $T_1$ as a variety over $k_1$, and
\begin{align*}
m=\mathrm{length}(\OO_{T_{k_1},\eta_1})\cdot\#\mathrm{Irr}(T_{k_1})=\begin{cases}
\#\mathrm{Irr}(T_{k_1})& \text{if $k_1/k$ is Galois,}\\
\mathrm{length}(\OO_{T_{k_1},\eta_1})& \text{if $k_1/k$ is purely inseparable},
\end{cases}
\end{align*} 
with $\OO_{T_{k_1},\eta_1}$ the local ring of $T_{k_1}$ at the generic point $\eta_1$ of $T_1$ and $\#\mathrm{Irr}(T_{k_1})$ the number of irreducible components of $T_{k_1}$. In particular, we have a covering $$S_{T_1; \psi^*D_1,\dots, \psi^*D_{\delta-1}; rm}\to S.$$
	\end{enumerate1}
\end{proposition}
\begin{proof}
We prove \cref{prop:adelic curve of field extension} in the following steps.

\vspace{2mm} \noindent
\textbf{Step~1.} {\it Keeping the notation in \ref{adelic curve:finite field extension}. There is a canonical morphism $\alpha=(\alpha^\#,\alpha_\#,\gamma_\alpha)\colon S_{K'} \to S_{T';\varphi^*D_1,\dots,\varphi^*D_{\delta-1}; r[K':K]^{-1}}$ such that $\alpha^\#$ is the identity of $K'$, and $\alpha_\#$ is bijective. Here the canonicity means that, for any finite extension $K''/K'$, let $T''$ the normalization of $T'$ in $K''$ with the canonical morphism $\varphi'\colon T''\to T'$, then we have the following commutative diagram:
	\begin{equation}\label{naturality of morphism of adelic curves}
	\begin{aligned}
	\xymatrix{S_{T''; \psi^*\varphi^*D_1,\dots,\psi^*\varphi^*D_{\delta-1}; r[K'':K]^{-1}}\ar[d]_-{\beta_{\varphi'}}\ar[r]^-{\alpha'}& S_{K''}\ar[d]^-{\pi_{K''/K'}}\\
	 S_{T';\varphi^*D_1,\dots,\varphi^*D_{\delta-1}; r[K':K]^{-1}}	\ar[r]^-\alpha&S_{K'},}
\end{aligned}
	\end{equation}
where $\pi_{K''/K'}$ is the covering of extension given in \cref{def:extension of adelic curve}, $\alpha'$ is the morphism for $K''$ constructed similarly as $\alpha$, and $\beta_{\varphi'}$ is given in \cref{example:morphism given by morphism of schemes}.} 

\vspace{2mm} \noindent
We write $S_1=S_{K'}=(K',\Omega_1,\mathcal{A}_1,\nu_1)$ and $S_2=S_{T';\varphi^*D_1,\dots,\varphi^*D_{\delta-1}; r[K':K]^{-1}} = (K', \Omega_2, \mathcal{A}_2, \nu_2)$. Set $\alpha^\#$ to be the identity map on $K'$, and for any $\omega'\in\Omega_2$, set $\gamma_\alpha(\omega')$ 
to be the inverse of the ramification index of $\varphi$ at $\omega'$ (notice that $T'\to T$ is finite). To define $\alpha_\#\colon \Omega_2\to\Omega_1$, notice that for any $\omega'\in\Omega_2$ and any $f\in K$, we have 
\[|f|_{\omega'}=\exp\{-\ord_{\omega'}(f)\} = \exp\{-\gamma_{\alpha}(\omega')^{-1}\ord_{\varphi(\omega')}(f)\}= |f|_{\varphi(\omega')}^{\gamma_{\alpha}(\omega')^{-1}},\] 
so $\val_{\omega'}^{\gamma_{\alpha}(\omega')}$ is an extension of $\val_{\varphi(\omega')}$, then we set $\alpha_\#(\omega')$ to be the corresponding absolute value in $\Omega_1$.  Moreover, $\alpha_\#$ is a measurable map since $\mathcal{A}_2$, $\mathcal{A}_1$ are discrete. Hence $\alpha\coloneq (\alpha^\#,\alpha_\#, \gamma_\alpha)$ is a morphism from $S_2$ to $S_1$. It is easy to see that the construction of $\alpha$ is canonical, i.e. \eqref{naturality of morphism of adelic curves} holds for any finite extension $K''/K'$.

We claim that $\alpha_\#$ is bijective. For this, we construct an inverse of $\alpha_\#$.
Let $x\in \Omega_1$, and $\OO_x$ the ring of integers of $K'$ with respect to $\val_x$. Since $T'$ is projective over $k$, by the valuative criterion for properness, there is a morphism $\Spec(\OO_x)\to T'$. We denote by $\widetilde{\alpha}_\#(x)$ the image of maximal ideal of $\OO_x$ which is of codimension $1$ on $T'$, then $\OO_{T',\widetilde{\alpha}_\#(x)}\subset \OO_x$. Hence $\OO_{T',\widetilde{\alpha}_\#(x)}= \OO_x$ since $\OO_{T',\omega'}, \OO_x$ are both discrete valuation rings with the same fraction field $K'$. Since $x\in\Omega_1$ is arbitrary, in this way, we define a map $\widetilde{\alpha}_\#\colon \Omega_1\to\Omega_2$. It is easy to see that $\alpha_\#$ and $\widetilde{\alpha}_\#$ are inverse to each other. This proved Step~1.

\vspace{2mm} \noindent

By \cref{lemma:characterization of covering and isomorphism}~\ref{isomorphism for discrete measurable}, to show \ref{adelic curve:finite field extension}, it remains to show that $\alpha$ constructed above is strong. To do so, we need the following result.

\vspace{2mm} \noindent
\textbf{Step~2.} {\it Keeping the notation in Step~1. When $K'/K$ is separable, the morphism $\alpha$ is strong.}

\vspace{2mm} \noindent
We keep the notation in the proof of Step~1. For any $\omega'\in\Omega_2$, by \cite[Proposition~2.6.8]{neukirch1999algebraic}, we have 
\begin{align}\label{eq:fundamental equality}
[K'_{\omega'} : K_{\varphi(\omega')}] = [\widetilde{K'_{\omega'}} : \widetilde{K_{\varphi(\omega')}}]\cdot \gamma_\alpha(\omega')^{-1},
\end{align}
where $\widetilde{K'_{\omega'}}$ (resp. $\widetilde{K_{\varphi(\omega')}}$) is the residue field of $\omega'$ (resp. $\varphi(\omega')$), and $K'_{\omega'}$ (resp. $K_{\varphi(\omega')}$) is the completion of $K'$ (resp. $K$) with respect to $\val_{\omega'}$ (resp. $\val_{\varphi(\omega')}$). By projection formula (see \cref{def:intersection number}~\ref{example:projection formula}) and $\varphi_*[\omega'] = [\widetilde{K'_{\omega'}} : \widetilde{K_{\varphi(\omega')}}] [\varphi(\omega')]$, we have 
\begin{align*}
	\nu_2(\omega')=&\frac{r}{[K':K]}\deg_k(\varphi^*D_1\cdots \varphi^*D_{\delta-1}\cdot[\omega']) \\
	=&\frac{[\widetilde{K'_{\omega'}} : \widetilde{K_{\varphi(\omega')}}]}{[K':K]}\cdot r\deg_k(D_1\cdots D_{\delta-1}\cdot[\varphi(\omega')])\\
	=&\frac{[K'_{\omega'} : K_{\varphi(\omega')}]}{[K':K]}\cdot\gamma_\alpha(\omega')\nu(\varphi(\omega')) \\
	=&\gamma_\alpha(\omega')\nu_1(\alpha_\#(\omega')),
\end{align*}
where the third equality follows from \eqref{eq:fundamental equality} and the last equality follows from the definition of extension of $S$ over $K'/K$ in \cref{def:extension of adelic curve}~\ref{extension of adelic curve 1e}. This immediately implies that $\alpha$ is strong by \cref{rmk:strong between discrete sigma}, so Step~2 holds.

\vspace{2mm} \noindent
\textbf{Step~3.} {\it Keeping the notation in Step~1. When $K'/K$ is purely inseparable, the morphism $\alpha$ is strong.}

\vspace{2mm} \noindent
We keep the notation in the proof of Step~1. When $K'/K$ is purely inseparable over $K$ with $p^n\coloneq [K':K]$ a power of a prime $p$, 
as before, by \cite[Proposition~2.6.8]{neukirch1999algebraic} and projection formula, for any $\omega'\in\Omega_2$, we have 
\begin{align}\label{eq:strong for purely inseparable}
	\nu_2(\omega') = \frac{[\widetilde{K'_{\omega'}} : \widetilde{K_{\varphi(\omega')}}]}{[K':K]}\cdot\nu(\varphi(\omega')) = \frac{[K'_{\omega'}: K_{\varphi(\omega')}]}{p^n}\cdot\gamma_\alpha(\omega')\nu_1(\alpha_\#(\omega')).
\end{align}
We claim that for any $\omega'\in \Omega_2=T'^{(1)}$, we have $[K'_{\omega'}: K_{\varphi(\omega')}] = p^n$, this implies that $\alpha$ is strong, so Step~3 holds. 
To prove this, we set
\[E\coloneq \{\omega'\in T'^{(1)}\mid [K_{\omega'}':K_{\varphi(\omega')}]\not=p^n\}.\]
For any $\omega'\in E$, we have $[K'_{\omega'}: K_{\varphi(\omega')}] < p^n$. 
Notice that $\varphi\colon T'\to T$ is a bijection since $K'/K$ is purely inseparable (see \cite[Exercise~5.3.9]{liu2006algebraic}). Let $U$ be the flat locus of $\varphi\colon T'\to T$ which is non-empty and open in $T'$. For any $\omega'\in U^{(1)}\subset\Omega_2$, $\OO_{T',\omega'}$ is free of rank $[K':K]$ as an $\OO_{T,\varphi(\omega')}$-module. By \cite[Lemma~7.1.36~(c)]{liu2006algebraic} and \cite[Remark of Proposition~2.6.8]{neukirch1999algebraic}, we have
\[p^n=[K':K] = [\widetilde{K_{\omega'}'}: \widetilde{K_{\varphi(\omega')}}]\gamma_\alpha(\omega')^{-1} = [K_{\omega'}':K_{\varphi(\omega')}].\] 
Then $E\subset T'^{(1)}\setminus U^{(1)}$, so $E$ is finite. If $E\not=\emptyset$, by the approximation theorem \cite[(2.3.4)~Approximation Theorem]{neukirch1999algebraic}, there is an $f\in K'$ such that for any $\omega'\in E$, 
\begin{align}\label{eq:extension of adelic curve geometric0}
	|f|_{\omega'}<1.
\end{align}
We may assume that $f\not=0$. Indeed, we fix $\omega_0\in E$ and $0\not=f_0\in K'$ such that $|f_0|_{\omega_0}=\varepsilon_0<1$. For any $\omega'\in E$, set $a_{\omega'} = f_0$ if $\omega'=\omega_0$, otherwise set $a_{\omega'} = 0$. Then approximation theorem implies that there is $f\in K'$ such that for any $\omega'\in E$, 
\[|f-a_{\omega'}|_{\omega'}<\varepsilon_0.\]
Obviously, such $f$ satisfies \eqref{eq:extension of adelic curve geometric0} and $f\not=0$.
By the product formula for $S_1$ (see \cite[Proposition~3.4.10]{chen2020arakelov}), we have
\begin{align*}
	0= &\sum\limits_{\omega'\in \Omega_2}(\log|f|_{\alpha_\#(\omega')})\cdot\nu_1(\alpha_\#(\omega'))\\
	=& \sum\limits_{\omega'\in \Omega_2}(\log|f|_{\omega'}^{\gamma_\alpha(\omega')})\cdot \frac{p^n}{[K'_{\omega'}: K_{\varphi(\omega')}]\cdot\gamma_\alpha(\omega')}\nu_2(\omega')\\
	=&p^n\cdot\sum\limits_{\omega'\in \Omega_2}\frac{\log|f|_{\omega'}}{[K'_{\omega'}: K_{\varphi(\omega')}]}\cdot\nu_2(\omega')
\end{align*}
Hence
\begin{align}\label{eq:extension of adelic curve geometric1}
	\sum\limits_{\omega'\in E}\frac{\log|f|_{\omega'}}{[K'_{\omega'}: K_{\varphi(\omega')}]}\cdot \nu_2(\omega')=-\frac{1}{p^n}\sum\limits_{\omega'\in \Omega_2\setminus E}\log|f|_{\omega'}\cdot \nu_2(\omega').
\end{align}
On the other hand, by the product formula for $S_2$, we have
\begin{align}\label{eq:extension of adelic curve geometric2}
	\sum\limits_{\omega'\in E}\log|f|_{\omega'}\cdot \nu_2(\omega')=-\sum\limits_{\omega'\in \Omega_2\setminus E}\log|f|_{\omega'}\cdot \nu_2(\omega').
\end{align}
Combining \eqref{eq:extension of adelic curve geometric1} and \eqref{eq:extension of adelic curve geometric2}, we have
\begin{align*}
	\sum\limits_{\omega'\in E}\left(\frac{1}{[K'_{\omega'}: K_{\varphi(\omega')}]}-\frac{1}{p^n}\right)\log|f|_{\omega'}\cdot \nu_2(\omega')=0
\end{align*}
which contradicts our choice of $f$ by \eqref{eq:extension of adelic curve geometric0}. Hence $E=\emptyset$ which shows that our claim holds, so Step~3 holds.

\vspace{2mm} \noindent
\textbf{Step~4.} {\it \ref{adelic curve:finite field extension} holds.}

\vspace{2mm} \noindent
We keep the notation in the proof of Step~1.
Let $L$ be the separable closure of $K$ in the field $K'$. Then $K'$ is a purely inseparable extension of $L$. By Step~2, we know that $S_L$ is of the form given in \cref{adelic curve defined by projective variety}, i.e. the corresponding $\alpha$ for $L$ is strong. Notice that the morphism $\alpha\colon S_{K'} \to S_{T';\varphi^*D_1,\dots,\varphi^*D_{\delta-1}; r[K':K]^{-1}}$ can be constructed from $S_L$ and $K'/L$, then by Step~3, $\alpha$ is strong. This completes the proof of \ref{adelic curve:finite field extension}.

\vspace{2mm} \noindent
\textbf{Step~5.} {\it Keeping the notation in \ref{adelic curve:base field extension}, $K_1$ is the composite field of $K, k_1$ in $\overline{K}$ (after choosing suitable embeddings).}

\vspace{2mm} \noindent
 Obviously, $K, k_1\subset K_1$. Let $K'\subset \overline{K}$ be the composite field of $K, k_1$, we have a natural surjective morphism $K\otimes_kk_1\to K'$ by the universal property of tensor product. If $k_1/k$ is Galois, then $T_{k_1}$ is reduced and $K\otimes_kk_1$ is a product of fields which are isomorphic over $K$. Hence $K_1\simeq K'$ over $K$.  So after choosing suitable embeddings in $\overline{K}$, we may view $K_1$ as the composite field of $K, k_1$ in $\overline{K}$. If $k_1/k$ is purely inseparable, then $T_{k_1}$ is homeomorphic to $T$, and $K\otimes_kk_1$ is a local $K$-algebra with residue field $K_1$. Hence $K_1\simeq K'$ both over $K$ and over $k_1$. This completes the proof of Step~5.

\vspace{2mm} \noindent
\textbf{Step~6.} {\it Keeping the notation in \ref{adelic curve:base field extension}, if $k_1/k$ is finite, then \ref{adelic curve:base field extension} holds.}

\vspace{2mm} \noindent	
	Write 
	$S_2=S_{T_1;\psi^*D_1,\dots,\psi^*D_{\delta-1}}=(K_1,\Omega_2,\mathcal{A}_2,\nu_2)$. Furthermore, we set $S_2'=(K_1,\Omega_2,\mathcal{A}_2,\nu_2')$, where $\nu_2'$ is defined as follows: for any $\omega'\in\Omega_2$,
	\[\nu_2'(\omega')=r[K_1:K]^{-1}\cdot \deg_k(\psi^*D_1\cdots \psi^*D_{\delta-1}\cdot [\omega']),\]
	here we view $T_1$ as an integral scheme projective over $k$. 
	By Step~5, $S_2'$ is the adelic curve in the left-hand side of \eqref{eq:adelic curve of field extension1} for the finite extension $K_1/K$ and the canonical morphism $\psi\colon T_1\to T$. We claim that $S_2=S_2'$ which implies Step~6 by \ref{adelic curve:finite field extension}. Indeed, 
for any $\omega'\in \Omega_2=T_1^{(1)}$, 
\begin{align*}
	\nu_2(\omega')=& rm\cdot\deg_{k_1}(\psi^*D_1\cdots \psi^*D_{\delta-1}\cdot [\omega'])\\
	= &\frac{rm}{[k_1:k]}\cdot\deg_k(\psi^*D_1\cdots \psi^*D_{\delta-1}\cdot [\omega'])\\
	=&\frac{m[K_1:K]}{[k_1:k]}\cdot \nu_2'(\omega').
\end{align*}
Hence, to show the claim, it remains to show $\dim_k(k_1) = m\cdot[K_1:K]$, or equivalently, 
\begin{align}\label{eq:dim length degree}
\dim_{K}(K\otimes_kk_1)= \mathrm{length}(\OO_{T_{k_1},\eta_1})\cdot\#\mathrm{Irr}(T_{k_1})\cdot [K_1:K].
\end{align}
If $k_1/k$ is Galois, let $\eta_1,\dots, \eta_n$ be all the generic points of $T_{k_1}=T\times_{\Spec(k)}\Spec(k_1)$, and $K_1=\OO_{T_{k_1},\eta_1},\dots, K_n=\OO_{T_{k_1},\eta_n}$ the corresponding residue fields (hence $\mathrm{length}(\OO_{T_{k_1},\eta_j})=1$). Then $n=\#\mathrm{Irr}(T_{k_1})$ and
\[\text{$K\otimes_kk_1=\prod\limits_{j=1}^nK_j,$ and $K_j\overset{\sigma_j}{\simeq} K_1$}\]
 as $K$-algebras for any $j=1,\dots, n$. Hence \eqref{eq:dim length degree} holds. If $k_1/k$ is purely inseparable, then $\#\mathrm{Irr}(T_{k_1})=1$, $K\otimes_kk_1$ is a local $K$-algebra with residue field $K_1$.
By \cite[Lemma~7.1.36~(c)]{liu2006algebraic} (here we consider the morphism $K\to K\otimes_kk_1$), \eqref{eq:dim length degree} holds in this case. So Step~6 holds. 

\vspace{2mm} \noindent
\textbf{Step~7.} {\it \ref{adelic curve:base field extension} holds.}

\vspace{2mm} \noindent
We keep the notation in \ref{adelic curve:base field extension}. If $k_1/k$ is Galois (resp. purely inseparable), let $\mathscr{E}_{k_1/k}$ be the set of all finite Galois (resp. purely inseparable) extensions of $k$ which are contained in $k_1$. For any field extension ${k_1'/k}$ contained in $k_1$, we set $k_1'K$ the composite field of $K$ and $k_1'$ in $\overline{K}$, $T_{k_1'}\coloneq T\times_{\Spec(k)}\Spec(k_1')$, $T'_1$ the normalization of the Zariski image of $T_1$ in $T_{k_1'}$, and denote by $\psi'_1\colon T_1\to T_1'$, $\psi'\colon T_1'\to T$ the canonical morphisms. We also denote by $S_{T,k_1'}\coloneq (k'_1K, \Omega_{T,k_1'}, \mathcal{A}_{T,k_1'}, \nu_{T,k_1'})$ the adelic curve in the left-hand side of \eqref{eq:adelic curve of field extension2} for base field $k_1'$, $T_{k_1'}$ and $T_1'$. By Step~6, we have an isomorphism $\alpha'=({\alpha'}^\#, \alpha'_\#,\gamma_{\alpha'})\colon S_{T,k_1'}\simeq S_{k_1'K}$, where $\gamma_{\alpha'}(x')$ is the inverse of the ramification index of $\psi'$ at $x'$ for any  $x'\in \Omega_{T,k_1'}=T_1'^{(1)}$. 

To show \ref{adelic curve:base field extension}, we can assume that $m=1$, i.e. $\mathrm{length}(\OO_{T_{k_1},\eta_1})=\#\mathrm{Irr}(T_{k_1})=1$. Indeed, there is $k_1'\in \mathscr{E}_{k_1/k}$ such that 
\[\#\mathrm{Irr}(T_{k_1'})=\#\mathrm{Irr}(T_{k_1}) \text{ and } \mathrm{length}(\OO_{T_{k_1'},\eta_1'}) = \mathrm{length}(\OO_{T_{k_1},\eta_1}),\]
where $\eta_1'$ is the image of $\eta_1$ via the canonical morphism $T_{k_1}\to T_{k_1'}$. 
By Step~6, we have
$S_{T,k_1'}\simeq S_{k_1'K}.$
After replacing $S$ by $S_{T,k_1'}$, we have  $m=1$.
 
By Step~6 and \cref{def:extension of adelic curve}~\ref{general extension omega}, we have 
\begin{align}\label{eq:prop base change}
\varprojlim\limits_{k'\in \mathscr{E}_{k_1/k}}(\Omega_{T,k'_1}, \mathcal{A}_{T,k'_1})\simeq\varprojlim_{k_1'\in \mathscr{E}_{k_1/k}}(\Omega_{k_1'K},\mathcal{A}_{k_1'K})=(\Omega_{K_1},\mathcal{A}_{K_1})
\end{align}
in the category of measurable spaces.
It is easy to see that the left-hand side of \eqref{eq:prop base change} is isomorphic to $(\Omega_{T,k_1}, \mathcal{A}_{T,k_1})$ (from the projective limit of schemes). Hence we have an isomorphism 
\[(\Omega_{T,k_1}, \mathcal{A}_{T,k_1})\simeq (\Omega_{K_1},\mathcal{A}_{K_1})\] of measurable spaces, denoted by $\alpha_\#$. 
In particular, $\mathcal{A}_{K_1}$ is discrete. We also denote by $\alpha^\#$ the identity map of $K_1$. For any $x\in \Omega_{T,k_1}$, as before, we set $\gamma_\alpha(x)$ the inverse of the ramification index of $\psi$ at $x$. Then $\alpha=(\alpha^\#,\alpha_\#,\gamma_\alpha)\colon S_{T,k_1}\to S_{K_1}$ is a morphism of adelic curves. It remains to show that $\alpha$ is strong. Let $x\in \Omega_{T,k_1}$ and $\omega\coloneq \psi(x)$. We take a field $k_1'\in\mathscr{E}_{k_1/k}$ (we then use the notation at the beginning of the proof of this step) such that there is $x'\in \Omega_{T,k_1'}$ with ${\psi'_1}^{-1}(x')=\{x\}$ and $\overline{\{x'\}}\times_{\Spec(k_1')}\Spec(k_1)$ is reduced. Then $\gamma_\alpha(x)=\gamma_{\alpha'}(x')$, i.e. the ramification index of $\psi_1'$ at $x$ is $1$. By \cref{def:intersection number}~\ref{intersection number base change},
\begin{align}\label{eq:step7 1}
\nu_{T,k_1}(x)=r\deg_{k_1}(\psi^*D_1\cdots\psi^*D_{\delta-1}\cdot [x])=r\deg_{k_1'}({\psi'}^*D_1\cdots{\psi'}^*D_{\delta-1}\cdot [x'])= \nu_{T,k_1'}(x').
\end{align}
On the other hand, by \cref{def:extension of adelic curve}~\ref{extension of adelic curve 2b}~\ref{extension of adelic curve 2c}~\ref{extension of adelic curve 1d}~\ref{extension of adelic curve 1e},
\begin{align}\label{eq:step7 2}
\nu_{k_1K}(x)=\frac{[(k_1'K)_{x'}:K_\omega]_s}{[k_1'K:K]_s}\nu(\omega)=\nu_{k_1'K}(x').
\end{align}
Since $\alpha'\colon S_{T,k_1'}\simeq S_{k'K}$ is an isomorphism, we have  \begin{align}\label{eq:step7 3}
\nu_{k_1'K}(x') = \gamma_{\alpha'}(x')^{-1}\nu_{T,k_1'}(x').
\end{align} Then
$$\nu_{k_1K}(x) = \gamma_{\alpha}(x)^{-1}\nu_{T,k_1}(x)$$
follows from \eqref{eq:step7 1}, \eqref{eq:step7 2} and \eqref{eq:step7 3}.
Hence $\alpha$ is strong by \cref{rmk:strong between discrete sigma}. This completes the proof of Step~7.
\end{proof}


In this paper, we assume that every adelic curve $(K,\Omega,\mathcal{A},\nu)$ satisfies the following properties:  
\begin{enumeratea}
	\item for any $\omega\in\Omega$, the absolute value $\val_\omega$ is non-trivial and non-archimedean;
	\item the $\sigma$-algebra $\mathcal{A}$ is discrete;
	\item for any $\omega\in\Omega$, $\nu(\omega)\not\in\{0,\infty\}$.
\end{enumeratea}

\section{Compactified metrized line bundles}
\label{sec:compactified}

In this section, we fix a field $K$ and a $d$-dimensional quasi-projective variety $U$ over $K$.

\subsection{Compactified (geometric) line bundles}



\begin{art}
		A \emph{projective $K$-model} of $U$ is a variety $X$ projective over $K$ with an open immersion $U\hookrightarrow X$. A (\emph{geometric}) \emph{boundary divisor} of $U$ is a pair $({X}_0,B)$ consisting of a projective $K$-model $U\hookrightarrow X_0$ and an effective divisor $B\in \Div_\Q(X_0)$ such that  $\Supp(B)=X_0\setminus U$.
\end{art}

\begin{art} \label{geometric intersection number}
	For a projective variety $X$ over $K$, let $\Pic_{\Q}(X)_\nef$ be the cone of $\Pic_\Q(X)$ generated by the nef $\Q$-line bundles on $X$. We set 
\[\text{$P_{\gm,\Q}(U) = \varinjlim_{X}\Pic_\Q(X)$ \ and \ $Q_{\gm,\Q}(U) = \varinjlim_{X}\Pic_\Q(X)_\nef$,}\]
where $X$ ranges over all projective $K$-models of $U$.

	Let $(X_0, {B})$ be a boundary divisor of $U$. Let $L\in P_{\gm,\Q}(U)$, we denote by $L|_U \in \Pic_\Q(U)$ the restriction of $L$ to $U$. For $r \in \Q_{>0}$, we define $B(r,{L}) \subset P_{\gm,\Q}(U)$ as follows. An element $L' \in P_{\gm,\Q}(U)$ is in $B(r,L)$ if and only if there is an isomorphism $\sigma\colon L|_U\simeq L'|_U$ for some representatives $L,L'$ living on a common projective $K$-model $X$ such that the induced $\Q$-Cartier divisor $\mathrm{div}(\sigma)$ of $L'\otimes L^{-1}$ on $X$ satisfies 
$$-rB\leq \mathrm{div}(\sigma)\leq rB.$$
There is a unique topology on  $P_{\gm,\Q}(U)$, called the \emph{boundary topology}, such that $(B(r,L))_{r \in \Q_{>0}}$ forms a neighborhood basis at each $L\in P_{\gm,\Q}(U)$.
The boundary topology is independent of the choice of the  boundary divisor. Set $\widetilde{\Pic}_\Q(U)_\cpt$ the completion of $P_{\gm,\Q}(U)$ with respect to the boundary topology, $\widetilde{\Pic}_\Q(U)_\snef$ the closure of $Q_{\gm,\Q}(U)$ in $\widetilde{\Pic}_\Q(U)_\cpt$, and $\widetilde{\Pic}_\Q(U)_\integrable\coloneq \widetilde{\Pic}_\Q(U)_\snef-\widetilde{\Pic}_\Q(U)_\snef$. An element in $\widetilde{\Pic}_\Q(U)_\cpt$ is called a \emph{compactified} (\emph{geometric}) \emph{line bundle} on $U$. We have a symmetric multilinear map
	\[\underbrace{\widetilde{\Pic}_{\Q}({U})_\integrable\times\cdots\times\widetilde{\Pic}_{\Q}({U})_\integrable}_{d\text{-times}}\to \R, \ \ (\widetilde{L_1}, \dots, \widetilde{L_d})\mapsto \widetilde{L_1}\cdots\widetilde{L_d}\]
which is induced by the intersection numbers. For $\widetilde{L}\in \widetilde{\Pic}_\Q(U)_\integrable$ and an integral closed subvariety $Y$ of $U_{\overline{K}}\coloneq U\times_{\Spec(K)}\Spec(\overline{K})$, we have the restriction  $\widetilde{L}|_Y$ of  $\widetilde{L}$ to $Y$, and $\widetilde{L}|_Y\in \widetilde{\Pic}_\Q(Y)_\integrable$, see \cite[\S 2.6.5]{yuan2021adelic}. We denote $$\deg_{\widetilde{L}}(Y)\coloneq (\widetilde{L}|_Y)^{\dim(Y)}.$$
\end{art}
{
\begin{art}
Let $L\in \Pic_\Q(U)$. A \emph{rational section} of $L$ on $U$ is an element of $$H^0(\eta,L)\coloneq\varinjlim_mH^0(\eta,L^{\otimes m}),$$
	where $\eta$ is the generic point of $U$, and $m$ runs through positive integers such that ${L}^{\otimes m}\in \Pic(U)$. If a rational section $s$ of $L$ is represented by $s_m\in H^0(\eta,L^{\otimes m})$, then define 
	\[\mathrm{div}(s)=\frac{1}{m}\mathrm{div}(s_m)\in \Div_\Q(U).\]
	We define $H^0(U,L)=\{s\in H^0(\eta,L)\mid \mathrm{div}(s)\geq 0\}$. Notice that $H^0(U,L)$ is a $K$-vector space if $U$ is normal (see \cite[Proposition~2.4.13~(6)]{chen2020arakelov}). 
\end{art}}

\begin{art}\label{def:compactified divisor and sections}
	{As in the case of $P_{\gm,\Q}(U)$ in \cref{geometric intersection number}, given a boundary divisor $(X_0,B)$ of $U$, we can define a boundary topology on $${\Div}_{\Q}(U)_\mo\coloneq \varinjlim_{X}\Div_\Q(X),$$
	where $X$ ranges over all projective $K$-models of $U$. Set $\widetilde{\Div}_\Q(U)_\cpt$ the completion of $\Div_\Q(U)_\mo$ with respect to the boundary topology, see \cite[Definition~3.5]{cai2024abstract} for a more precise definition.	We have a surjective morphism $\widetilde{\Div}_\Q(U)_\cpt\twoheadrightarrow\widetilde{\Pic}_\Q(U)_\cpt$.}
	
	Let $\widetilde{L}\in\widetilde{\Pic}_\Q(U)_\cpt$. A \emph{rational section} $s$ of $\widetilde{L}$ is given by a Cauchy sequence $(L_n)_{n\in\N_{\geq 1}}$ in $P_{\gm,\Q}(U)$ representing $\widetilde{L}$ and a sequence of rational sections $s_n$ of $L_n$ satisfying the following Cauchy condition: for any $\varepsilon\in\Q_{>0}$, there exists $n_0\in\N$ with 
	$$-\varepsilon B \leq {\rm div}(s_n) - {\rm div}(s_m) \leq \varepsilon B$$ 
	for all $n,m \geq n_0$. Then a rational section $s= (s_n)_{n\in\N}$ of $\widetilde{L}$ naturally gives an element $(\mathrm{div}(s_n))_{n\in\N}$ in $\widetilde{\Div}_\Q(U)_\cpt$, denoted by $\mathrm{div}(s)$. We consider two rational sections $s$, $s'$ to be equivalent if $\mathrm{div}(s)=\mathrm{div}(s')$. We define $H^0(\eta,\widetilde{L})$ as the set of rational sections of $\widetilde{L}$ (here $\eta$ is the generic point of $U$), and 
	\[H^0(U,\widetilde{L})= \{s\in H^0(\eta,\widetilde{L})\mid \mathrm{div}(s)\geq 0 \}.\]
Notice that $H^0(U,\widetilde{L})\subset H^0(U,\widetilde{L}|_U)$, 
where $\widetilde{L}|_U\in\Pic_\Q(U)$ denotes the underlying $\Q$-line bundle 
of $\widetilde{L}$ on $U$.
\end{art}

\subsection{Local theory: compactified metrized line bundle}

In this subsection, we assume that $K$ is complete with respect to a non-trivial non-archimedean absolute value $\val_v$. We denote by $U^\an=U_v^\an$ the Berkovich analytification of $U$ with respect to $\val_v$, see \cite[\S 3.4, \S 3.5]{berkovich1990spectral}. 

\begin{art} \label{metrics}
	Let $L$ be a line bundle on $U$. Then $L$ induces an analytic line bundle $L^\an=L_v^\an$ on $U^\an$. A \emph{metric} $\metr$ of $L$ is a continuous family of $v$-norms on fibers of $L^\an$ over $U^\an$. The isometry classes of line bundles with continuous metrics on $U^\an$ form a group, denoted by $\widehat{\Pic}(U)$. We set $\widehat{\Pic}_\Q(U)\coloneq \widehat{\Pic}(U)\otimes_\Z\Q.$
	
	Assume that $U=X$ is projective. A (\emph{projective}) \emph{$K^\circ$-model} of $L$ is a pair $(\mathcal{X},\mathcal{L})$, where $\mathcal{X}$ is a flat projective scheme over the valuation ring $K^\circ$ and $\mathcal{L}$ is a line bundle on $\mathcal{X}$ such that the generic fiber of $\mathcal{X}$ is $X$ and $\mathcal{L}|_X=L$. A metric $\metr$ on $L$ is called a \emph{model metric} if there is a nonzero $m\in\N$ and a $K^\circ$-model $(\mathcal{X},\mathcal{L})$ of $L^{\otimes m}$ such that $\metr^{\otimes m}$ is induced by $(\mathcal{X},\mathcal{L})$, see \cite{boucksom2021non}. A model metric $\metr$ of $L$ is called \emph{semipositive} if the corresponding line bundle $\mathcal{L}$ is nef. We denote by $P_\Q(X)$ the subspace of $\widehat{\Pic}_\Q(X)$ spanned by line bundles with model metrics, and by $Q_\Q(X)$ the subcone generated by line bundles with semipositive model metrics. 
\end{art} 


\begin{art} \label{Green functions}
Let $D\in \Div_\Q(U)$. Locally $D$ can be written as local equations $f \in \OO_U(V)^\times \otimes_\Z \Q$.  Then we get a well-defined continuous real function $\log|f|$ on $V^\an$ for such open subsets $V$.  A \emph{Green function} for $D$ is a continuous real function $g_D$ on $(U \setminus \Supp(D))^\an$ such that for any local equation $f$ of $D$ on an open subset $V$ of $U$, we have that $g_D+\log|f|$ extends to a continuous function on $V^\an$. Notice that giving a Green function for a Cartier divisor $D$ is equivalent to giving a continuous metric on $\OO_U(D)$.
We denote by $\widehat{\Div}(U)$ (resp. $\widehat{\Div}_\Q(U)$) the group of pairs $(D,g)$ with $D\in \Div(U)$ (resp. $D\in\Div_\Q(U)$) and $g$ a Green function for $D$. Obviously, we have $\widehat{\Div}_\Q(U)\simeq\widehat{\Div}(U)\otimes_\Z\Q$. 

Notice that $\widehat{\Div}_\Q(U)$ is an ordered $\Q$-vector space. For any $\overline{D}\in \widehat{\Div}_\Q(U)$, we say that $\overline{D}$ is \emph{effective} if $\overline{D}\geq 0$; we say that $\overline{D}$ is \emph{strictly effective} if $\overline{D}$ is effective and $\overline{D}\not=0$. 
\end{art}

	

\begin{art} \label{model Green function}
	Let $X$ be a projective variety over $K$, and $D\in\Div(X)$. A \emph{model Green function} for $D$ is a Green function $g$ for $D$ which is induced by a model metric $\metr$ on $\OO_X(D)$. We denote by $\widehat{\Div}(X)_{\mathrm{mo}}$ the subgroup of $\widehat{\Div}(X)$ consisting of pairs $(D,g)$ with $g$ a model Green function. Set $\widehat{\Div}_\Q(X)_{\mathrm{mo}}\coloneq \widehat{\Div}(X)_{\mathrm{mo}}\otimes_\Z\Q.$ 
	
	A (\emph{cofinal}) \emph{boundary divisor} of $U$ is a pair $(X_0,\overline{B})$ consisting of a projective $K$-model $U\hookrightarrow X_0$ and a strictly effective element $\overline{B}\in\widehat{\Div}_\Q(X_0)_{\mathrm{mo}}$ such that $U=X_0\setminus\Supp(B)$.
\end{art}

\begin{art} \label{boundary topology on metrized line bundles}
	We fix a (cofinal) boundary divisor $(X_0, \overline{B})$ of $U$ with $\overline{B}=(B, g_B)$. Notice that $g_B$ is a continuous function on $U^\an$. The \emph{boundary topology} on $\widehat{\Pic}_\Q(U)$ is defined such that a neighborhood basis at an element $\overline{L}=(L, \metr)\in \widehat{\Pic}_\Q(U)$ of the topology is given by
	\[B(r,\overline{L})\coloneq \left\{(L,\metr')\in \widehat\Pic_\Q(U) \,\middle\vert\, -rg_B\leq \log\frac{\metr'}{\metr}\leq rg_B\right\}, \quad r\in \Q_{>0}.\]
	The boundary topology is independent of the choice of the cofinal boundary divisor. 
	Then $\widehat{\Pic}_\Q(U)$ is complete since the space of continuous functions on $U^\an$ is complete, see \cite[Lemma~3.6.3]{yuan2021adelic}. 
	
	As in \cref{geometric intersection number}, set 
\[\text{$P_{\Q}(U) = \varinjlim_{X}P_\Q(X)$ (resp. $Q_{\Q}(U) = \varinjlim_{X}Q_\Q(X)$),}\]
and denote by $\widehat{\Pic}_\Q(U)_\cpt$ (resp. $\widehat{\Pic}_\Q(U)_\snef$) the closure of $P_{\Q}(U)$ (resp. $Q_{\Q}(U)$) in  $\widehat{\Pic}_\Q(U)$. Set $\widehat{\Pic}_\Q(U)_\integrable= \widehat{\Pic}_\Q(U)_\snef-\widehat{\Pic}_\Q(U)_\snef$ and denote by $\widehat{\Pic}_\Q(U)_\nef$ the closure of $\widehat{\Pic}_\Q(U)_\snef$ in $\widehat{\Pic}_\Q(U)_\integrable$ with respect to the finite subspace topology.

Given elements $\overline{L_1}, \dots, \overline{L_d} \in \widehat{\Pic}_\Q(U)_\nef$, we can associate a positive Radon measure $c_1(\overline{L_1})\wedge\cdots\wedge c_1(\overline{L_d})$ on $U^\an$, see \cite[\S 3.6.7]{yuan2021adelic} or \cite[Proposition~4.45]{cai2024abstract}. 
\end{art}
\begin{remark}
	Let $X$ be a projective variety over $K$. By \cite[Proposition~4.14]{cai2024abstract}, a metrized $\Q$-line bundle $(L,\metr)\in \widehat{\Pic}_\Q(X)_\snef$ if and only if $\metr$ is a uniform limit of semipositive model metrics of $L$.
\end{remark}

\subsection{Global theory: compactified $S$-metrized $\YZ$-line bundles}
\label{subsection:compactified S-metrized YZ-line bundles}
In this subsection, we fix a field $k$, a normal integral {scheme} $T$ projective over $k$ with function field $K=k(T)$ and an ample class $\mathbf{c}\in\Pic(T)$. Set $\delta=\dim(T)$ and $S=S_{T,\mathbf{c},\dots, \mathbf{c}}= (K,\Omega,\mathcal{A}, \nu)$ the adelic curve associated to $T, \mathbf{c},\dots,\mathbf{c}$ ($(\delta-1)$-times) defined in \cref{adelic curve defined by projective variety} ($r=1$). 

\begin{art}
	Let $L$ be a line bundle on $U$. An \emph{$S$-metric} on $L$ is a family of metrics $\metr_\omega$ on $L_\omega^\an$, $\omega\in\Omega$.  The isometry classes of line bundles on $U$ endowed with \emph{locally $S$-bounded} metrics (see \cite[Definition~6.9]{cai2024abstract}) form a group, denoted by $\widehat{\Pic}_{S}(U)$ (notice that $\mathcal{A}$ is discrete by assumption). We also denote $\widehat{\Pic}_{S,\Q}(U) \coloneq\widehat{\Pic}_{S}(U)\otimes_\Z\Q$. 
\end{art}

\begin{art}\label{global:weakly boundary}
	For a Cartier divisor (resp.~a $\Q$-Cartier divisor) $D$ of $U$, an \emph{$S$-Green function} $g$ is a family of Green functions $g_{\omega}$ for the base change $D_\omega$ of $D$ to $U_\omega$ with $\omega$ running over $\Omega$. There is an associated continuous metric $\metr_\omega$ of $L_\omega^\an$ for the line bundle (resp.~$\Q$-line bundle) $L=\OO_U(D)$. 
	We define $\widehat{\Div}_{S}(U)$ (resp. $\widehat{\Div}_{S,\Q}(U)$) as the group of Cartier divisors (resp. $\Q$-Cartier divisors) with Green functions corresponding to locally $S$-bounded metrics, so we have a surjective morphism
	\[\widehat{\Div}_{S,\Q}(U)\to \widehat{\Pic}_{S,\Q}(U).\]
Notice that $\widehat{\Div}_{S,\Q}(U)$ is an ordered $\Q$-vector space. For any $\overline{D}\in \widehat{\Div}_{S,\Q}(U)$, we say that $\overline{D}$ is \emph{effective} if $\overline{D}\geq 0$.
\end{art}

\begin{art} \label{model metrics and model Green functions}
	Let $X$ be a projective variety over $K$, and $L$ a line bundle on $X$. A \emph{model $S$-metric} $\metr$ on $L$ is a family of metrics $\metr_\omega$ on $L_\omega^\an$, $\omega\in\Omega$, 
	satisfying the following property: there is a projective scheme $\mathcal{X}$ over $T$ with generic fiber $X$, $N\in \N_{\geq 1}$ and a line bundle $\mathcal{L}$ over $\mathcal{X}$ such that $L^{\otimes N}=\mathcal{L}|_X$ and $\metr_\omega^{\otimes N}$ is induced by $\mathcal{L}$ at any $\omega\in \Omega$. We denote by $\widehat{\Pic}_{S}(X)_\mo$ the group of isometry classes of line bundles with model $S$-metrics. 
	An element $(L,\metr)\in \widehat{\Pic}_{S}(X)_\mo$ 
	\begin{itemize}
		\item is in $R_{S}(X)_\mo$ if and only if there are $\mathcal{X}$, $N\in\N_{\geq 1}$ and $\mathcal{L}$ as above such that the reduction of $\mathcal{L}$ at each $\omega\in\Omega$ is nef.
		\item is in $Q_S(X)_\mo$ if and only if there are $\mathcal{X}$, $N\in\N_{\geq 1}$ and $\mathcal{L}$ as above such that $\mathcal{L}$ is nef on $\mathcal{X}$.
	\end{itemize} We set the space of \emph{model $S$-metrized} ($\Q$-)\emph{line bundle} as $\widehat{\Pic}_{S,\Q}(X)_\mo \coloneq \widehat{\Pic}_{S}(X)_\mo\otimes_\Z\Q$, and denote by $R_{S,\Q}(X)_\mo$ (resp. $Q_{S,\Q}(X)_\mo$) the cone in $\widehat{\Pic}_{S,\Q}(X)_\mo$ generated by $R_S(X)_\mo$ (resp. $Q_{S}(X)_\mo$).

\end{art}
\begin{remark}
	For a projective variety $X$ over $K$, we have $\widehat{\Pic}_{S}(X)_\mo\subset \widehat{\Pic}_{S}(X)$. 
	Let $\mathcal{X}$ be a projective scheme over $T$, $\mathcal{L}$ a line bundle on $\mathcal{X}$, and $\overline{L}$  the model $S$-metrized line bundle on $X$ corresponding to $\mathcal{L}$. We shall show that $\overline{L}\in\widehat{\Pic}_{S}(X)$. We take an affine open subset $W=\Spec(A)\subset T$ such that every fiber $\mathcal{X}_y$ of $\mathcal{X}\to T$ is reduced for any $y\in W$, such a $W$ exists since the smooth locus forms an open subset of $\mathcal{X}$. We firstly assume that $\mathcal{L}$ is very ample on $\mathcal{X}_W\coloneq\mathcal{X}\times_TW$ over $W$ with global sections $s_1,\dots, s_m$ which generate $\mathcal{L}|_{\mathcal{X}_y}$ for any $y\in W$.
	By \cite[Proposition~2.3.12~(3), Corollary~2.3.13, Proposition~2.3.16~(3)]{chen2020arakelov}, for any $\omega\in\Omega\cap W$, the model metric $\metr_\omega$ is the Fubini-Study metric induced by $s_1,\dots, s_m$. Notice that $\mathcal{A}$ is discrete and $\Omega\setminus W$ is finite,  so $\overline{L}$ is $S$-bounded (or $S$-dominated), see \cite[Remark~6.13]{cai2024abstract} (or \cite[Definition~4.1.8]{chen2021arithmetic}), hence $\overline{L}\in \widehat{\Pic}_{S}(X)$. In general, write $\mathcal{L}|_{\mathcal{X}_W}=\mathcal{L}_1\otimes\mathcal{L}_2^{\otimes -1}$ with $\mathcal{L}_1, \mathcal{L}_2\in \Pic(\mathcal{X}_W)$ very ample over $W$. We can endow $L_1\coloneq\mathcal{L}_1|_X$ (resp. $L_2\coloneq\mathcal{L}_2|_X$) with an arbitrary metric $\metr_{1,\omega}$ (resp. with $\metr_{2,\omega}\coloneq \metr_{1,\omega}/\metr_\omega$) when $\omega\in\Omega\setminus W$, with the model metric $\metr_{1,\omega}$ (resp. $\metr_{2,\omega}$) given by $\mathcal{L}_1$ (resp. $\mathcal{L}_2$) when $\omega\in \Omega\cap W$. Then $(L_1,\metr_1), (L_2, \metr_2)\in \widehat{\Pic}_{S}(X)$ by the discussion above. Hence $\overline{L}\in \widehat{\Pic}_{S}(X)$.
\end{remark}

\begin{art}
	Let $X$ be a projective variety over $K$. A \emph{model $S$-Green function} for a divisor $D\in\Div_\Q(X)$ is an $S$-Green function $g$ such that the $\Q$-line bundle $\OO_X(D)$ with the corresponding metric $(\metr_\omega)_{\omega\in\Omega}$ is a model $S$-metrized line bundle. We denote by $\widehat{\Div}_{S,\Q}(X)_{\mo}$ the group of pairs $(D,g)$ with $D$ a $\Q$-Cartier divisor on $X$  and $g$ a model $S$-Green function for $D$ which is a subspace of $\widehat{\Div}_{S,\Q}(X)$. We denote by $N_{S,\Q}(U)_\mo$ the subcone of $\widehat{\Div}_{S,\Q}(X)_\mo$ consisting of pairs $(D,g)$ such that the corresponding $S$-metrized ($\Q$-)line bundle $(\OO_U(D),\metr)\in Q_{S,\Q}(X)_\mo$.
	
	A (\emph{model}) \emph{weak boundary divisor} of $U$ is a pair $({X}_0,\overline{B})$ consisting of a projective $K$-model ${X}_0$ of $U$ and an effective divisor $\overline{B}\in \widehat{\Div}_{S,\Q}(X)_{\mo}$  such that $\Supp(B)\subset{X}_0\setminus U$. 
\end{art}

\begin{art} 	\label{def:boundarytopologyglobal}
 Since the projective $K$-models of $U$ form a directed system, it is clear that the (model) weak boundary divisors form a directed subset of $\widehat{\Div}_{S,\Q}(U)_{\mo}$. 
	As in \cref{boundary topology on metrized line bundles}, a (model) weak boundary divisor $(X_0,\overline{B})$ gives a topology, called the \emph{$\overline{B}$-boundary topology}, on $\widehat{\Pic}_{S,\Q}(U)$, such that a neighborhood basis at an element $\overline{L}=(L,\metr)\in \widehat{\Pic}_{S,\Q}(U)$  is given by
	$$B(r,\overline{L})\coloneq\left\{(L,\metr')\in \widehat{\Pic}_{S,\Q}(U)\mid -rg_{B,\omega}\leq \log\frac{\metr_\omega'}{\metr_\omega}\leq rg_{B,\omega} \text{ for any $\omega\in\Omega$}\right\}, \, r \in \Q_{>0}.$$
	As in the local case, the space $\widehat{\Pic}_{S,\Q}(U)$ is complete with respect to the $\overline{B}$-boundary topology for any (model) weak boundary divisor $\overline{B}$ of $U$.
	
	Set
	\[P_{S,\Q}(U)_\mo = \varinjlim_{X}\widehat{\Pic}_{S,\Q}(X)_\mo\]
	\[\text{(resp. $R_{S,\Q}(U)_\mo = \varinjlim_{X}R_{S,\Q}(X)_\mo$, resp. $Q_{S,\Q}(U)_\mo = \varinjlim_{X}Q_{S,\Q}(X)_\mo$),}\]
	where $X$ ranges over all projective $K$-models of $U$. For any (model) weak boundary divisor $\overline{B}$ of $U$, we denote by $P_{S,\Q}(U)^{d_{\overline{B}}}_\mo$ (resp. $R_{S,\Q}(U)^{d_{\overline{B}}}_\mo$, resp. $Q_{S,\Q}(U)^{d_{\overline{B}}}_\mo$) the closure of $P_{S,\Q}(U)_\mo$ (resp. $R_{S,\Q}(U)$, resp. $Q_{S,\Q}(U)_\mo$) in $\widehat{\Pic}_{S,\Q}(U)$ with respect to the $\overline{B}$-boundary topology, and set 
	\[\widehat{\Pic}_{S,\Q}(U)_\cpt^\YZ\coloneq \varinjlim_{\overline{B}}P_{S,\Q}(U)^{d_{\overline{B}}}_\mo\] \[\text{(resp. $\widehat{\Pic}_{S,\Q}(U)_\relsnef^\YZ\coloneq \varinjlim_{\overline{B}}R_{S,\Q}(U)^{d_{\overline{B}}}_\mo$, resp. $\widehat{\Pic}_{S,\Q}(U)_\snef^\YZ\coloneq \varinjlim_{\overline{B}}Q_{S,\Q}(U)^{d_{\overline{B}}}_\mo$),}\]
	where $\overline{B}$ runs through all (model) weak boundary divisors of $U$.
	We also set
	\[\widehat{\Pic}_{S,\Q}(U)_\relint^\YZ\coloneq\widehat{\Pic}_{S,\Q}(U)_\relsnef^\YZ-\widehat{\Pic}_{S,\Q}(U)_\relsnef^\YZ\]
	\[\text{(resp. $\widehat{\Pic}_{S,\Q}(U)_\integrable^\YZ\coloneq\widehat{\Pic}_{S,\Q}(U)_\snef^\YZ-\widehat{\Pic}_{S,\Q}(U)_\snef^\YZ$)},\]
	and denote by $\widehat{\Pic}_{S,\Q}(U)_\relnef^\YZ$ (resp. $\widehat{\Pic}_{S,\Q}(U)_\nef^\YZ$) the closure of $\widehat{\Pic}_{S,\Q}(U)_\relsnef^\YZ$ (resp. $\widehat{\Pic}_{S,\Q}(U)_\snef^\YZ$) in $\widehat{\Pic}_{S,\Q}(U)_\relint^\YZ$ (resp. $\widehat{\Pic}_{S,\Q}(U)_\integrable^\YZ$) with respect to the finite subspace topology. We have a symmetric multilinear map 
	\[\underbrace{\widehat{\Pic}_{S,\Q}({U})_\integrable^\YZ\times\cdots\times\widehat{\Pic}_{S,\Q}({U})_\integrable^\YZ}_{(d+1)\text{-times}}\to \R, \ \ (\overline{L_0}, \dots, \overline{L_d})\to (\overline{L_0}\cdots\overline{L_d}\mid U)_S\]
	induced by the arithmetic intersection numbers. An element in $\widehat{\Pic}_{S,\Q}(U)_\cpt^\YZ$ is called a \emph{compactified $S$-metrized $\YZ$-line bundle}. A compactified $S$-metrized ($\YZ$-)line bundle $(L,\metr)$ is called \emph{relatively nef}  (resp. \emph{nef}) if $(L,\metr)\in \widehat{\Pic}_{S,\Q}(U)_\relnef^\YZ$ (resp. $(L,\metr)\in \widehat{\Pic}_{S,\Q}(U)_\nef^\YZ$).
	We have a canonical homomorphism
	\begin{align}\label{eq:canonical morphism from metrzied to geometric}
\widehat{\Pic}_{S,\Q}(U)_\cpt^\YZ\to \widetilde{\Pic}_\Q(U)_\cpt,  \ \ \overline{L}\mapsto \widetilde{L}.
	\end{align}
\end{art}

\begin{remark} \label{rmk:relation with yz adelic line bundle}
	Since $U$ is a quasi-projective variety over $K=k(T)$, we have that $U$ is an essential quasi-projective variety over $k$ in the sense of \cite[\S 2.3.2]{yuan2021adelic}. For every quasi-projective $k$-model $\mathcal{U}$ of $U$ and every projective $k$-model $\mathcal{X}$ of $\mathcal{U}$, we have a \emph{pro-open immersion} $U\to \mathcal{X}$ over $k$ (see \cite[\S 2.3.1]{yuan2021adelic}) which gives a rational map $\mathcal{X}\dashrightarrow T$. We can assume that this rational map is a morphism after replacing $\mathcal{X}$ by a $k$-projective model of $U$ birational to $\mathcal{X}$. 
	Hence, in this case, we have a homomorphism
	\[P_{\gm,\Q}(\mathcal{U}) \rightarrow P_{S,\Q}(U)_\mo.\]
	We set $\widetilde{\Pic}_\Q(U/k)^\YZ_\cpt\coloneq\varinjlim\limits_{\mathcal{U}}\widetilde{\Pic}_{\Q}(\mathcal{U})_\cpt$, where $\mathcal{U}$ runs through all quasi-projective $k$-models of $U$ and $\widetilde{\Pic}_{\Q}(\mathcal{U})_\cpt$ (here $\mathcal{U}$ is viewed as a variety over $k$) is the completion of $P_{\gm,\Q}(\mathcal{U})$ with respect to the boundary topology, see \cref{geometric intersection number}.
	Hence we have a homomorphism
	\[\iota\colon\widetilde{\Pic}_{\Q}(U/k)_\cpt^\YZ\rightarrow \widehat{\Pic}_{S,\Q}(U)_\cpt^\YZ\]
	which is injective by \cite[Proposition~3.5.1, \S 3.5.3]{yuan2021adelic}. 
	
	Let $\mathcal{U}$ be a quasi-projective $k$-model of $U$, and $\mathcal{L}_0,\dots,\mathcal{L}_{d}\in \widetilde{\Pic}_{\Q}(\mathcal{U})_{\mathrm{int}}$. Notice that $\iota(\mathcal{L}_i)\in \widehat{\Pic}_{S,\Q}(U)^\YZ_\integrable$. Assuming that the rational map $\pi\colon\mathcal{U}\dashrightarrow T$ is a morphism, we have 
	\begin{align}\label{eq:arithmetic = geometric}
		(\iota(\mathcal{L}_0)\cdots\iota(\mathcal{L}_d)\mid U)_{S}=\mathcal{L}_0\cdots\mathcal{L}_d\cdot(\pi^*\mathbf{c})^{\delta-1},
	\end{align}
	where the product on the right-hand side is intersection number induced by the geometric intersection pairing on $\mathcal{U}$ in \cref{geometric intersection number}. If $\mathcal{L}_0,\dots,\mathcal{L}_{d}$ are on projective $k$-models of $\mathcal{U}$, then \eqref{eq:arithmetic = geometric} follows from \cite[Proposition~4.5.1]{chen2021arithmetic}. In general, \eqref{eq:arithmetic = geometric} is deduced from the projective case and the linearity, continuity of intersection pairings. In particular, if $\dim(T)=1$, we have  $(\iota(\mathcal{L}_0)\cdots\iota(\mathcal{L}_d)\mid U)_{S}=\mathcal{L}_0\cdots\mathcal{L}_d$. Moreover, it is not hard to show that when $\dim(T)=1$, for compactified $S$-metrized $\YZ$-line bundles, similar results to \cite[Theorem~5.2.1, Theorem~5.2.2, Lemma~5.3.4]{yuan2021adelic} hold. 
\end{remark}
\begin{remark}
Given finitely many $\overline{L_1},\dots, \overline{L_t}\in\widehat{\Pic}_{S,\Q}(U)_\relnef^\YZ$, after shrinking $U$, we can find a weak boundary divisor $\overline{B}\in N_{S,\Q}(U)_{\mo}$ such that there are sequences $(\overline{L_{j,m}})_{m\in\N_{\geq 1}}$ in $R_{S,\Q}(U)_\mo$ converging to $\overline{L_j}$ with respect to the $\overline{B}$-boundary topology. It is similar for nef case.
\end{remark}

\begin{remark}\label{rmk:arithmetic intersection number for subvarieties}
	For a $t$-dimensional variety $Y$ of $U_{\overline{K}}$ and $\overline{L_0},\dots, \overline{L_t}\in\widehat{\Pic}_{S,\Q}(U)_\integrable^\YZ$, the restrictions $\overline{L_j}|_Y\in\widehat{\Pic}_{S_{K'},\Q}(Y)_\integrable^\YZ$ of $\overline{L_j}$ to $Y$ are well-defined, and we denote 
	\[(\overline{L_0}\cdots \overline{L_t}\mid Y)_S\coloneq (\overline{L_0}|_Y\cdots \overline{L_t}|_Y\mid Y)_{S_{K'}},\]
	where $K'/K$ is a finite extension such that $Y$ is defined over $K'$, and $S_{K'}=(K',\Omega_{K'}, \mathcal{A}_{K'}, \nu_{K'})$ is the extension of $S$ over $K'/K$, see \cref{def:extension of adelic curve}. Notice that $S_{K'}$ is also of the form given in \cref{adelic curve defined by projective variety} by \cref{prop:adelic curve of field extension}~\ref{adelic curve:finite field extension}. The number is independent of the choices of $K'/K$, see \cite[Theorem~7.22~(vi)]{cai2024abstract}. 
	In particular, if $Y=x\in U(\overline{K})$, for any $\overline{L}\in\widehat{\Pic}_{S,\Q}(U)_\integrable^\YZ$, we set
	\[h_{\overline{L}}(x)=h_{S,\overline{L}}(x)\coloneq (\overline{L}\mid x)_S.\]
	Similar for the relatively nef case in \cref{extension of YZ intersection number}. 
\end{remark}

The notion of effectiveness will be used in the proof of the equidistribution theorem for subvarieties. It is easier to be understood using divisors.

\begin{art}\label{def:effective compactified metrized line bundle}
	Let $\overline{L}=(L,\metr)\in \widehat{\Pic}_{S,\Q}(U)_\cpt^\YZ$. We say that $\overline{L}$ is \emph{effective} if there is a \emph{rational section} (see \cite[\S 2.3.1]{yuan2021adelic}) $s$ of $L$ and a sequence $(\overline{L_m})_{m\in\N_{\geq 1}} = (L_m,\metr_m)_{m\in\N_{\geq 1}} \subset P_{S,\Q}(U)_\mo$ converging to $\overline{L}$ with respect to the $\overline{B}$-boundary topology for some weak boundary divisor $\overline{B}\in\widehat{\Div}_{S,\Q}(U)_\mo$ such that the divisors $(\mathrm{div}(s), -\log\|s\|_m)\in\widehat{\Div}_{S,\Q}(U)_\mo$ are effective, where we view $s$ as a rational section of $L_m$.
	
	If $\overline{L}\in \widehat{\Pic}_{S,\Q}(U)_\integrable^\YZ$ is effective, then for any $\overline{L_1},\dots, \overline{L_d}\in\widehat{\Pic}_{S,\Q}(U)_\nef^\YZ$, we have
	\[(\overline{L_1}\cdots\overline{L_d}\cdot\overline{L}\mid U)_S\geq 0,\]
	see \cite[Proposition~7.22~(ii)]{cai2024abstract}.
\end{art}
\begin{remark}\label{remark:effective compactified metrized line bundle}
	If $\overline{L}=(L,\metr)\in \widehat{\Pic}_{S,\Q}(U)_\cpt^\YZ$ is effective, let $s$ be a rational section of $L$ as in \cref{def:effective compactified metrized line bundle}. Let $Y$ be a subvariety of $U$. If $Y$ is not contained in the support $\Supp(\mathrm{div}(s))$, then the restriction $\overline{L}|_Y$ is effective.
\end{remark}

We can extend the arithmetic intersection number to the certain relatively nef line bundles. Denote by $\widehat{\Pic}_{S,\Q}({U})_{\relnef}^{\nef,\YZ}$ the monoid of relatively nef compactified $S$-metrized $\YZ$-line bundle $\overline{L}=(L,\metr)\in \widehat{\Pic}_{S,\Q}(U)_{\relnef}^\YZ$ satisfying the following property: there is an element $\overline{L'}=(L,\metr')\in \widehat{\Pic}_{S,\Q}(U)^\YZ_\nef$ such that $\overline{L}$, $\overline{L'}$ have the same image in $\widetilde{\Pic}_\Q(U)_\cpt$, and $\metr_{\omega}=\metr'_{\omega}$ for all but finitely many $\omega\in\Omega$.


\begin{art} \label{extension of YZ intersection number}
	Let $\overline{L_j}=({L}_j,\metr_j)\in \widehat{\Pic}_{S,\Q}({U})_{\relnef}^{\nef,\YZ}$ for $j=0,\dots, d$. Let $\overline{L_j'}=({L}_j,\metr_j')\in \widehat{\Pic}_{S,\Q}({U})_{\nef}^\YZ$ satisfying the following conditions: for any $j=0,\dots, d$,
	\begin{enumerate}[resume,leftmargin=*,label=\it(\alph*),ref=\it{(\alph*)}]
		\item\label{extension of YZ intersection number 1} $\overline{L_j}$, $\overline{L_j'}$ have the same image in $\widetilde{\Pic}_\Q(U)_\integrable$, and $\metr_{j,\omega}= \metr_{j,\omega}'$ for all but finitely many $\omega\in\Omega$;
		\item\label{extension of YZ intersection number 2} there are $C_j\colon \Omega\to \R$ such that $C_j(\omega)=0$ for all but finitely many $\omega\in\Omega$ and $-\log\frac{\metr_{j,\omega}}{\metr_{j,\omega}'}\leq C_j(\omega)$ for any $\omega\in\Omega$.
	\end{enumerate}  Notice that such $\overline{L_j'}$ always exists after shrinking $U$. Indeed, by definition of $\widehat{\Pic}_{S,\Q}({U})_{\relnef}^{\nef,\YZ}$, we can find $\overline{L_j'}=(L_j,\metr_j')\in \widehat{\Pic}_{S,\Q}({U})_{\nef}^\YZ$ such that $\metr_{j,\omega}=\metr_{j,\omega}'$ for all but finitely many $\omega\in\Omega$ and \ref{extension of YZ intersection number 1} holds, then after shrinking $U$ and replacing $\metr_{j,\omega}'$ by $\min\{\metr_{j,\omega},\metr_{j,\omega}'\}$ (\cite[Lemma~10.3]{cai2024abstract}), we find such $\overline{L_j'}$.
	For any $\omega\in \Omega$, we set 
	\[E(\overline{\mathbf{L}_\omega'},\overline{\mathbf{L}_\omega})\coloneq \sum\limits_{j=0}^d\int_{U_\omega^\an}-\log\frac{\metr_{j,\omega}}{\metr_{j,\omega}'} \, c_1(\overline{L_{0,\omega}})\wedge\cdots\wedge c_1(\overline{L_{j-1,\omega}})\wedge c_1(\overline{L_{j+1,\omega}'})\wedge \cdots\wedge c_1(\overline{L_{d,\omega}'}),\]
	\begin{align}\label{eq:extension of YZ intersection number}
E(\overline{\mathbf{L}'},\overline{\mathbf{L}})\coloneq \int_{\Omega} E(\overline{\mathbf{L}_\omega'},\overline{\mathbf{L}_\omega}) \, \nu(d\omega)
	\end{align}
	and define
	\begin{align*}(\overline{L_0}\cdots\overline{L_d}\mid U)_S\coloneq (\overline{L_0'}\cdots\overline{L_d'}\mid U)_S+E(\overline{\mathbf{L}'},\overline{\mathbf{L}}).
	\end{align*}
	Notice that \eqref{eq:extension of YZ intersection number} is a finite sum. The value $(\overline{L_0}\cdots\overline{L_d}\mid U)_S$ is in $\R\cup\{-\infty\}$ and it is finite if $\log\frac{\metr_{j,\omega}}{\metr_{j,\omega}'}$ is bounded above for any $j$. It is shown in \cite[Proposition~10.10, Theorem~11.2]{cai2024abstract} that $(\overline{L_0}\cdots\overline{L_d}\mid U)_S$ is independent of the choices of $\overline{L_j'}$. Hence, we have a symmetric multilinear map
	\[\underbrace{\widehat{\Pic}_{S,\Q}({U})^{\nef,\YZ}_\relnef\times\cdots\times\widehat{\Pic}_{S,\Q}({U})^{\nef,\YZ}_\relnef}_{(d+1)\text{-times}}\to \R\cup\{-\infty\}, \ \ (\overline{L_0},\dots, \overline{L_d}) \mapsto (\overline{L_0}\cdots \overline{L_d}\mid U)_S.\]
\end{art}

The following lemma can be proved similarly as \cite[Lemma~6.3.3]{biswas2024concave}, 
see also \cite[Lemma~10.6]{cai2024abstract}. We omit the proof.

\begin{lemma} 
	\label{lemma:existence of sequence of intersections to nef intersection}
	Let $\overline{L_0}=(L_0,\metr_0), \dots, \overline{L_d}=(L_d,\metr_d)\in\widehat{\Pic}_{S,\Q}({U})_{\relsnef}^\YZ$ and $\overline{L_0'}=({L}_0,\metr_0'), \dots, \overline{L_d'}=({L}_d,\metr_d')\in \widehat{\Pic}_{S,\Q}({U})_{\arsnef}^\YZ$. Assume that for any $j=0,\dots, d$,
	\begin{itemize}
		\item  there is a boundary divisor $\overline{B}\in N_{S,\Q}(U)_\mo$ and a sequence in $R_{S,\Q}(U)_\mo$ (resp. $Q_{S,\Q}(U)_\mo$) converging to $\overline{L_j}$ (resp. $\overline{L_j'}$);
		\item $\overline{L_j}$, $\overline{L_j'}$ have the same image in $\widetilde{\Pic}_{\Q}(U)_\integrable$, and $\metr_{j,\omega}= \metr_{j,\omega}'$ for all but finitely many $\omega\in\Omega$.
	\end{itemize}
	For any $m\in\N_{\geq 1}$, set 
	\[\metr_{j,m}\coloneq\min\{\metr_{j},e^m\metr_{j}'\}.\] 
	Then the following statements hold.
	\begin{enumerate1}
		\item \label{nefness for constructed divisors} For each $j=0,\dots, d$, we have $\overline{L_{j,m}}\coloneq({L}_j,\metr_{j,m})\in \widehat{\Pic}_{S,\Q}(U)^\YZ_\relsnef$.
		\item \label{convergences for constructed divisors} For each $j=0,\dots, d$, the sequence $(\overline{L_{j,m}})_{m\in\N_{\geq 1}}$ converges increasingly to $\overline{L_j}$ with respect to the $\overline{B}$-boundary topology for some weak boundary divisor $\overline{B}\in N_{S,\Q}(U)_\mo$.
		\item \label{convergences of intersection number for constructed divisors} We have 
		\[\lim\limits_{m\to\infty}(\overline{L_{0,m}}\cdots\overline{L_{d,m}}\mid U)_S=(\overline{L_0}\cdots\overline{L_d}\mid U)_S.\]
	\end{enumerate1}
\end{lemma}

\section{Generic curves}
\label{sec:generic curves}
In this section, we keep the notation in \cref{subsection:compactified S-metrized YZ-line bundles}. 
Following Gubler's generic curve technique \cite[3.11, 3.12]{gubler2008equidistribution}, 
we give a complete formulation of this method in the language of adelic curves. 
This technique can be used to extend, to some extent, results for curves to higher-dimensional cases.  

\subsection{The generic curve construction}

\begin{lemma}
	\label{lemma:algebraic field extension}
	Let $T_1$ be the normalization of some irreducible component of 
	$T\times_{\Spec(k)}\Spec(\overline{k})$, and let $\mathbf{c}_1$ be the pull-back of 
	$\mathbf{c}$ to $T_1$. Set $K_1\coloneq\overline{k}(T_1)$, and let 
	$S_1=(K_1,\Omega_1, \mathcal{A}_1,\nu_1)$ be the adelic curve associated to $T_1$ 
	and $\mathbf{c}_1$ as in \cref{adelic curve defined by projective variety} (with $r=1$). 
	Then the canonical morphism $T_1\to T$ induces a covering $S_1\to S$ of adelic curves.
	
	Moreover, let $K'/K$ be a finite extension and let $K'_1$ be the composite field of 
	$K'$ and $K_1$ in an algebraic closure $\overline{K_1}$ of $K_1$. 
	Let $S_{K'}$ (resp.~$S_{1,K'_1}$) denote the extension of $S$ over $K'/K$ 
	(resp.~of $S_1$ over $K'_1/K_1$) as in \cref{def:extension of adelic curve}. 
	Then the induced morphism $S_{1,K'_1}\to S_{K'}$ 
	(see \cref{rmk:morphism of extensions}) is also a covering of adelic curves.
\end{lemma}
\begin{proof}
Let $k^{\mathrm{sep}}$ be the separable closure of $k$, then $k^{\mathrm{sep}}/k$ is Galois, $\overline{k}/k^{\mathrm{sep}}$ is purely inseparable. The existence of covering $S_1\to S$ is given by \cref{prop:adelic curve of field extension}~\ref{adelic curve:base field extension}.
	
	For the second statement, let $K'/K$ be a finite extension. 
	We have the following commutative diagram
	\[
	\xymatrix{
		S_{1,K_1'} \ar[r]^{\sim} \ar[d] & S_{K_1'} \ar[d] \ar[r] & S_{K'} \ar[d] \\
		S_1 \ar[r]^{\sim} & S_{K_1} \ar[r] & S,
	}
	\]
	where $S_{K_1}$ and $S_{K_1'}$ denote the extensions of $S$ over $K_1/K$ and $K_1'/K$, 
	respectively. The right-hand square is the canonical covering from 
	\cref{def:extension of adelic curve}; the lower horizontal isomorphism comes from 
	\cref{prop:adelic curve of field extension}~\ref{adelic curve:base field extension}, 
	and the upper one from \cref{rmk:morphism of extensions}.
	
	Since $K'\subset K_1'$, the upper right horizontal arrow $S_{K_1'}\to S_{K'}$ is a covering 
	(induced by the field inclusion). As $S_{1,K_1'}\to S_{K_1'}$ is an isomorphism, the 
	composition $S_{1,K_1'}\xrightarrow{\sim} S_{K_1'}\to S_{K'}$ is a covering. 
	This completes the proof.
\end{proof}

\begin{art} \label{generic curves}
	Assume that $\delta=\dim(T)\geq 2$. 
	Let $n$ be the smallest positive integer such that $\mathbf{c}^{\otimes n}$ is very ample. 
	Let $T_1$ be the normalization of some irreducible component of 
	$T\times_{\Spec(k)}\Spec(\overline{k})$. Choose a basis 
	$t_0,\dots, t_N$ of $H^0(T,\mathbf{c}^{\otimes n})$. Via pull-back to $T_1$, 
	they define a closed embedding $T_1\hookrightarrow \P_{\overline{k}}^N$.
	
	Consider a collection of indeterminates
	\[
	\xi=(\xi_j^{(i)})_{\substack{1\leq i\leq \delta-1\\ 0\leq j\leq N}}
	\]
	which are algebraically independent over $\overline{k}$. Let $\eta$ be the vector 
	obtained from $\xi$ by deleting the entries $\xi_{0}^{(i)}$ for $i=1,\dots, \delta-1$. 
	Set
	\[
	k'\coloneq \overline{k}(\eta), \qquad k''\coloneq \overline{k}(\xi),
	\]
	and define
	\[
	T'\coloneq T_1\times_{\Spec(\overline{k})}\Spec(k'), \qquad 
	T''\coloneq T_1\times_{\Spec(\overline{k})}\Spec(k'').
	\]
	Then $T'$ and $T''$ are geometrically integral projective varieties which are 
	geometrically regular in codimension $1$.
	
	The \emph{generic curve} $T_{\mathbf{c}}$ is the complete intersection on $T''$ 
	defined by
	\[
	T_{\mathbf{c}}\coloneq 
	\mathrm{div}\Bigl(\sum_{j=0}^N \xi_j^{(1)}t_j\Bigr)\cdot\;\cdots\;\cdot 
	\mathrm{div}\Bigl(\sum_{j=0}^N \xi_j^{(\delta-1)}t_j\Bigr) \cdot T'',
	\]
	where each $\sum_{j=0}^N \xi_j^{(i)}t_j$ is viewed as a global section of the 
	pull-back of $\mathbf{c}^{\otimes n}$ to $T''$, and the product denotes the intersection of 
	the corresponding Cartier divisors on $T''$.
	By \cite[\S VIII.6~Proposition~13]{lang1958introduction}, $T_{\mathbf{c}}$ is a 
	geometrically irreducible smooth projective curve over $k''$. Moreover, we have 
	a natural chain of morphisms
	\begin{align}\label{eq:generic curve}
		T_{\mathbf{c}}\hookrightarrow T''\longrightarrow T_1\longrightarrow T,
	\end{align}
	and an equality of function fields
	\[
	k'(T')=k''(T_{\mathbf{c}}).
	\]
	Indeed, on a non-empty open subset of $T_{\mathbf{c}}$ where $t_0\neq 0$, the 
	defining equations allow us to express each $\xi_0^{(i)}$ rationally in terms of 
	$\xi_j^{(i)}$ ($j\geq 1$) and the functions $t_j/t_0\in \overline{k}(T_1)\subset k'(T')$; 
	hence $k''\subset k'(T')$. Since $T_{\mathbf{c}}\to T_1$ is dominant, we also have 
	$\overline{k}(T_1)\subset k''(T_{\mathbf{c}})$ and therefore $k'(T')\subset k''(T_{\mathbf{c}})$, 
	which yields the equality.
	
	Finally, let $S_{\mathbf{c}}=(K_{\mathbf{c}}, \Omega_{\mathbf{c}},\mathcal{A}_{\mathbf{c}},\nu_{\mathbf{c}})$ 
	denote the adelic curve associated to $T_{\mathbf{c}}$ as in 
	\cref{adelic curve defined by projective variety} with the normalizing constant 
	$r=1/n^{\delta-1}$, so that the construction is normalized with respect to the original line 
	bundle $\mathbf{c}$.
\end{art}

\begin{theorem} \label{prop:B is covered by a curve as adelic curves}
	Assume that $\delta=\dim(T)\geq 2$. 
	Let $T_{\mathbf{c}}$ and $S_{\mathbf{c}}=(K_{\mathbf{c}}, \Omega_{\mathbf{c}},\mathcal{A}_{\mathbf{c}},\nu_{\mathbf{c}})$ 
	be as in \cref{generic curves}. 
	Then the morphism of schemes $T_{\mathbf{c}}\to T$ induces a covering of adelic curves
	\[S_{{\mathbf{c}}}\longrightarrow S\]
	in the sense of \cref{def:morphismofadeliccurves}. 
	
	Moreover, for any finite field extension $K'/K$, let $K_{\mathbf{c}}'$ be the composite field of 
	$K'$ and $K_{\mathbf{c}}$ in $\overline{K_{\mathbf{c}}}$. Let $S_{K'}$ \textup{(}resp.~$S_{\mathbf{c},K'_{\mathbf{c}}}$\textup{)} 
	be the extension of $S$ \textup{(}resp.~$S_{\mathbf{c}}$\textup{)} over $K'/K$ \textup{(}resp.~$K_{\mathbf{c}}'/K_{\mathbf{c}}$\textup{)} 
	as in \cref{def:extension of adelic curve}. Then the induced morphism 
	$S_{\mathbf{c},K'_{\mathbf{c}}}\to S_{K'}$ \textup{(}see \cref{rmk:morphism of extensions}\textup{)} is also a 
	covering of adelic curves.
\end{theorem}

\begin{proof}
		By \cref{lemma:algebraic field extension}, we may assume that $k$ is algebraically closed. 
	Since $\mathbf{c}$ is ample, let $n$ be the smallest positive integer such that 
	$\mathbf{c}^{\otimes n}$ is very ample. By \cref{generic curves}, the generic curve 
	$T_{\mathbf{c}}$ and the adelic curve $S_{\mathbf{c}}$ are constructed from 
	$\mathbf{c}^{\otimes n}$ with the normalizing constant $r=1/n^{\delta-1}$. 
	Thus, upon replacing $\mathbf{c}$ by $\mathbf{c}^{\otimes n}$ (which does not change 
	the function field $K$), we may assume in the sequel that $\mathbf{c}$ is very ample 
	and that $n=1$, $r=1$.
	
	We keep the notation of \cref{generic curves} and proceed as follows.
	
	Let $K'/K$ be a finite field extension and let $K_{\mathbf{c}}'$ be the composite field of 
	$K'$ and $K_{\mathbf{c}}$ in $\overline{K_{\mathbf{c}}}$. Let $T_{K'}$ denote the normalization 
	of $T$ in $K'$, and set $T''_{K'}\coloneq T_{K'}\times_{\Spec(k)}\Spec(k'')$. 
	Consider the complete intersection on $T''_{K'}$
	\[
	T_{\mathbf{c},K'}\coloneq 
	\mathrm{div}\Bigl(\sum_{j=0}^N \xi_j^{(1)}t_j\Bigr)\cdot\;\cdots\;\cdot 
	\mathrm{div}\Bigl(\sum_{j=0}^N \xi_j^{(\delta-1)}t_j\Bigr)\cdot T''_{K'},
	\]
	where each $\mathrm{div}(\sum_{j=0}^N \xi_j^{(i)}t_j)$ is regarded as a divisor on $T''_{K'}$ 
	via pull-back from $T''$. By \cite[\S VIII.6~Proposition~13]{lang1958introduction}, 
	$T_{\mathbf{c},K'}$ is a geometrically irreducible smooth projective curve over $k''$, 
	and $k''(T_{\mathbf{c},K'})=K_{\mathbf{c}}'$.
	
	Let $\varphi\colon T_{\mathbf{c},K'}\to T_{K'}''\to T_{K'}$ be the natural morphism. 
	Denote by $S_1=(K',\Omega_1,\mathcal{A}_1,\nu_1)$ \textup{(}resp.~$S_2=(K'_{\mathbf{c}},\Omega_2,\mathcal{A}_2,\nu_2)$\textup{)} 
	the adelic curve associated to $(T_{K'}, \mathbf{c},\dots, \mathbf{c}, [K':K]^{-1})$ 
	\textup{(}resp.~to $(T_{\mathbf{c},K'}, [K_{\mathbf{c}}': K_{\mathbf{c}}]^{-1})$\textup{)} as in 
	\cref{adelic curve defined by projective variety}; here we view $\mathbf{c}$ as an ample 
	line bundle on $T_{K'}$ via pull-back. Notice that
	\[
	[K':K]=[K_{\mathbf{c}}': K_{\mathbf{c}}]
	\]
	since $K_{\mathbf{c}}\cap K'=K$. We claim that $\varphi$ induces a covering $S_2\to S_1$. 
	
	If the claim holds, then setting $K'=K$ shows in particular that $T_{\mathbf{c}}\to T$ induces 
	a covering $S_{\mathbf{c}}\to S$. Moreover, by 
	\cref{prop:adelic curve of field extension}~\ref{adelic curve:finite field extension}, 
	we have canonical isomorphisms $S_{K'}\simeq S_1$ and $S_{\mathbf{c}, K_{\mathbf{c}}'}\simeq S_2$, 
	under which the morphism $S_{\mathbf{c}, K_{\mathbf{c}}'}\to S_{K'}$ corresponds to $S_2\to S_1$. 
	Hence the theorem follows from the claim.
	
	We now prove the claim. Let $\alpha^\#\colon K'\to K_{\mathbf{c}}'$ be the canonical inclusion. 
	We first show that the scheme-theoretic morphism $\varphi\colon T_{\mathbf{c},K'}\to T_{K'}$ 
	maps $\Omega_2=T_{\mathbf{c},K'}^{(1)}$ into $\Omega_1=T_{K'}^{(1)}$. 

	Indeed, let $x\in\Omega_2$ be a closed point. If $\varphi(x)$ were the generic point of 
	$T_{K'}$, then $K'$ and $k''$ would both be contained in the residue field $k''(x)$. 
	But $k''(x)/k''$ is a finite extension, so every element of $K'$ would be algebraic over 
	$k''$ in $k''(T_{K'}'')=K'\otimes_k k''$. On the other hand, since 
	$\delta=\dim(T)=\dim(T_{K'})=\dim(T_{K'}'')\geq 2$, there exist elements of $K'$ which are 
	transcendental over $k''$, a contradiction. Hence $\varphi(x)\in T_{K'}^{(1)}=\Omega_1$. 
	
	Consequently, setting $\alpha_\#\coloneq \varphi\colon \Omega_2\to \Omega_1$ is well-defined. 
Moreover, let $\omega\in\Omega_1$ and let $Y$ be the closure of $\omega$ in $T_{K'}$. 
Since $\varphi$ is dominant and $Y\subsetneq T_{K'}$, the pre-image 
$\varphi^{-1}(Y)=T_{\mathbf{c},K'}\times_{T_{K'}}Y$ is a proper closed subset of the curve 
$T_{\mathbf{c},K'}$. As shown above, every point of $\varphi^{-1}(Y)$ is a closed point of 
$T_{\mathbf{c},K'}$, i.e. lies in $\Omega_2$. Hence $\varphi^{-1}(Y)$ is a proper closed 
subset of a curve, and therefore finite. Thus $\alpha_\#^{-1}(\omega)$ is finite.
	
	By \cite[\S VIII.6~Proposition~12]{lang1958introduction}, 
	$T_{\mathbf{c},K'}\times_{T_{K'}}Y$ is an integral subscheme of $T_{\mathbf{c},K'}$; 
	we denote by $x\in\Omega_2=T_{\mathbf{c},K'}^{(1)}$ its generic point. 
	Since $T_{K'}$ is regular in codimension $1$, the local ring 
	$\mathcal{O}_{T_{K'},\omega}$ is a discrete valuation ring. Let $\mathfrak{m}_\omega$ 
	be its maximal ideal and let $f\in\mathcal{O}_{T_{K'},\omega}$ be a uniformizer. 
	The pullback of $f$ to $T_{\mathbf{c},K'}$ cuts out the closed subscheme 
	$\varphi^{-1}(Y)=T_{\mathbf{c},K'}\times_{T_{K'}}Y$. Because this subscheme is integral 
	and $x$ is its generic point, the quotient 
	$\mathcal{O}_{T_{\mathbf{c},K'},x}/(f)$ is a $0$-dimensional integral local ring, 
	hence a field. Thus the ideal generated by the pullback of $f$ in 
	$\mathcal{O}_{T_{\mathbf{c},K'},x}$ equals the maximal ideal $\mathfrak{m}_x$. 
	Consequently $\mathfrak{m}_\omega\mathcal{O}_{T_{\mathbf{c},K'},x}=\mathfrak{m}_x$, 
	so the ramification index of $\varphi$ at $x$ equals $1$. Therefore $\val_x$ 
	restricts to $\val_\omega$ on $K'$, and $\alpha=(\alpha^\#,\alpha_\#,1)$ is a morphism of adelic curves.
	
	Furthermore, since the divisors $\mathrm{div}(\sum_{j=0}^N \xi_j^{(i)}t_j)$ are sections of the line bundle $\mathbf{c}''\coloneq \mathbf{c}\times_{\Spec({k})}\Spec(k'')$ on $T_{K'}''$, the projection formula for the base change 
	$T''_{K'}\to T_{K'}$ yields
	\[\deg_{k''}(T_{\mathbf{c},K'}\times_{T_{K'}}Y)=\deg_{k''}((\mathbf{c}'')^{\delta-1}\cdot[Y''])=\deg_k(\mathbf{c}^{\delta-1}\cdot[\omega]),\]
	It follows that
	\[
	\nu_2(x)=[K_{\mathbf{c}}': K_{\mathbf{c}}]^{-1}\deg_{k''}(T_{\mathbf{c},K'}\times_{T_{K'}}Y)
	=[K':K]^{-1}\deg_k(\mathbf{c}^{\delta-1}\cdot[\omega])
	=\nu_1(\omega).
	\]
	Thus $\alpha$ is strong by \cref{rmk:strong between discrete sigma}. We have already shown that $\alpha_\#^{-1}(\omega)$ is finite for every 
	$\omega\in\Omega_1$, and that $\alpha$ is a strong morphism of adelic curves. 
	By \cref{lemma:characterization of covering and isomorphism}~\ref{strong for discrete measurable}, 
	$\alpha$ is a covering of adelic curves. This completes the proof.
\end{proof}

\begin{remark}
	Keeping the notation in \cref{prop:B is covered by a curve as adelic curves}, write 
	$S_{\mathbf{c},K'_{\mathbf{c}}}=(K'_{\mathbf{c}}, \Omega_{\mathbf{c},K'_{\mathbf{c}}}, \mathcal{A}_{\mathbf{c},K'_{\mathbf{c}}}, \nu_{\mathbf{c},K'_{\mathbf{c}}})$ 
	and $S_{K'}=(K',\Omega_{K'},\mathcal{A}_{K'},\nu_{K'})$. If the base field $k$ is algebraically closed, 
	the proof above shows that the covering $S_{\mathbf{c},K'_{\mathbf{c}}}\to S_{K'}$ induces a bijection 
	$\alpha_\#\colon\Omega_{\mathbf{c},K'_{\mathbf{c}}}\to\Omega_{K'}$ satisfying 
	$\nu_{\mathbf{c},K'_{\mathbf{c}}}(x)=\nu_{K'}(\alpha_\#(x))$ for all $x\in\Omega_{\mathbf{c},K'_{\mathbf{c}}}$. 
	Hence it induces an isomorphism of discrete measure spaces
	\[
	(\Omega_{\mathbf{c},K'_{\mathbf{c}}},\mathcal{A}_{\mathbf{c},K'_{\mathbf{c}}},\nu_{\mathbf{c},K'_{\mathbf{c}}})
	\;\xrightarrow{\;\sim\;}\;
	(\Omega_{K'},\mathcal{A}_{K'},\nu_{K'}).
	\]
\end{remark}

\subsection{Invariance of intersection numbers}

Recall the arithmetic intersection number on $U$ defined in 
\cref{def:boundarytopologyglobal} and in \cref{extension of YZ intersection number}.

\begin{corollary} \label{prop:intersection number after base changes}
	Assume that $\delta=\dim(T)\geq 2$. 
	Let $S_{\mathbf{c}}=(K_{\mathbf{c}}, \Omega_{\mathbf{c}},\mathcal{A}_{\mathbf{c}},\nu_{\mathbf{c}})$ 
	and $U_{\mathbf{c}}=U\times_{\Spec(K)}\Spec(K_{\mathbf{c}})$ be as in \cref{generic curves}. 
	Then the pull-back induces a map
	\begin{align} \label{eq:pull-back via covering}
		\widehat{\Pic}_{S,\Q}(U)_\cpt\longrightarrow \widehat{\Pic}_{S_{\mathbf{c}},\Q}(U_{\mathbf{c}})_\cpt,
		\qquad \overline{L}\longmapsto \overline{L_{\mathbf{c}}},
	\end{align}
	which sends $\widehat{\Pic}_{S,\Q}(U)_{\integrable}^\YZ$ 
	\textup{(}resp.\ $\widehat{\Pic}_{S,\Q}(U)_{\nef}^{\YZ}$, 
	resp.\ $\widehat{\Pic}_{S,\Q}(U)_{\relnef}^{\nef,\YZ}$\textup{)} 
	to $\widehat{\Pic}_{S_{\mathbf{c}},\Q}(U_{\mathbf{c}})_{\integrable}^\YZ$ 
	\textup{(}resp.\ $\widehat{\Pic}_{S_{\mathbf{c}},\Q}(U_{\mathbf{c}})_{\nef}^\YZ$, 
	resp.\ $\widehat{\Pic}_{S_{\mathbf{c}},\Q}(U_{\mathbf{c}})_{\relnef}^{\nef,\YZ}$\textup{)}. 
	Moreover, for any $t$-dimensional subvariety $Y\subset U_{\overline{K}}$ and any 
	$\overline{L_0},\dots,\overline{L_t}\in\widehat{\Pic}_{S,\Q}(U)_{\integrable}^\YZ$ 
	\textup{(}resp.\ $\in\widehat{\Pic}_{S,\Q}(U)_{\relnef}^{\nef,\YZ}$\textup{)}, we have
	\begin{align}\label{eq:intersection number is preserved}
		(\overline{L_0}\cdots\overline{L_t}\mid Y)_{S} 
		= (\overline{L_{0,\mathbf{c}}}\cdots\overline{L_{t,\mathbf{c}}}\mid Y_{\mathbf{c}})_{S_{\mathbf{c}}},
	\end{align}
	where $Y_{\mathbf{c}} = Y\times_{\Spec(\overline{K})}\Spec(\overline{K_{\mathbf{c}}})$.
\end{corollary}

\begin{proof}
	For the projective case, the existence of a pull-back preserving the relevant 
	subspaces is proved in \cite[3.12]{gubler2008equidistribution}; it also follows from 
	\cref{prop:B is covered by a curve as adelic curves} together with 
	\cite[Theorem~4.3.6]{chen2021arithmetic}. For the quasi-projective case, observe that 
	the pull-back of a weak boundary divisor of $U$ (see \cref{global:weakly boundary}) is a 
	weak boundary divisor of $U_{\mathbf{c}}$, and base change preserves effectivity. Hence 
	the pull-back map \eqref{eq:pull-back via covering} is well-defined. Moreover, if 
	$\overline{L}\in P_{S,\Q}(U)_\mo$ is induced by a line bundle $\mathcal{L}$ on a projective 
	scheme $\mathcal{X}$ over $T$, then $\overline{L_{\mathbf{c}}}$ is induced by the pull-back 
	of $\mathcal{L}$ to $\mathcal{X}\times_T T_{\mathbf{c}}$. Thus 
	$\overline{L_{\mathbf{c}}}\in P_{S_{\mathbf{c}},\Q}(U_{\mathbf{c}})_\mo$.
	
	We claim that the pull-back \eqref{eq:pull-back via covering} sends $Q_{S,\Q}(U)_\mo$ to $Q_{S_{\mathbf{c}},\Q}(U_{\mathbf{c}})_\mo$ 
	and $R_{S,\Q}(U)_\mo$ to $R_{S_{\mathbf{c}},\Q}(U_{\mathbf{c}})_\mo$. Let $\mathcal{X}$ be a 
	projective model of $U$ over $T$ and let $\mathcal{L}$ be a model line bundle on 
	$\mathcal{X}$ inducing $\overline{L}$. Since the nef cone is the closure of the ample cone, 
	we may assume that $\mathcal{L}$ is ample. Consider the tower of base changes
	\[
	\mathcal{X}_{\mathbf{c}}\longrightarrow\mathcal{X}''\longrightarrow\mathcal{X}_1
	\longrightarrow\mathcal{X}
	\]
	induced by $T_{\mathbf{c}}\to T''\to T_1\to T$, and denote by $\mathcal{L}_1$, 
	$\mathcal{L}''$, $\mathcal{L}_{\mathbf{c}}$ the respective pull-backs of $\mathcal{L}$. 
	Because $\mathcal{X}_1\to\mathcal{X}$ factors through 
	$\mathcal{X}\times_{\Spec(k)}\Spec(\overline{k})$, the line bundle $\mathcal{L}_1$ is nef. 
	To see that $\mathcal{L}_{\mathbf{c}}$ is nef, it suffices to show that $\mathcal{L}''$ is 
	nef. Take an ample line bundle $\mathcal{A}_1$ on $\mathcal{X}_1$. For every $n\geq 1$, the 
	$\Q$-line bundle $\mathcal{L}_1+\frac{1}{n}\mathcal{A}_1$ is ample on $\mathcal{X}_1$, hence 
	its pull-back $\mathcal{L}''+\frac{1}{n}\mathcal{A}''$ is ample on 
	$\mathcal{X}''=\mathcal{X}_1\times_{\Spec(k)}\Spec(k'')$. Passing to the limit as 
	$n\to\infty$ shows that $\mathcal{L}''$ is nef, and therefore so is 
	$\mathcal{L}_{\mathbf{c}}$. This proves the claim for $Q_{S,\Q}(U)_\mo$. Now assume $\overline{L}\in R_{S,\Q}(U)_\mo$. Then $L$ is nef and, for every 
	$\omega\in\Omega=T^{(1)}$, the restriction $\mathcal{L}_\omega$ is nef on the special 
	fiber of $\mathcal{X}\times_T\Spec(\mathcal{O}_{T,\omega})\to\Spec(\mathcal{O}_{T,\omega})$. 
	For any $\omega'\in\Omega_{\mathbf{c}}=T_{\mathbf{c}}^{(1)}$, let $\omega\in\Omega$ be its 
	image under $T_{\mathbf{c}}\to T$. Then the special fiber of 
	$\mathcal{X}_{\mathbf{c}}\times_{T_{\mathbf{c}}}\Spec(\mathcal{O}_{T_{\mathbf{c}},\omega'})\to
	\Spec(\mathcal{O}_{T_{\mathbf{c}},\omega'})$ is a base change of the special fiber of 
	$\mathcal{X}\times_T\Spec(\mathcal{O}_{T,\omega})\to\Spec(\mathcal{O}_{T,\omega})$. Hence 
	the pull-back of $\mathcal{L}$ remains nef on the special fibers of 
	$\mathcal{X}_{\mathbf{c}}$ over $T_{\mathbf{c}}$, and 
	$\overline{L_{\mathbf{c}}}\in R_{S_{\mathbf{c}},\Q}(U_{\mathbf{c}})_\mo$. Taking completions, 
	we conclude that the map \eqref{eq:pull-back via covering} sends 
	$\widehat{\Pic}_{S,\Q}(U)_{\integrable}^\YZ$ 
	\textup{(}resp.\ $\widehat{\Pic}_{S,\Q}(U)_{\nef}^{\YZ}$, 
	resp.\ $\widehat{\Pic}_{S,\Q}(U)_{\relnef}^{\nef,\YZ}$\textup{)} to the corresponding 
	subspace on the generic-curve side.
	
It remains to show the invariance of intersection numbers. By 
\cref{prop:B is covered by a curve as adelic curves}, replacing $K$ with a finite 
extension preserves the covering property; hence we may assume that $Y$ is defined over 
$K$. Replacing $U$ by $Y$ and the $\overline{L_j}$ by their restrictions to $Y$, we may 
further assume that $Y=U$ (so $t=d=\dim(U)$). Let 
$\overline{L_0},\dots,\overline{L_d}\in\widehat{\Pic}_{S,\Q}(U)_{\integrable}^\YZ$. After 
shrinking $U$, \eqref{eq:intersection number is preserved} follows from the projective 
case together with \cite[Proposition~7.22~(i)~(v)]{cai2024abstract}.
	
	Finally, let $\overline{L_0},\dots,\overline{L_d}\in\widehat{\Pic}_{S,\Q}(U)^{\integrable,\YZ}_\relnef$. 
	By \cref{lemma:existence of sequence of intersections to nef intersection}, after shrinking 
	$U$ there exist sequences 
	$(\overline{L_{j,m}})_{m\geq 1}\subset\widehat{\Pic}_{S,\Q}(U)^\YZ_{\integrable}$ converging 
	to $\overline{L_j}$ in the $\overline{B}$-boundary topology for some weak boundary divisor 
	$\overline{B}$, and similarly sequences $(\overline{L_{j,\mathbf{c},m}})_{m\geq 1}$ on the 
	generic-curve side converging to $\overline{L_{j,\mathbf{c}}}$. The construction in 
	\cref{lemma:existence of sequence of intersections to nef intersection} is functorial under 
	base change, so $\overline{L_{j,m,\mathbf{c}}}=\overline{L_{j,\mathbf{c},m}}$ for all $j,m$. 
	Consequently,
	\begin{align*}
		(\overline{L_0}\cdots\overline{L_d}\mid U)_S
		&=\lim_{m\to\infty}(\overline{L_{0,m}}\cdots\overline{L_{d,m}}\mid U)_S \\
		&=\lim_{m\to\infty}(\overline{L_{0,\mathbf{c},m}}\cdots\overline{L_{d,\mathbf{c},m}}\mid U_{\mathbf{c}})_{S_{\mathbf{c}}} \\
		&=(\overline{L_{0,\mathbf{c}}}\cdots\overline{L_{d,\mathbf{c}}}\mid U_{\mathbf{c}})_{S_{\mathbf{c}}},
	\end{align*}
	where the second equality uses the already established integrable case. This completes 
	the proof.
\end{proof}

\begin{remark}
	The geometric version of the corollary also holds: one has a pull-back 
	$\widetilde{\Pic}_\Q(U)_\integrable\to\widetilde{\Pic}_\Q(U_{\mathbf{c}})_\integrable$,
	and the geometric intersection numbers are preserved under the base change.
\end{remark}

\section{Equidistribution theorem in relatively nef case}

\label{sec:equidistribution nef}

As before, in this section, we keep the notation in \cref{subsection:compactified S-metrized YZ-line bundles}. We will use the generic curve technique to prove the equidistribution theorem over function fields. We fix a $d$-dimensional quasi-projective variety over $K$.

\subsection{Equidistribution theorem for points}

\begin{art}
	 Let $X$ be an algebraic variety over a field $F$, and $(Y_m)_{m\in I}$ a net of proper closed subvarieties of $X_{\overline{F}}$. We say that $(Y_m)_{m\in I}$ is \emph{generic} (on $X$) if for any proper closed subvariety $Y$ of $X$, there is $m_0\in I$ such that $Y_m$ is not contained in $Y$ for any $m\geq m_0$. 
\end{art}


\begin{lemma}\label{lemma:lift closed subset}
	Let $F'/F$ be a field extension and let $X$ be an algebraic variety over $F$. 
	Fix an embedding 
	$\overline{F}\hookrightarrow\overline{F'}$ of algebraic closure. Let $Y'$ be a proper closed subset of 
	$X_{F'}\coloneq X\times_{\Spec(F)}\Spec(F')$. Then there exists a proper closed subset 
	$Y$ of $X$ such that
	\[
	Y'(\overline{F'})\cap X(\overline{F})\subset Y(\overline{F}),
	\]
	where $X(\overline{F})$ is identified with its image in $X_{F'}(\overline{F'})$.
\end{lemma}
\begin{proof}
	We proceed in three steps.
	
\vspace{2mm}\noindent
\textbf{Step~1.} \textit{Reduction to finitely generated extensions.}
The closed subset $Y'\subset X_{F'}$ is defined by finitely many equations, whose 
coefficients lie in some finitely generated subextension $F_1/F$. Hence there exists 
a proper closed subset $Y'_1\subset X_{F_1}$ such that $Y'$ is its base change to $F'$. 
We have a Cartesian diagram
\[
\xymatrix{
	Y' \ar[r]\ar[d] & Y'_1 \ar[d] \\
	X_{F'} \ar[r] & X_{F_1}
}
\]
induced by the inclusion $F_1\subset F'$. Fix compatible embeddings 
$\overline{F}\hookrightarrow\overline{F_1}\hookrightarrow\overline{F'}$. For any 
$x\in X(\overline{F})$, let $x_{F_1}\in X_{F_1}(\overline{F_1})$ and 
$x_{F'}\in X_{F'}(\overline{F'})$ denote the induced images. By the universal property 
of the fibre product, $x_{F'}\in Y'(\overline{F'})$ if and only if 
$x_{F_1}\in Y'_1(\overline{F_1})$. Under the identifications 
$X(\overline{F})\subset X_{F'}(\overline{F'})$ and 
$X(\overline{F})\subset X_{F_1}(\overline{F_1})$, this means
\[
Y'(\overline{F'})\cap X(\overline{F})
= Y'_1(\overline{F_1})\cap X(\overline{F}).
\]
Consequently, any proper closed subset $Y\subset X$ satisfying the conclusion for 
$Y'_1$ also satisfies it for $Y'$. Thus we may assume that $F'/F$ is finitely generated.
	
	\vspace{2mm}\noindent
	\textbf{Step~2.} \textit{Reduction to the function field of a geometrically 
		integral variety.}
	If $F'/F$ is finite, let $Y$ be the Zariski image of $Y'$ under the projection 
	$X_{F'}\to X$. Then $Y$ is a proper closed subset of $X$, and every point of 
	$Y'(\overline{F'})\cap X(\overline{F})$ maps into $Y(\overline{F})$ under the projection, 
	so the inclusion holds.
	
	For a general finitely generated extension $F'/F$, choose an integral scheme $T'$ of 
	finite type over $F$ with function field $F'$. There exists a finite extension $M/F$ 
	such that some irreducible component $T''$ of $T'\times_{\Spec(F)}\Spec(M)$ is 
	geometrically integral over $M$. Let $F''$ be the function field of $T''$; then 
	$F''/F'$ is finite. Fix compatible embeddings 
	$\overline{F}\hookrightarrow\overline{M}\hookrightarrow\overline{F''}$, so that
	\[
	X(\overline{F})\subset X_M(\overline{M})\subset X_{F''}(\overline{F''}).
	\]
	
	Assuming the lemma holds for the extension $F''/M$, we obtain a proper closed subset 
	$Y_M\subset X_M$ such that
	\[
	Y'_{F''}(\overline{F''})\cap X_M(\overline{M})\subset Y_M(\overline{M}),
	\]
	where $Y'_{F''}$ denotes the base change of $Y'$ to $F''$. Since the lemma also holds 
	for the finite extension $M/F$, there exists a proper closed subset $Y\subset X$ such 
	that
	\[
	Y_M(\overline{M})\cap X(\overline{F})\subset Y(\overline{F}).
	\]
	Now take any $x\in Y'(\overline{F'})\cap X(\overline{F})$. Viewing $x$ as a point of 
	$X_{F''}(\overline{F''})$, we have $x\in Y'_{F''}(\overline{F''})$; hence 
	$x\in Y'_{F''}(\overline{F''})\cap X_M(\overline{M})\subset Y_M(\overline{M})$. 
	Since $x\in X(\overline{F})$, we obtain $x\in Y_M(\overline{M})\cap X(\overline{F})
	\subset Y(\overline{F})$. Thus it suffices to treat the case where $F'$ is the function 
	field of a geometrically integral variety $T'$ over $F$.
	
\vspace{2mm}\noindent
\textbf{Step~3.} \textit{The geometrically integral case.}
Assume that $F'$ is the function 
field of a geometrically integral variety $T'$ over $F$.
We first reduce to the case where $X$ is affine. Let $X=\bigcup_{i=1}^m X_i$ be a finite 
affine open covering. Set $X_i'\coloneq X_i\times_{\Spec(F)}\Spec(F')$ and 
$Y_i'\coloneq Y'\cap X_i'$. Suppose that for each $i$ there exists a proper closed subset 
$Y_i\subset X_i$ such that 
$Y_i'(\overline{F'})\cap X_i(\overline{F})\subset Y_i(\overline{F})$. Let 
$\overline{Y_i}$ denote the closure of $Y_i$ in $X$, and set 
$Y\coloneq\bigcup_{i=1}^m\overline{Y_i}$. Then $Y$ is a proper closed subset of $X$. 
If $x\in Y'(\overline{F'})\cap X(\overline{F})$, then $x$ lies in some 
$X_i(\overline{F})$, hence $x\in Y_i'(\overline{F'})\cap X_i(\overline{F})\subset 
Y_i(\overline{F})\subset Y(\overline{F})$. Thus we may assume that $X\subset\mathbb{A}_F^n$ 
is affine.

Next, we reduce to the case where $Y'$ is defined by a single equation. Since 
$X_{F'}\subset\mathbb{A}_{F'}^n$ is affine, the closed subset $Y'$ is the zero locus of 
finitely many functions $f_1,\dots,f_r\in\mathcal{O}_{X_{F'}}(X_{F'})$. For each $j$, write 
$Z(f_j)\subset X_{F'}$ for the closed subset defined by the vanishing of $f_j$. 
Suppose that for every $j$ there exists a proper closed subset $Y_j\subset X$ such that
\[
Z(f_j)(\overline{F'})\cap X(\overline{F})\subset Y_j(\overline{F}).
\]
Set $Y\coloneq\bigcap_{j=1}^r Y_j$. Then $Y$ is a proper closed subset of $X$ 
(contained in the proper subset $Y_1$). Moreover, if 
$x\in Y'(\overline{F'})\cap X(\overline{F})$, then $f_j(x)=0$ for all $j$, so 
$x\in Z(f_j)(\overline{F'})\cap X(\overline{F})\subset Y_j(\overline{F})$ for every 
$j$; hence $x\in Y(\overline{F})$. Thus we may assume that $Y'=Z(f)$ for a single 
non-zero function $f\in\mathcal{O}_{X_{F'}}(X_{F'})$.

After multiplying $f$ by a suitable element of $\mathcal{O}_{T'}(T')$, we may write
\[
f=\sum_{i_1,\dots,i_n}f_{i_1,\dots,i_n}\,x_1^{i_1}\cdots x_n^{i_n},
\qquad f_{i_1,\dots,i_n}\in\mathcal{O}_{T'}(T').
\]
View $f$ as a rational function on $X_{T'}\coloneq X\times_{\Spec(F)}T'$. Since 
$X_{F'}$ is dense in $X_{T'}$ and $f$ does not vanish identically on $X_{F'}$, it does 
not vanish identically on $X_{T'}$. Hence there exists a separable point 
$t_0\in T'(F^{\mathrm{sep}})$ such that the restriction $f|_{X_{t_0}}$ is non-zero 
(here we use that $T'(F^{\mathrm{sep}})$ is dense in $T'$, see 
\cite[Schemes, Lemma~33.25.3]{stacks-project}). Let $M/F$ be the minimal Galois 
extension contained in $F^{\mathrm{sep}}$ such that $F(t_0)\subset M$, and let 
$G=\Gal(M/F)$. Define
\[
g\coloneq\prod_{\sigma\in G}\Bigl(\sum_{i_1,\dots,i_n}
\sigma\bigl(f_{i_1,\dots,i_n}(t_0)\bigr)x_1^{i_1}\cdots x_n^{i_n}\Bigr)
\in\mathcal{O}_X(X).
\]
Since $f|_{X_{t_0}}\neq 0$, the polynomial $g$ is non-zero, so its zero locus 
$Y\coloneq Z(g)\subset X$ is a proper closed subset.

Now let $x=(a_1,\dots,a_n)\in Y'(\overline{F'})\cap X(\overline{F})$. Because $T'$ 
is geometrically integral, $F'\otimes_F\overline{F}$ is a field; evaluating $f$ 
at $x$ gives
\[
f(x)=\sum_{i_1,\dots,i_n}f_{i_1,\dots,i_n}\,a_1^{i_1}\cdots a_n^{i_n}=0
\quad\text{in }F'\otimes_F\overline{F}.
\]
Specializing at $t_0$ yields 
$\sum_{i_1,\dots,i_n}f_{i_1,\dots,i_n}(t_0)\,a_1^{i_1}\cdots a_n^{i_n}=0$ in 
$\overline{F}$, hence $g(x)=0$, i.e. $x\in Y(\overline{F})$. Therefore
\[
Y'(\overline{F'})\cap X(\overline{F})\subset Y(\overline{F}),
\]
which completes the proof.
\end{proof}

\begin{lemma}\label{lemma:generic stable under base changes}
	Let $F'/F$ be a field extension and let $X$ be an algebraic variety over $F$. 
	Fix an embedding $\overline{F}\hookrightarrow\overline{F'}$ of algebraic closures. 
	Let $(Y_m)_{m\in I}$ be a net of proper closed subvarieties of $X_{\overline{F}}$, , and $Y_m'\coloneq Y_m\times_{\Spec(\overline{F})}\Spec(\overline{F'})$. If $(Y_m)_{m\in I}$ is 
	generic on $X$, then $(Y_m')_{m\in I}$ is generic on $X_{F'}$.
\end{lemma}

\begin{proof}
	Let $Z'\subset X_{F'}$ be a proper closed subset. By \cref{lemma:lift closed subset}, 
	there exists a proper closed subset $Z\subset X$ such that
	\[
	Z'(\overline{F'})\cap X(\overline{F})\subset Z(\overline{F}).
	\]
	Since $(Y_m)_{m\in I}$ is generic on $X$, there exists $m_0\in I$ such that 
	$Y_m\not\subset Z$ for all $m\geq m_0$. We claim that $Y_m'\not\subset Z'$ for all 
	$m\geq m_0$. 
	
	Indeed, fix $m\geq m_0$. Since $Y_m\not\subset Z$, we can choose a point 
	$x\in Y_m(\overline{F})\setminus Z(\overline{F})$. By the inclusion above, the image 
	of $x$ in $X_{F'}(\overline{F'})$ does not lie in $Z'(\overline{F'})$. On the other 
	hand, this image clearly lies in $Y_m'(\overline{F'})$. Hence 
	$Y_m'(\overline{F'})\not\subset Z'(\overline{F'})$, which implies $Y_m'\not\subset Z'$. 
	Since $Z'$ is arbitrary, $(Y_m')_{m\in I}$ is generic on $X_{F'}$.
\end{proof}

Recall that we have a canonical homomorphism
\[
\widehat{\Pic}_{S,\Q}(U)_\cpt^\YZ\longrightarrow\widetilde{\Pic}_{\Q}(U)_\cpt
\]
which sends $\widehat{\Pic}_{S,\Q}(U)_{\integrable}^\YZ$ to 
$\widetilde{\Pic}_{\Q}(U)_\integrable$, see \cref{def:boundarytopologyglobal}. 
For $\widetilde{L}\in\widetilde{\Pic}_{\Q}(U)_\integrable$ and a subvariety 
$Y\subset U_{\overline{K}}$, we have $\deg_{\widetilde{L}}(Y)\in\R$ defined in 
\cref{geometric intersection number}. 

\begin{art}\label{def:small}
	Let $(Y_m)_{m\in I}$ be a net of proper closed subvarieties of $U_{\overline{K}}$, 
	and let $\overline{L}\in\widehat{\Pic}_{S,\Q}(U)^{\nef,\YZ}_\relnef$ be such that 
	$\deg_{\widetilde{L}}(U)>0$ and $\deg_{\widetilde{L}}(Y_m)>0$ for all $m\in I$, 
	where $\widetilde{L}$ denotes the image of $\overline{L}$ in 
	$\widetilde{\Pic}_\Q(U)_{\mathrm{int}}$. We say that $(Y_m)_{m\in I}$ is 
	\emph{small with respect to $\overline{L}$} if
	\[
	\lim_{m\in I}\frac{(\overline{L}^{\dim(Y_m)+1}\mid Y_m)_S}
	{(\dim(Y_m)+1)\deg_{\widetilde{L}}(Y_m)}
	=\frac{(\overline{L}^{d+1}\mid U)_S}{(d+1)\deg_{\widetilde{L}}(U)},
	\]
	where $d=\dim(U)$.
\end{art}

Before we state the main theorem, we fix the following setup.

\begin{art}
	For each $\omega\in\Omega$, we denote by $\C_\omega$ the completion of 
	$\overline{K_\omega}$, and we fix an embedding 
	$\jmath_\omega\colon\overline{K}\hookrightarrow\C_\omega$. For an algebraic variety 
	$X$ over $K$, we denote by $X^\an_{\C_\omega}$ the analytification of 
	$X\times_{\Spec(K)}\Spec(\C_\omega)$. By 
	\cite[Corollary~1.3.6]{berkovich1990spectral}, we have
	\[
	X_\omega^\an=X_{\C_\omega}^\an/\Gal(\overline{K_\omega}/K_\omega).
	\]
	
	Since we will study the equidistribution of Galois orbits of subvarieties defined 
	over $\overline{K}$, it is more convenient to work on $U_{\C_\omega}^\an$ instead of 
	$U_\omega^\an$. The equidistribution theorems over $U_{\C_\omega}^\an$ and over 
	$U_{\omega}^\an$ are equivalent; see \cite[Section~3]{yuan2008big}.
	
	Let $\omega\in\Omega$ and let $\overline{L_1},\dots,\overline{L_d}\in 
	\widehat{\Pic}_{S,\Q}(U)_\relnef^{\nef, \YZ}$. The pull-back of 
	$\overline{L_j}$ to $U_{\C_\omega}$ is denoted by $\overline{L_{j,\C_\omega}}\in 
	\widehat{\Pic}_{\Q}(U_{\C_\omega})_\relnef$, and we can associate a positive Radon 
	measure $c_1(\overline{L_{1,\C_\omega}})\wedge\cdots\wedge 
	c_1(\overline{L_{d,\C_\omega}})$ on $U_{\C_\omega}^\an$. It turns out that the 
	measure $c_1(\overline{L_{1,\omega}})\wedge\cdots\wedge 
	c_1(\overline{L_{d,\omega}})$ on $U_\omega^\an$ is the push-forward of the former 
	under the quotient map, see \cite[Proposition~4.45~(d)]{cai2024abstract}.
	
	For a subvariety $Y\subset U_{\overline{K}}$, we denote by $O(Y)$ the 
	\emph{Galois orbit} of $Y$ under the action of $\Gal(\overline{K}/K)$. For any 
	$\omega\in\Omega$ and any $Y^\sigma\in O(Y)$, we set
	\[
	Y^\sigma_\omega\coloneq 
	Y^\sigma\times_{\Spec(\overline{K}),\jmath_\omega}\Spec(\C_\omega),
	\]
	and we denote by $Y_{\omega}^{\sigma,\an}$ its analytification.
\end{art}

\begin{proposition}\label{prop:intersection on subvariety}
	Let $\overline{L_1},\dots,\overline{L_t}\in\widehat{\Pic}_{S,\Q}(U)_\integrable^\YZ$, 
	and let $\overline{M}=(\OO_U,e^{-f})\in\widehat{\Pic}_{S,\Q}(U)_\integrable^\YZ$ be 
	such that the image of $\overline{M}$ in $\widetilde{\Pic}_\Q(U)_\integrable$ is 
	trivial. Let $Y$ be a $t$-dimensional closed subvariety of $U_{\overline{K}}$. Then
	\begin{multline*}
		(\overline{L_1}\cdots\overline{L_t}\cdot\overline{M}\mid Y)_S
		=\frac{1}{|O(Y)|}\sum_{\omega\in\Omega}\nu(\omega)
		\sum_{Y^\sigma\in O(Y)}
		\int_{Y_\omega^{\sigma,\an}}f_\omega\,
		c_1(\overline{L_{1,\C_\omega}}|_{Y_\omega^\sigma})\wedge\cdots\wedge 
		c_1(\overline{L_{t,\C_\omega}}|_{Y_\omega^\sigma}),
	\end{multline*}
	where $f_\omega$ denotes the $\omega$-component of the Green function $f$.
\end{proposition}
\begin{proof}
	The proof is similar to the one of \cite[Proposition~2.3]{burgos2019the}. By linearity, we may assume that $\overline{L_1}, \dots, \overline{L_t}\in \widehat{\Pic}_{S,\Q}(U)_\nef^\YZ$.
	
	Let $F$ be a finite normal extension $F\subset \overline{K}$ of $K$ such that $Y$ is defined over $F$, and $S_F=(F,\Omega_F, \mathcal{A}_F, \nu_F)$ the extension of $S$ over $F/K$. In the following, we view $Y$ as an algebraic variety over $F$, similarly, we view each $Y^\sigma\in O(Y)$ as an algebraic variety over $F$. For any $\omega\in\Omega$, we denote by $M_{F,\omega}$ the set of elements in $\Omega_F$ above $\omega$. The group $G\coloneq \Gal(F/K)$ acts on $M_{F,\omega}$ transitively, then for any $v\in M_{K,\omega}$,
	\[|G|=|\Gal(F_{v}/K_\omega)|\cdot|M_{F,\omega}|\]
	which implies that
	\begin{align}\label{eq:action 1}
[F:K]_s=[F_{v}:K_\omega]_s\cdot|M_{F,\omega}|
	\end{align}
since $[F:K]_s=|G|$ and $[F_{v}:K_\omega]_s=|\Gal(F_{v}/K_\omega)|$. Moreover, $g\in G$ induces an isomorphism $g\colon F_v\simeq F_{g(v)}$ over $K_\omega$.
	
	Let $\overline{M}_1, \overline{M}_2\in \widehat{\Pic}_{S,\Q}(U)^\YZ_{\nef}$ such that $\overline{M}=\overline{M_1}-\overline{M_2}$. Since $\overline{M_1}, \overline{M}_2$ have the same image in $\widetilde{\Pic}_\Q(U)_\integrable$, By \cite[Theorem~11.2]{cai2024abstract} (the singularity condition on Green functions can be removed if we take a third Green function less singular than both), we have 
	\begin{align*}
		(\overline{L_1}\cdots \overline{L_t}\cdot\overline{M}\mid Y)_S=& (\overline{L_1}\cdots \overline{L_t}\cdot\overline{M_1}\mid Y)_S - (\overline{L_1}\cdots \overline{L_t}\cdot\overline{M_2}\mid Y)_S\\
		=&\sum\limits_{v\in \Omega_{F}}\nu_F(v)\int_{Y_{v}^\an}f_\omega\, c_1(\overline{L_{1,\omega}}|_{Y_{v}})\wedge\cdots\wedge c_1(\overline{L_{t,\omega}}|_{Y_{v}})\\
		=&\sum\limits_{\omega\in \Omega}\sum\limits_{v\in M_{F,\omega}}\nu_F(v)\int_{Y_{v}^\an}f_\omega\, c_1(\overline{L_{1,\omega}}|_{Y_{v}})\wedge\cdots\wedge c_1(\overline{L_{t,\omega}}|_{Y_{v}})\\
		=&\sum\limits_{\omega\in \Omega}\nu(\omega)\sum\limits_{v\in M_{F,\omega}}\frac{[F_v:K_\omega]_s}{[F:K]_s}\int_{Y_{v}^\an}f_\omega\, c_1(\overline{L_{1,\omega}}|_{Y_{v}})\wedge\cdots\wedge c_1(\overline{L_{t,\omega}}|_{Y_{v}})\\
		=&\sum\limits_{\omega\in \Omega}\frac{\nu(\omega)}{|M_{F,\omega}|}\sum\limits_{v\in M_{F,\omega}}\int_{Y_{v}^\an}f_\omega\, c_1(\overline{L_{1,\omega}}|_{Y_{v}})\wedge\cdots\wedge c_1(\overline{L_{t,\omega}}|_{Y_{v}}).
	\end{align*}
	The last equality is from \eqref{eq:action 1}.
	
	Let $\omega\in \Omega$, and  $\val_{v_0}$ (with $v_0\in\Omega_F$) the restriction of $\val_\omega$ on $F$ via $F\to \overline{K}\overset{\jmath_\omega}{\to}\C_\omega$. For $Y^\sigma\in O(Y)$, via $F\to \overline{K}\overset{\jmath_\omega}{\to}\C_\omega$, we have $Y_\omega^\sigma=Y^\sigma\times_{\Spec(F),\jmath_\omega}\Spec(\C_\omega)$ . 
	Then
	\[Y_{v_0}^{\sigma,\an}=Y_{\omega}^{\sigma,\an}/\Gal(\overline{F_{v_0}}/F_{v_0})\] 
	and $c_1(\overline{L_{1,\omega}}|_{Y_{v_0}^\sigma})\wedge\cdots\wedge c_1(\overline{L_{t,\omega}}|_{Y_{v_0}^\sigma})$ is the push-forward of $c_1(\overline{L_{1,\omega}}|_{Y_{\omega}^{\sigma,\an}})\wedge\cdots\wedge c_1(\overline{L_{t,\omega}}|_{Y_{\omega}^{\sigma,\an}})$.
    The group $G$ acts on $O(Y)$ transitively since $Y$ is defined over $F$ and $F$ is normal. Notice that for any $v\in\Omega_\omega$, we have  $g(Y)_v= Y_{g(v)}$ via $g\colon F_v\simeq F_{g(v)}$. Then
	\begin{align*}
		&\frac{1}{|M_{F,\omega}|}\sum\limits_{v\in M_{F,\omega}}\int_{Y_{v}^\an}f_\omega\, c_1(\overline{L_{1,\omega}}|_{Y_{v}})\wedge\cdots\wedge c_1(\overline{L_{t,\omega}}|_{Y_{v}})\\
		=&\frac{1}{|G|}\sum\limits_{g\in G}\int_{Y_{g(v_0)}^{\an}}f_\omega\, c_1(\overline{L_{1,\omega}}|_{Y_{g(v_0)}})\wedge\cdots\wedge c_1(\overline{L_{t,\omega}}|_{Y_{g(v_0)}})\\
		=&\frac{1}{|G|}\sum\limits_{g\in G}\int_{g(Y)_{v_0}^{\an}}f_\omega\, c_1(\overline{L_{1,\omega}}|_{g(Y)_{v_0}})\wedge\cdots\wedge c_1(\overline{L_{t,\omega}}|_{g(Y)_{v_0}})\\
		=&\frac{1}{|G|}\sum\limits_{g\in G}\int_{g(Y)_{\omega}^{\an}}f_\omega\, c_1(\overline{L_{1,\C_\omega}}|_{g(Y)_{\omega}})\wedge\cdots\wedge c_1(\overline{L_{t,\C_\omega}}|_{g(Y)_{\omega}})\\
		=&\frac{1}{|O(Y)|}\sum\limits_{Y^\sigma\in O(Y)}\int_{Y_\omega^{\sigma,\an}}f_\omega\, c_1(\overline{L_{1,\C_\omega}}|_{Y_\omega^\sigma})\wedge\cdots\wedge c_1(\overline{L_{t,\C_\omega}}|_{Y_\omega^\sigma}).
	\end{align*}
Hence \[(\overline{L_1}\cdots \overline{L_t}\cdot\overline{M}\mid Y)_S=\frac{1}{|O(Y)|}\sum\limits_{\omega\in \Omega}\nu(\omega)\sum\limits_{Y^\sigma\in O(Y)}\int_{Y_\omega^{\sigma,\an}}f_\omega\, c_1(\overline{L_{1,\C_\omega}}|_{Y_\omega^\sigma})\wedge\cdots\wedge c_1(\overline{L_{t,\C_\omega}}|_{Y_\omega^\sigma}).\]
\end{proof}

\begin{theorem} \label{thm:equidistribution over function fields}
	Let $\overline{L}\in \widehat{\Pic}_{S,\Q}(U)_{\relnef}^{\nef,\YZ}$ be such that 
	$\deg_{\widetilde{L}}(U)>0$, where $\widetilde{L}$ denotes the image of $\overline{L}$ 
	in $\widetilde{\Pic}_\Q(U)_\integrable$. Let $(x_m)_{m\in I}$ be a generic net of 
	points in $U(\overline{K})$ which is small with respect to $\overline{L}$. Then for 
	any $v\in\Omega$ and any $f_v\in C_c(U_v^\an)$, we have 
	\begin{align}\label{eq:thm:main equi theorem}
		\lim_{m\in I}\frac{1}{|O(x_m)|}
		\sum_{x_m^{\sigma}\in O(x_m)}f_v(x_m^{\sigma})
		=\frac{1}{\deg_{\widetilde{L}}(U)}
		\int_{U_v^\an}f_v\, c_1(\overline{L}_v)^d.
	\end{align}
\end{theorem}

\begin{proof}
	We may assume that $U$ is normal (otherwise replace it by its normalization, over 
	which the equidistribution statement implies the one on $U$ by push-forward).
	
	When $\dim(T)=1$, the theorem follows from 
	\cite[Theorem~6.2.3]{biswas2024concave} (see also 
	\cite[Theorem~5.4.3]{yuan2021adelic} for the arithmetic nef case). Moreover, notice that the convergence given in \eqref{eq:thm:main equi theorem} is in fact a weakly convergence in the sense of \cite[Definition~4.2.1]{bogachev2018weak} by \cite[Proposition~4.5.11]{bogachev2018weak}, i.e.  \eqref{eq:thm:main equi theorem} holds for bounded functions on $U^\an_v$. 
	
	Now assume $\dim(T)\geq 2$. 
	$v\in\Omega$, and let $f_v\in C(X_v^\an)$ be a model function, i.e. 
	$(0,f_v)\in\widehat{\Div}_\Q(X_v)_\mo$ (see \cref{model Green function}). Define 
	an $S$-metrized line bundle $\overline{M}=(\OO_X,(\metr_\omega)_{\omega\in\Omega})$ 
	by
	\[
	-\log\|1\|_\omega\coloneq
	\begin{cases}
		f_v & \text{if }\omega=v,\\[2mm]
		0 & \text{if }\omega\in\Omega\setminus\{v\}.
	\end{cases}
	\]
	Then $\overline{M}\in\widehat{\Pic}_{S,\Q}(U)_\integrable^\YZ$.
	
	Let $T_{\mathbf{c}}$, $S_{\mathbf{c}}$ and the covering 
	$\alpha=(\alpha^\#,\alpha_\#,1)\colon S_{\mathbf{c}}\to S$ be as in 
	\cref{prop:B is covered by a curve as adelic curves}. Set 
	$U_{\mathbf{c}}\coloneq U\times_{\Spec(K)}\Spec(K_{\mathbf{c}})$ and let 
	$\overline{L_{\mathbf{c}}}$, $\overline{M_{\mathbf{c}}}$ be the pull-backs of 
	$\overline{L}$, $\overline{M}$ to $U_{\mathbf{c}}$. By 
	\cref{lemma:generic stable under base changes}, the net 
	$(x_{m,\mathbf{c}})_{m\in I}$ with 
	$x_{m,\mathbf{c}}\coloneq x_m\times_{\Spec(\overline{K})}\Spec(\overline{K_{\mathbf{c}}})$ 
	is generic on $U_{\mathbf{c}}$. 
	
	Since $\dim(T_{\mathbf{c}})=1$, the already established case of the theorem applies 
	to the adelic curve $S_{\mathbf{c}}$. Hence for every 
	$\omega\in\alpha_\#^{-1}(v)$, writing $\pi_\omega\colon U_{\mathbf{c},\omega}\to U_v$ 
	for the canonical morphism, we have
	\begin{equation}\label{eq:equidistribution over generic curve}
		\lim_{m\in I}\frac{1}{|O(x_{m,\mathbf{c}})|}
		\sum_{x_{m,\mathbf{c}}^\sigma\in O(x_{m,\mathbf{c}})}
		\pi_\omega^*(f_v)(x_{m,\mathbf{c}}^{\sigma})
		=\frac{1}{\deg_{\widetilde{L_{\mathbf{c}}}}(U_{\mathbf{c}})}
		\int_{U_{\mathbf{c},\omega}^\an}\pi_\omega^*(f_v)\,
		c_1(\overline{L_{\mathbf{c},\omega}})^d.
	\end{equation}
	(Notice that the fibre $\alpha_\#^{-1}(v)$ is finite.)
	
	On the other hand, by \cref{prop:intersection on subvariety} we may write the 
	arithmetic height of $x_{m,\mathbf{c}}$ with respect to $\overline{M_{\mathbf{c}}}$ as
	\begin{equation}\label{eq:height as sum over omega}
		(\overline{M_{\mathbf{c}}}\mid x_{m,\mathbf{c}})_{S_{\mathbf{c}}}
		=\sum_{\omega\in\alpha_\#^{-1}(v)}
		\frac{\nu_{\mathbf{c}}(\omega)}{|O(x_{m,\mathbf{c}})|}
		\sum_{x_{m,\mathbf{c}}^\sigma\in O(x_{m,\mathbf{c}})}
		\pi_\omega^*(f_v)(x_{m,\mathbf{c}}^{\sigma}),
	\end{equation}
	and similarly the arithmetic intersection number on $U_{\mathbf{c}}$ as
	\begin{equation}\label{eq:intersection as sum over omega}
		(\overline{L_{\mathbf{c}}}^d\cdot\overline{M_{\mathbf{c}}}\mid U_{\mathbf{c}})_{S_{\mathbf{c}}}
		=\sum_{\omega\in\alpha_\#^{-1}(v)}
		\nu_{\mathbf{c}}(\omega)
		\int_{U_{\mathbf{c},\omega}^\an}\pi_\omega^*(f_v)\,
		c_1(\overline{L_{\mathbf{c},\omega}})^d.
	\end{equation}
	Multiplying \eqref{eq:equidistribution over generic curve} by 
	$\nu_{\mathbf{c}}(\omega)$ and summing over the finite set 
	$\omega\in\alpha_\#^{-1}(v)$, then using \eqref{eq:height as sum over omega} and 
	\eqref{eq:intersection as sum over omega}, we obtain
	\[
	\lim_{m\in I}(\overline{M_{\mathbf{c}}}\mid x_{m,\mathbf{c}})_{S_{\mathbf{c}}}
	=\frac{1}{\deg_{\widetilde{L_{\mathbf{c}}}}(U_{\mathbf{c}})}
	(\overline{L_{\mathbf{c}}}^d\cdot\overline{M_{\mathbf{c}}}\mid U_{\mathbf{c}})_{S_{\mathbf{c}}}.
	\]
	By \cref{prop:intersection number after base changes}, the pull-back along 
	$S_{\mathbf{c}}\to S$ preserves arithmetic intersection numbers and geometric 
	degrees. Therefore 
	$\deg_{\widetilde{L_{\mathbf{c}}}}(U_{\mathbf{c}})=\deg_{\widetilde{L}}(U)$ and
	\[
	(\overline{M_{\mathbf{c}}}\mid x_{m,\mathbf{c}})_{S_{\mathbf{c}}}
	=(\overline{M}\mid x_m)_S,\qquad
	(\overline{L_{\mathbf{c}}}^d\cdot\overline{M_{\mathbf{c}}}\mid U_{\mathbf{c}})_{S_{\mathbf{c}}}
	=(\overline{L}^d\cdot\overline{M}\mid U)_S.
	\]
	Consequently,
	\[
	\lim_{m\in I}(\overline{M}\mid x_m)_S
	=\frac{1}{\deg_{\widetilde{L}}(U)}(\overline{L}^d\cdot\overline{M}\mid U)_S.
	\]
	By the very definition of $\overline{M}$ and \cref{prop:intersection on subvariety}, the left-hand side equals the average of 
	$f_v$ over the Galois orbit of $x_m$, while the right-hand side equals the integral 
	of $f_v$ against the measure $c_1(\overline{L}_v)^d$. Hence
	\begin{equation}\label{eq:equid for points 1}
		\lim_{m\in I}\frac{1}{|O(x_m)|}
		\sum_{x_m^{\sigma}\in O(x_m)}f_v(x_m^{\sigma})
		=\frac{1}{\deg_{\widetilde{L}}(U)}
		\int_{U_v^\an}f_v\, c_1(\overline{L}_v)^d.
	\end{equation}
	
	Finally, the model functions are dense in $C(X_v^\an)$ by Gubler's density theorem 
	(see \cite[Theorem~7.12]{gubler1998local}). Thus \eqref{eq:equid for points 1} holds for any $f_v\in C(X_v^\an)$. In particular, \eqref{eq:equid for points 1} holds for any $f_v\in C_c(U_v^\an)\subset C(X_v^\an)$.
	This completes the proof of \cref{thm:equidistribution over function fields}.
\end{proof}

\begin{remark}
	As mentioned in \cite[Theorem~3.1]{yuan2008big}, the theorem remains valid if one 
	replaces $U_v^\an$ by $U_{\C_v}^\an$, the analytification of 
	$U\times_{\Spec(K)}\Spec(\C_v)$. By the Stone--Weierstrass theorem, the two versions 
	are equivalent. 
\end{remark}

\subsection{Fundamental inequality}

\begin{definition}\label{def:ess minimal}
	Let $\overline{L}\in\widehat{\Pic}_{S,\Q}(U)_\integrable^\YZ$. The 
	\emph{essential minimum} of $\overline{L}$ is defined as
	\[
	\zeta_{\mathrm{ess}}(\overline{L})\coloneq
	\sup_{V\subset U}\inf_{x\in V(\overline{K})}h_{\overline{L}}(x),
	\]
	where the supremum is taken over all Zariski open subschemes $V$ of $U$.
\end{definition}

\begin{proposition}\label{proposition:fundamental inequality}
	Let $\overline{L}\in\widehat{\Pic}_{S,\Q}(U)_{\nef}^\YZ$ and 
	$\overline{M}\in\widehat{\Pic}_{S,\Q}(U)_{\integrable}^\YZ$ be such that the image 
	of $\overline{M}$ in $\widetilde{\Pic}_\Q(U)_{\mathrm{int}}$ is trivial. Then for any 
	$\epsilon\in\Q$, we have
	\[
	\zeta_{\mathrm{ess}}(\overline{L}\otimes\overline{M}^{\otimes\epsilon})
	\geq\frac{(\overline{L}^{d+1}\mid U)_S+\epsilon(d+1)(\overline{L}^d\cdot\overline{M}\mid U)_S}
	{(d+1)\deg_{\widetilde{L}}(U)}+O(\epsilon^2),
	\]
	where $\widetilde{L}$ denotes the image of $\overline{L}$ in 
	$\widetilde{\Pic}_\Q(U)_{\integrable}$. The implicit constant in $O(\epsilon^2)$ may 
	depend on $\overline{L}$ and $\overline{M}$, but is independent of $\epsilon$.
\end{proposition}

\begin{proof}
	When $\dim(T)=1$, the proposition follows from \cite[Lemma~5.3.4, 
	Theorem~5.2.2]{yuan2021adelic} for function fields of one variable together with 
	\cref{rmk:relation with yz adelic line bundle}; see also the proof of 
	\cite[Theorem~5.4.3]{yuan2021adelic}.
	
	Now assume $\dim(T)\geq 2$. 
	Let $T_{\mathbf{c}}$ be the generic curve given in 
	\cref{generic curves}, and let $S_{\mathbf{c}}$ be the corresponding adelic curve. 
	By \cref{prop:intersection number after base changes}, we have
	\[
	(\overline{L}^{d+1}\mid U)_S
	=(\overline{L_{\mathbf{c}}}^{d+1}\mid U_{\mathbf{c}})_{S_{\mathbf{c}}},\qquad
	(\overline{L}^d\cdot\overline{M}\mid U)_S
	=(\overline{L_{\mathbf{c}}}^d\cdot\overline{M_{\mathbf{c}}}\mid U_{\mathbf{c}})_{S_{\mathbf{c}}},
	\]
	and $\deg_{\widetilde{L}}(U)=\deg_{\widetilde{L_{\mathbf{c}}}}(U_{\mathbf{c}})$, where 
	$\overline{L_{\mathbf{c}}},\overline{M_{\mathbf{c}}}$ are the pull-backs of 
	$\overline{L},\overline{M}$ to $U_{\mathbf{c}}$, and $\widetilde{L_{\mathbf{c}}}$ is the image of $\overline{L_{\mathbf{c}}}$ in $\widetilde{\Pic}_{\Q}(U_{\mathbf{c}})_\cpt$.
	
	Let $Y'$ be a proper closed subset of $U_{\mathbf{c}}$. By 
	\cref{lemma:lift closed subset}, there exists a proper closed subset $Y$ of $X$ such 
	that every $x\in U(\overline{K})\setminus Y(\overline{K})$ lies outside $Y'$ when 
	viewed as a point of $U_{\mathbf{c}}(\overline{K_{\mathbf{c}}})$. Hence
	\[\inf\limits_{x\in U(\overline{K})\setminus Y}h_{\overline{L}}(x)\geq \inf\limits_{x\in U(\overline{K})\setminus Y'}h_{\overline{L}}(x)\geq \inf\limits_{x\in U(\overline{K_{\mathbf{c}}})\setminus Y'}h_{\overline{L_{\mathbf{c}}}}(x),\]
	and consequently $\zeta_{\mathrm{ess}}(\overline{L})
	\geq\zeta_{\mathrm{ess}}(\overline{L_{\mathbf{c}}})$.
	
	Since $\dim(T_{\mathbf{c}})=1$, applying the already established case to 
	$S_{\mathbf{c}}$ yields
	\begin{align*}
		\zeta_{\mathrm{ess}}(\overline{L}\otimes\overline{M}^{\otimes\epsilon})
		&\geq\zeta_{\mathrm{ess}}(\overline{L_{\mathbf{c}}}\otimes\overline{M_{\mathbf{c}}}^{\otimes\epsilon})\\
		&\geq\frac{(\overline{L_{\mathbf{c}}}^{d+1}\mid U_{\mathbf{c}})_{S_{\mathbf{c}}}
			+\epsilon(d+1)(\overline{L_{\mathbf{c}}}^d\cdot\overline{M_{\mathbf{c}}}\mid U_{\mathbf{c}})_{S_{\mathbf{c}}}}
		{(d+1)\deg_{\widetilde{L_{\mathbf{c}}}}(U_{\mathbf{c}})}+O(\epsilon^2)\\
		&=\frac{(\overline{L}^{d+1}\mid U)_S+\epsilon(d+1)(\overline{L}^d\cdot\overline{M}\mid U)_S}
		{(d+1)\deg_{\widetilde{L}}(U)}+O(\epsilon^2).
	\end{align*}
	This completes the proof.
\end{proof}

Using the fundamental inequality, we now give a second proof of 
\cref{thm:equidistribution over function fields} for the nef case.

\begin{proof}[Second proof of \cref{thm:equidistribution over function fields} 
	for the nef case]
	Assume that $\overline{L}\in\widehat{\Pic}_{S,\Q}(U)_\nef^\YZ$. We follow the strategy 
	of \cite[Theorem~5.4.3]{yuan2021adelic}. For any 
	$\overline{M}=(M,\metr_M)\in\widehat{\Pic}_{S,\Q}(U)_{\integrable}^\YZ$ with trivial 
	image in $\widetilde{\Pic}_\Q(U)_\cpt$, the fundamental inequality 
	(\cref{proposition:fundamental inequality}) gives
	\begin{equation}\label{eq:fundamental inequality}
		\liminf_{m\in I}h_{\overline{L}\otimes\overline{M}^{\otimes\epsilon}}(x_m)
		\geq\frac{(\overline{L}^{d+1}\mid U)_S+\epsilon(d+1)(\overline{L}^d\cdot\overline{M}\mid U)_S}
		{(d+1)\deg_{\widetilde{L}}(U)}+O(\epsilon^2).
	\end{equation}
	Since $(x_m)$ is small with respect 
	to $\overline{L}$, we have
	\[
	\lim_{m\in I}h_{\overline{L}}(x_m)
	=\frac{(\overline{L}^{d+1}\mid U)_S}{(d+1)\deg_{\widetilde{L}}(U)}.
	\]
	Since $h_{\overline{L}\otimes\overline{M}^{\otimes\epsilon}}
	=h_{\overline{L}}+\epsilon h_{\overline{M}}$, subtracting limit above from \eqref{eq:fundamental inequality} yields
	\[
	\liminf_{m\in I}\epsilon h_{\overline{M}}(x_m)
	\geq\epsilon\frac{(\overline{L}^d\cdot\overline{M}\mid U)_S}{\deg_{\widetilde{L}}(U)}
	+O(\epsilon^2).
	\]
	If $\epsilon>0$, dividing by $\epsilon$ we obtain
	\[
	\liminf_{m\in I}h_{\overline{M}}(x_m)
	\geq\frac{(\overline{L}^d\cdot\overline{M}\mid U)_S}{\deg_{\widetilde{L}}(U)}+O(\epsilon).
	\]
	Similarly, if $\epsilon<0$, dividing by $\epsilon$ reverses the inequality and gives
	\[
	\limsup_{m\in I}h_{\overline{M}}(x_m)
	\leq\frac{(\overline{L}^d\cdot\overline{M}\mid U)_S}{\deg_{\widetilde{L}}(U)}+O(|\epsilon|).
	\]
	Letting $\epsilon\to0$, we conclude that
	\begin{equation}\label{eq:key equality in pf of equidistribution}
		\lim_{m\in I}h_{\overline{M}}(x_m)
		=\frac{(\overline{L}^d\cdot\overline{M}\mid U)_S}{\deg_{\widetilde{L}}(U)}.
	\end{equation}
	
	Now fix $v\in\Omega$. Let $(\overline{L_j})_{j\in N_{\geq 1}}$ be a sequence in 
	$P_{S,\Q}(U)_\mo$ converging to $\overline{L}$ with respect to the 
	$\overline{B}$-boundary topology for some weak boundary divisor 
	$\overline{B}\in\widehat{\Div}_{S,\Q}(U)_\mo$. Let $X_j$ be projective $K$-models 
	of $U$ such that $\overline{L_j}\in P_{S,\Q}(X_j)_\mo$. For any 
	$f_v\in C_c(U_v^\an)$, viewed as a continuous function on $X_{1,v}^\an$ via the open 
	immersion $U\hookrightarrow X_1$, let $\overline{M}=(\OO_{X_1},\metr_M)\in 
	\widehat{\Pic}_{S,\Q}(X_1)_\mo$ be the model metrized line bundle determined by
	\[
	-\log\|1\|_{M,\omega}\coloneq
	\begin{cases}
		f_v & \text{if }\omega=v,\\[2mm]
		0 & \text{if }\omega\in\Omega\setminus\{v\}.
	\end{cases}
	\]
	By \cite[Proposition~2.3]{burgos2019the} (or \cref{prop:intersection on subvariety} 
	for $0$-dimensional subvarieties), we have
	\begin{equation}\label{eq:second proof 1}
		h_{\overline{M}}(x_m)
		=\frac{\nu(v)}{|O(x_m)|}\sum_{x_m^\sigma\in O(x_m)}f_v(x_m^\sigma).
	\end{equation}
	On the other hand, by \cite[Theorem~7.22~(v)]{cai2024abstract},
	\[
	(\overline{L}^d\cdot\overline{M}\mid U)_S
	=\lim_{j\to\infty}(\overline{L_j}^d\cdot\overline{M}\mid X_j)_S
	=\nu(v)\lim_{j\to\infty}\int_{X_{j,v}^\an}f_v\,
	c_1(\overline{L_{j,v}})^d.
	\]
	Combining this with \eqref{eq:key equality in pf of equidistribution} gives
	\begin{equation}\label{eq:second proof 2}
		\lim_{m\in I}h_{\overline{M}}(x_m)
		=\frac{\nu(v)}{\deg_{\widetilde{L}}(U)}
		\lim_{j\to\infty}\int_{X_{j,v}^\an}f_v\,c_1(\overline{L_{j,v}})^d.
	\end{equation}
	Since $f_v$ is compactly supported on $U_v^\an$, we may replace $X_{j,v}^\an$ by 
	$U_v^\an$ in the integral; moreover, by \cite[\S 3.6.7]{yuan2021adelic} or 
	\cite[Proposition~4.45~(e)]{cai2024abstract}, the measures 
	$c_1(\overline{L_{j,v}})^d|_{U_v^\an}$ converge weakly to 
	$c_1(\overline{L_v})^d$. Hence
	\begin{equation}\label{eq:second proof 3}
		\lim_{j\to\infty}\int_{X_{j,v}^\an}f_v\,c_1(\overline{L_{j,v}})^d
		=\int_{U_v^\an}f_v\,c_1(\overline{L_v})^d.
	\end{equation}
	The nef case of the theorem now follows from 
	\eqref{eq:second proof 1}, \eqref{eq:second proof 2} and \eqref{eq:second proof 3}.
\end{proof}

\subsection{Equidistribution theorem for subvarieties}

The following result follows immediately from \cite[Theorem~5.2.2~(2)]{yuan2021adelic}.

\begin{proposition}\label{prop:an criterion for effectiveness}
	Assume that $\dim(T)=1$. Let $\overline{L_1},\overline{L_2}\in 
	\widehat{\Pic}_{S,\Q}(U)_\nef^\YZ$. If
	\[
	(\overline{L_1}^{d+1}\mid U)_S-(d+1)(\overline{L_1}^d\cdot\overline{L_2}\mid U)_S>0,
	\]
	then $\overline{L_1}\otimes\overline{L_2}^{\otimes-1}$ is effective in the sense of 
	\cref{def:effective compactified metrized line bundle}.
\end{proposition}

\begin{lemma}\label{lemma for eq of subvar}
	Let $\overline{L}\in\widehat{\Pic}_{S,\Q}(U)^\YZ_\nef$ and 
	$\overline{M}\in\widehat{\Pic}_{S,\Q}(U)_\integrable^\YZ$ be such that the image of 
	$\overline{M}$ in $\widetilde{\Pic}_\Q(U)_\integrable$ is trivial. Let 
	$(Y_m)_{m\in I}$ be a generic net of proper closed subvarieties of 
	$U_{\overline{K}}$. If $(\overline{L}^d\cdot\overline{M}\mid U)_S>0$, then there 
	exist $N\in\N_{\geq 1}$ and $m_0\in I$ such that for all $m\geq m_0$,
	\[
	(\overline{L}^{\dim(Y_m)}\cdot(\overline{L}^{\otimes N}\otimes\overline{M})\mid Y_m)_S\geq 0,
	\]
	where the arithmetic intersection number on the left is defined in 
	\cref{rmk:arithmetic intersection number for subvarieties}.
\end{lemma}

\begin{proof}
	We first treat the case $\dim(T)=1$. The proof is similar to the argument in 
	\cite[Theorem~4.1]{faber2009equidistribution}. Write 
	$\overline{M}=\overline{M_1}\otimes\overline{M_2}^{\otimes-1}$ with 
	$\overline{M_1},\overline{M_2}\in\widehat{\Pic}_{S,\Q}(U)^\YZ_\nef$. For 
	$\epsilon\in\Q_{>0}$, expanding the intersection number gives
	\begin{align*}
		((\overline{L}\otimes\overline{M}^{\otimes\epsilon})^{d+1}\mid U)_S
		&=\sum_{i=0}^{d+1}\binom{d+1}{i}\epsilon^i(\overline{L}^{d+1-i}\cdot\overline{M}^i\mid U)_S\\
		&=(\overline{L}^{d+1}\mid U)_S+\epsilon(d+1)(\overline{L}^d\cdot\overline{M}\mid U)_S+O(\epsilon^2).
	\end{align*}
	Since $(\overline{L}^{d+1}\mid U)_S\geq 0$ and 
	$(\overline{L}^d\cdot\overline{M}\mid U)_S>0$, we have
	\begin{equation}\label{eq:lemma for eq of subvar 1}
		((\overline{L}\otimes\overline{M}^{\otimes\epsilon})^{d+1}\mid U)_S>0
	\end{equation}
	for $\epsilon$ sufficiently small. On the other hand,
\begin{align*}
	((\overline{L}\otimes \overline{M}^{\otimes\epsilon})^{d+1}\mid U)_S&=  ((\overline{L}\otimes \overline{M_1}^{\otimes\epsilon}\otimes \overline{M_2}^{\otimes-\epsilon})^{d+1}\mid U)_S\\
	&=\sum\limits_{i=0}^{d+1}\binom{d+1}{i}(-\epsilon)^{i}((\overline{L}\otimes \overline{M_1}^{\otimes\epsilon})^{d+1-i}\cdot \overline{M_2}^{i}\mid U)_S\\
	&=((\overline{L}\otimes \overline{M_1}^{\otimes\epsilon})^{d+1}\mid U)_S -(d+1)((\overline{L}\otimes \overline{M_1}^{\otimes\epsilon})^{d}\cdot \overline{M_2}^{\otimes\epsilon}\mid U)_S +O(\epsilon^2).
\end{align*}
	Combining this with \eqref{eq:lemma for eq of subvar 1} yields
	\begin{equation}\label{eq:lemma for eq of subvar 2}
		((\overline{L}\otimes\overline{M_1}^{\otimes\epsilon})^{d+1}\mid U)_S
		-(d+1)((\overline{L}\otimes\overline{M_1}^{\otimes\epsilon})^d\cdot
		\overline{M_2}^{\otimes\epsilon}\mid U)_S>0
	\end{equation}
	for $\epsilon$ small. Take $\epsilon=1/N$ with $N\in\N_{\geq 1}$ large. By 
	\cref{prop:an criterion for effectiveness}, 
	$\overline{L}\otimes\overline{M}^{\otimes 1/N}
	=(\overline{L}\otimes\overline{M_1}^{\otimes 1/N})\otimes
	\overline{M_2}^{\otimes(-1/N)}$ is effective. Hence there exists a rational 
	section $s$ of $\overline{L}\otimes\overline{M}^{\otimes 1/N}$ satisfying the 
	conditions in \cref{def:effective compactified metrized line bundle}. Since 
	$(Y_m)_{m\in I}$ is generic, there is $m_0\in I$ such that for all $m\geq m_0$, $Y_m$ is 
	not contained in the support of $\mathrm{div}(s)$. Therefore the restriction of 
	$\overline{L}\otimes\overline{M}^{\otimes 1/N}$ to $Y_m$ is effective. Since $\overline{L}\in \widehat{\Pic}_{S,\Q}(U)^\YZ_\nef$, we have 
	\[(\overline{L}^{\dim(Y_m)}\cdot (\overline{L}\otimes\overline{M}^{\otimes 1/N})\mid Y_m)_S 
	\geq 0,\]
	which proves the lemma when $\dim(T)=1$.
	
	Now assume $\dim(T)\geq 2$. 
	Let $T_{\mathbf{c}}$, $S_{\mathbf{c}}$ be as in 
	\cref{generic curves}, and set $U_{\mathbf{c}}=U\times_{\Spec(K)}\Spec(K_{\mathbf{c}})$ 
	and $Y_{m,\mathbf{c}}=Y_m\times_{\Spec(\overline{K})}\Spec(\overline{K_{\mathbf{c}}})$. 
	Let $\overline{L_{\mathbf{c}}},\overline{M_{\mathbf{c}}}$ be the pull-backs of 
	$\overline{L},\overline{M}$. By \cref{lemma:generic stable under base changes}, 
	$(Y_{m,\mathbf{c}})$ is generic on $U_{\mathbf{c}}$, and by 
	\cref{prop:intersection number after base changes},
	\[
	(\overline{L_{\mathbf{c}}}^d\cdot\overline{M_{\mathbf{c}}}\mid U_{\mathbf{c}})_{S_{\mathbf{c}}}
	=(\overline{L}^d\cdot\overline{M}\mid U)_S>0.
	\]
	Applying the already established case to $S_{\mathbf{c}}$, we obtain $N$ and $m_0$ 
	such that for all $m\geq m_0$,
	\[
	(\overline{L_{\mathbf{c}}}^{\dim(Y_m)}\cdot(\overline{L_{\mathbf{c}}}^{\otimes N}
	\otimes\overline{M_{\mathbf{c}}})\mid Y_{m,\mathbf{c}})_{S_{\mathbf{c}}}\geq 0.
	\]
	By \cref{prop:intersection number after base changes} again, the same inequality 
	holds over $S$. This completes the proof.
\end{proof}

We are now ready to prove an equidistribution theorem for subvarieties.

\begin{theorem}\label{thm:equidistribution for subvarieties over function fields}
	Let $\overline{L}\in\widehat{\Pic}_{S,\Q}(U)_{\nef}^\YZ$ be such that 
	$\deg_{\widetilde{L}}(U)>0$ and $(\overline{L}^{d+1}\mid U)_S=0$, where 
	$\widetilde{L}$ denotes the image of $\overline{L}$ in 
	$\widetilde{\Pic}_\Q(U)_\integrable$. Let $(Y_m)_{m\in I}$ be a generic net of 
	$t$-dimensional proper closed subvarieties of $U_{\overline{K}}$. Assume that 
	$\deg_{\widetilde{L}}(Y_m)>0$ for all $m\in I$ and that $(Y_m)_{m\in I}$ is small 
	with respect to $\overline{L}$. Then for any $v\in\Omega$ and any 
	$f_v\in C_c(U_v^\an)$, we have
	\[
	\lim_{m\in I}\frac{1}{|O(Y_m)|\deg_{\widetilde{L}}(Y_m)}
	\sum_{Y_{m,v}^\sigma\in O(Y_m)}
	\int_{Y_{m,v}^{\sigma,\an}}f_v\,
	c_1(\overline{L_{\C_v}}|_{Y_{m,v}^\sigma})^t
	=\frac{1}{\deg_{\widetilde{L}}(U)}
	\int_{U_v^\an}f_v\,c_1(\overline{L}_v)^d,
	\]
	where $Y_{m,v}^{\sigma,\an}$ is the analytification of 
	$Y_m^\sigma\times_{\Spec(\overline{K}),\jmath_v}\Spec(\C_v)$, and $f_v$ is viewed 
	as a continuous function on $Y_{m,v}^{\sigma,\an}$ via the natural map 
	$Y_{m,v}^{\sigma,\an}\to U_{\C_v}^\an\to U_v^\an$.
\end{theorem}

\begin{proof}
	Let $X$ be a projective $k$-model of $U$. Let $f_v\in C(X_v^\an)$ be a model 
	function, i.e. $(\OO_{X_v},e^{-f_v})\in P_\Q(X_v)_\mo$ (see \cref{metrics}). Let 
	$c_v\in\log\sqrt{|K_v|^\times}=\Q$ be such that
	\begin{equation}\label{eq:cv-subvar}
		\frac{1}{\deg_{\widetilde{L}}(U)}
		\int_{U_v^\an}f_v\,c_1(\overline{L}_v)^d>c_v.
	\end{equation}
Set $\overline{M} = (\OO_X, (\metr_{M,\omega})_{\omega\in \Omega})$ the $S$-metric line bundle such that  
\[-\log\|1\|_{M,\omega}\coloneq\begin{cases}
	f_v-c_v & \text{ if $\omega=v$;}\\
	0& \text{ if $\omega\in\Omega\setminus\{v\}$.}
\end{cases}\] 
Then $\overline{M}\in \widehat{\Pic}_{S,\Q}(X)_\mo$.
	By \cref{prop:intersection on subvariety} and \eqref{eq:cv-subvar}, we have
	\[
	(\overline{L}^d\cdot\overline{M}\mid U)_S
	=\nu(v)\int_{U_{\C_v}^\an}(f_v-c_v)\,c_1(\overline{L_{\C_v}})^d
	=\nu(v)\int_{U_v^\an}f_v\,c_1(\overline{L}_v)^d
	-c_v\cdot\nu(v)\cdot\deg_{\widetilde{L}}(U)
	>0,
	\]
	where the second equality follows from Guo's equality (see \cite[Theorem~1.2]{guo2023an} 
	or \cite[Proposition~4.45~(g)]{cai2024abstract}). Similarly, for any $m\in I$,
	\begin{align*}
		(\overline{L}^t\cdot\overline{M}\mid Y_m)_S
		&=\frac{\nu(v)}{|O(Y_m)|}
		\sum_{Y_{m,v}^\sigma\in O(Y_m)}
		\int_{Y_{m,v}^{\sigma,\an}}(f_v-c_v)\,
		c_1(\overline{L_{\C_v}}|_{Y_m})^t\\
		&=\frac{\nu(v)}{|O(Y_m)|}
		\sum_{Y_{m,v}^\sigma\in O(Y_m)}
		\int_{Y_{m,v}^{\sigma,\an}}f_v\,
		c_1(\overline{L_{\C_v}}|_{Y_m})^t
		-c_v\cdot\nu(v)\cdot\deg_{\widetilde{L}}(Y_m).
	\end{align*}
	By \cref{lemma for eq of subvar}, there exist $N\in\N_{\geq 1}$ and $m_0\in I$ 
	such that for all $m\geq m_0$,
	\[
	(\overline{L}^t\cdot(\overline{L}^{\otimes N}\otimes\overline{M})\mid Y_m)_S
	=(\overline{L}^t\cdot\overline{M}\mid Y_m)_S
	+N(\overline{L}^{t+1}\mid Y_m)_S
	\geq 0,
	\]
	i.e.
	\begin{align*}
		\frac{\nu(v)}{|O(Y_m)|}
		\sum_{Y_{m,v}^\sigma\in O(Y_m)}
		\int_{Y_{m,v}^{\sigma,\an}}f_v\,
		c_1(\overline{L_{\C_v}}|_{Y_m})^t
		+N(\overline{L}^{t+1}\mid Y_m)_S
		\geq c_v\cdot\nu(v)\cdot\deg_{\widetilde{L}}(Y_m).
	\end{align*}
Dividing by $\deg_{\widetilde{L}}(Y_m)$ and letting $m\in I$ go to infinity, we obtain
\[
\liminf_{m\in I}\frac{1}{|O(Y_m)|\deg_{\widetilde{L}}(Y_m)}
\sum_{Y_{m,v}^\sigma\in O(Y_m)}
\int_{Y_{m,v}^{\sigma,\an}}f_v\,
c_1(\overline{L_{\C_v}}|_{Y_m})^t
\geq c_v,
\]
because $(Y_m)_{m\in I}$ is small with respect to $\overline{L}$ and $(\overline{L}^{d+1}\mid U)_S=0$, so
\[
\lim_{m\in I}\frac{(\overline{L}^{t+1}\mid Y_m)_S}{(t+1)\deg_{\widetilde{L}}(Y_m)}
=\frac{(\overline{L}^{d+1}\mid U)_S}{(d+1)\deg_{\widetilde{L}}(U)}=0.
\]
Letting $c_v$ converge to $\frac{1}{\deg_{\widetilde{L}}(U)}\int_{U_v^\an}f_v\,c_1(\overline{L}_v)^d$ from below, we deduce
\begin{equation}\label{ineq:liminf-subvar}
	\liminf_{m\in I}
	\frac{1}{|O(Y_m)|\deg_{\widetilde{L}}(Y_m)}
	\sum_{Y_{m,v}^\sigma\in O(Y_m)}
	\int_{Y_{m,v}^{\sigma,\an}}f_v\,
	c_1(\overline{L_{\C_v}}|_{Y_m})^t
	\geq
	\frac{1}{\deg_{\widetilde{L}}(U)}
	\int_{U_v^\an}f_v\,c_1(\overline{L}_v)^d.
\end{equation}
Applying \eqref{ineq:liminf-subvar} to $-f_v$ yields
\[
\limsup_{m\in I}
\frac{1}{|O(Y_m)|\deg_{\widetilde{L}}(Y_m)}
\sum_{Y_{m,v}^\sigma\in O(Y_m)}
\int_{Y_{m,v}^{\sigma,\an}}f_v\,
c_1(\overline{L_{\C_v}}|_{Y_m})^t
\leq
\frac{1}{\deg_{\widetilde{L}}(U)}
\int_{U_v^\an}f_v\,c_1(\overline{L}_v)^d.
\]
Combining this with \eqref{ineq:liminf-subvar}, we conclude that the limit exists and
\begin{align}\label{ineq:eq for subvar-v}
\lim_{m\in I}
\frac{1}{|O(Y_m)|\deg_{\widetilde{L}}(Y_m)}
\sum_{Y_{m,v}^\sigma\in O(Y_m)}
\int_{Y_{m,v}^{\sigma,\an}}f_v\,
c_1(\overline{L_{\C_v}}|_{Y_m})^t
=
\frac{1}{\deg_{\widetilde{L}}(U)}
\int_{U_v^\an}f_v\,c_1(\overline{L}_v)^d.
\end{align}
	
	By Gubler's density theorem (see \cite[Theorem~7.12]{gubler1998local}), 
	\eqref{ineq:eq for subvar-v} holds for all $f_v\in C(X_v^\an)$. In particular, 
	it holds for all $f_v\in C_c(U_v^\an)\subset C(X_v^\an)$. This proves the theorem.
\end{proof}

\begin{remark}
	The proof also works over number fields. Hence one obtains a number-field 
	version of \cref{thm:equidistribution for subvarieties over function fields}.
\end{remark}


\section{Equidistribution theorem in big case}
\label{sec:equidistribution big}
In this section, we will prove an equidistribution theorem for big compactified 
$S$-metrized line bundles on normal quasi-projective varieties over a function field 
of characteristic $0$. The assumption on characteristic is only used to prove the 
Fujita approximation theorem (\cref{thm:fujita approximation}) over a function field, 
which allows us to give an intrinsic definition of the measures (see 
\cref{measure given by positive intersection}). This measure can also be defined by 
push-forward without the assumption; in this case, the equidistribution theorem also 
holds.

In this section, we fix a field $k$ of characteristic $0$ and an adelic curve 
$S=(K,\Omega,\mathcal{A},\nu)$ given by a normal scheme $T$ projective over $k$ and 
an ample line bundle $\mathbf{c}$, as in \cref{adelic curve defined by projective variety}.

\subsection{Preparation}

In this subsection, we recall certain concepts.

\begin{art}
	An \emph{adelic vector bundle} $\overline{V}=(V,\metr)$ over $S$ is a finite-dimensional 
	vector space $V$ equipped with a norm family $\metr=(\metr_\omega)_{\omega\in\Omega}$ 
	such that $\metr$ is \emph{$\mathcal{A}$-measurable} (see 
	\cite[\S~4.1.13]{chen2020arakelov}) and \emph{strongly dominated} (see 
	\cite[Definition~4.1.11]{chen2020arakelov}, or \cite[2.4.9]{chen2022hilbert}). 
	
	For an adelic vector bundle $\overline{V}$ over $S$ in the sense of 
	\cite[Definition~4.1.28]{chen2020arakelov}, let $\det(\overline{V})=(\det(V), 
	\metr_{\det})$ be the determinant bundle of $\overline{V}$. The \emph{Arakelov degree} 
	of $\overline{V}$ is defined as
	\[
	\widehat{\deg}(\overline{V})
	=\widehat{\deg}(\det(\overline{V}))
	\coloneq\int_{\Omega}-\log\|s\|_{\det,\omega}\,\nu(d\omega),
	\]
	where $s$ is any non-zero section of $\det(V)$. The Arakelov degree is well-defined 
	since $S$ is proper and $\dim_K\det(V)=1$. The \emph{positive Arakelov degree} of 
	$\overline{V}$ is defined as
	\[
	\widehat{\deg}_+(\overline{V})
	\coloneq\sup_{\overline{W}}\{\widehat{\deg}(\overline{W})\},
	\]
	where $\overline{W}$ varies over all subspaces of $\overline{V}$ equipped with the 
	restriction norms.
\end{art}

\begin{art}\label{S-ample}
	Let $X$ be a projective variety over $K$, and $\overline{L}=(L,\metr)\in 
	\widehat{\Pic}_S(X)_\mo$. We say that $\overline{L}$ is \emph{$S$-ample} (or ample in 
	\cite[Definition~9.1.1]{chen2022hilbert}) if $\overline{L}$ lies in $R_S(X)_\mo$ 
	(see \cref{model metrics and model Green functions}), $L$ is ample, and there exists 
	$\varepsilon>0$ such that for every integral closed subscheme $Z$ of $X$, we have
	\[
	(\overline{L}|_{Z}^{\dim(Z)+1}\mid Z)_S
	\geq\varepsilon\deg_L(Z)(\dim(Z)+1).
	\]
	We say that $\overline{L}$ is \emph{$S$-nef} (or nef in 
	\cite[Definition~9.1.4]{chen2022hilbert}) if there exists an $S$-ample metrized line 
	bundle $\overline{A}\in\widehat{\Pic}_S(X)_\mo$ and a positive integer $N$ such that 
	$\overline{L}^{\otimes n}\otimes\overline{A}$ is $S$-ample for any $n\in\N_{\geq N}$. 
	These notions can be extended to elements in 
	$\widehat{\Pic}_{S,\Q}(X)_\mo=\widehat{\Pic}_{S}(X)_\mo\otimes_\Z\Q$.
	
	Let $\mathcal{L}$ be a line bundle on a $T$-model $\mathcal{X}$ of $X$ inducing the 
	metric on $\overline{L}$. If $\mathcal{L}$ is ample (resp. nef) on $\mathcal{X}$ over 
	$k$, then by \cite[Proposition~4.9]{luo2025a} and 
	\cite[Proposition~9.1.2]{chen2022hilbert}, $\overline{L}$ is $S$-ample (resp. $S$-nef).
\end{art}

\begin{art}\label{arithmetic volume}
Let $U$ be a normal quasi-projective variety over $K$, $\overline{L}=(L,\metr)\in\widehat{\Pic}_{S,\Q}(U)_\cpt^\YZ$, and
$\widetilde{L}\in\widetilde{\Pic}_\Q(U)_\cpt$ the image of $\overline{L}$ under the 
canonical morphism in \eqref{eq:canonical morphism from metrzied to geometric}. 
Since $U$ is normal, $H^0(U,\widetilde{L})$ is a $K$-vector space (see \cref{def:compactified divisor and sections}); when $U$ is 
projective and $L$ comes from an algebraic line bundle, the same holds without the 
normality assumption. For any 
$s\in H^0(U,\widetilde{L})$ and $\omega\in\Omega$, we define the \emph{super-norm}
\begin{align}\label{eq:sup-norm}
	\|s\|_{\sup,\omega}\coloneq\sup\{\|s\|_\omega(x)\mid x\in U_\omega^\an\}.
\end{align}
	called the \emph{super-norm} of $s$. We consider the space
 \[{H^0_+(U,\overline{L})}\coloneq\left\{s\in H^0(U,\widetilde L)\mid \text{{$\|s\|_{\sup,\omega}<\infty$ for any $\omega\in \Omega$,}} \tint_{\Omega} \log\|s\|_{\sup,\omega} \,\nu(d\omega)<\infty\right\}.\] 
where $\tint$ denotes the upper integral with respect to $\nu$.

	 We set
	\[
	\overline{V_{\overline{L},\bullet}}
	=\left\{\overline{V_{\overline{L},m}}\right\}_{m\in\N}
	\coloneq\left\{(H^0_+(U,\overline{L}^{\otimes m}),\metr_{\sup,m})\right\}_{m\in\N},
	\]
	and define the \emph{arithmetic volume} of $\overline{L}$ as
	\[
	\widehat{\vol}(\overline{L})
	\coloneq\widehat{\vol}(\overline{V_{\overline{L},\bullet}})
	=\limsup_{m\to\infty}\frac{\widehat{\deg}_+(\overline{V_{\overline{L},m}})}{m^{d+1}/(d+1)!}.
	\]
	It turns out that the limit superior above is a limit; see 
	\cite[Theorem~4.3.2]{biswas2024concave}. We say that $\overline{L}$ is 
	(\emph{arithmetically}) \emph{big} if $\widehat{\vol}(\overline{L})>0$. If $\overline{L}$ is big, then by \cite[Theorem~4.3.2.]{biswas2024concave}, the image $\widetilde{L}$ of $\overline L$ in $\widetilde{\Pic}_\Q(U)_\cpt$ is \emph{big} in the sense of \cite[3.1.3]{biswas2024concave} (or \cite[5.2.2]{yuan2021adelic}).
	
	Let $(\overline{L_j})_{j\in\N_{\geq 1}}$ be a sequence in 
	$P_{S,\Q}(U)_\mo=\varinjlim_{X}\widehat{\Pic}_{S,\Q}(X)_\mo$ (see 
	\cref{def:boundarytopologyglobal}) consisting of $S$-metrized line bundles on some 
	projective models converging to $\overline{L}$ with respect to some weak boundary 
	divisor. Then
	\begin{align}\label{eq:convergence of arithmetic volume}
		\lim_{j\to\infty}\widehat{\vol}(\overline{L_j})=\widehat{\vol}(\overline{L})
	\end{align}
	by \cite[Theorem~4.4.3]{biswas2024concave}.
\end{art}

\begin{lemma}\label{big implies effective}
	Let $X$ be a projective variety over $K$, $\overline{L}\in\widehat{\Pic}_{S}(X)_\mo$. If $\overline{L}$ is (arithmetically) 
	big, then $\overline{L}$ is effective.
\end{lemma}
\begin{proof}
	Let $\mathcal{L}$ be a line bundle on a projective $T$-model $\mathcal{X}$ of $X$ that 
	induces the metric on $\overline{L}$.
	
	When $\dim(T)=1$, since the arithmetic volume of $\overline{L}$ coincides with the one 
	defined in \cite[Definition~5.1.3, Section~5.2]{yuan2021adelic} (see 
	\cite[Example~6.2.3]{biswas2024concave}), we have $\dim_kH^0(\mathcal{X},\mathcal{L})\geq 1$. 
	Hence $\overline{L}$ is effective.
	
	When $\dim(T)\geq 2$, we 
	use the notation of \cref{generic curves}. Let 
	$\mathcal{L}_{\mathbf{c}}$ be the pull-back of $\mathcal{L}$ to 
	$\mathcal{X}_{\mathbf{c}}=\mathcal{X}\times_TT_{\mathbf{c}}$, and let 
	$\overline{L_{\mathbf{c}}}$ be the pull-back of $\overline{L}$ to 
	$X_{\mathbf{c}}=X\times_{\Spec(K)}\Spec(K_{\mathbf{c}})$. Since 
	$\widehat{\vol}(\overline{L_{\mathbf{c}}})\geq\widehat{\vol}(\overline{L})$, the line 
	bundle $\overline{L_{\mathbf{c}}}$ is big, hence effective by the case $\dim(T)=1$.
	
	Let $\omega\in\Omega$ and let $\omega'\in\Omega_{\mathbf{c}}$ be a preimage of $\omega$ 
	via the covering $S_{\mathbf{c}}\to S$. Set 
	$\mathcal{X}_{\omega}=\mathcal{X}\times_T\Spec(K_\omega^\circ)$, 
	$\mathcal{L}_\omega\coloneq\mathcal{L}|_{\mathcal{X}_{\omega}}$, 
	$\mathcal{X}_{\mathbf{c},\omega'}=\mathcal{X}_{\mathbf{c}}\times_{T_{\mathbf{c}}}
	\Spec(K_{\mathbf{c},\omega'}^\circ)=\mathcal{X}_\omega\times_{\Spec(K_{\omega}^\circ)}
	\Spec(K_{\mathbf{c},\omega'}^\circ)$ and 
	$\mathcal{L}_{\mathbf{c},\omega'}=\mathcal{L}_{\mathbf{c}}|_{\mathcal{X}_{\mathbf{c},\omega'}}$. 
Since $K^\circ_\omega$ is a DVR and $K^\circ_{\mathbf{c},\omega'}$ is torsion-free as a 
$K^\circ_\omega$-module, the morphism 
$\Spec(K^\circ_{\mathbf{c},\omega'})\to\Spec(K^\circ_\omega)$ is flat. Hence
\[
H^0(\mathcal{X}_{\mathbf{c},\omega'},\mathcal{L}_{\mathbf{c},\omega'})
=H^0(\mathcal{X}_{\omega},\mathcal{L}_{\omega})
\otimes_{K^\circ_\omega}K^\circ_{\mathbf{c},\omega'}.
\]
	The effectiveness of $\overline{L_{\mathbf{c}}}$ implies that 
	$H^0(\mathcal{X}_{\mathbf{c},\omega'},\mathcal{L}_{\mathbf{c},\omega'})\neq 0$. Hence 
	$H^0(\mathcal{X}_{\omega},\mathcal{L}_{\omega})\neq 0$, so $\mathcal{L}_{\omega}$ is 
	effective. As $\omega$ is arbitrary, $\overline{L}$ is effective.
\end{proof}

\begin{remark}\label{rmk:big implies effective}
	By the continuity of arithmetic volume (see \eqref{eq:convergence of arithmetic volume}), 
	\cref{big implies effective} holds for any normal quasi-projective variety $U$ and any 
	$\overline{L}\in\widehat{\Pic}_{S,\Q}(U)_\cpt^\YZ$.
\end{remark}

\subsection{Fujita approximation theorem over a function field}

The Fujita approximation theorem over a number field is proved independently in \cite{chen2010arithmetic} and \cite{yuan2009on}. Recently, Liu established an adelic-curve version Fujita approximation theorem for $K$ perfect in \cite[Corollary~5.2]{liu2024arithmetic}. Following his idea, we can show our version of Fujita approximation theorem over high-dimensional base; see \cref{thm:fujita approximation}.

We denote by $S_0\coloneq(K,\{0\},\val_0)$ the adelic curve associated to the trivially valued field $K$.

\begin{art}
	Let $R$ be a domain with fraction field $\Frac(R)$, and $\mathcal{M}$ an $R$-module. For a $\Frac(R)$-vector subspace $V$ of $M\coloneq \mathcal{M}\otimes_R\Frac(R)$, the \emph{saturation} of $V$ in $\mathcal M$ is defined as 
	\[\mathrm{Sat}_{\mathcal M}(V)\coloneq\{m\in \mathcal{M}\mid m\otimes 1\in V\}.\] 
	Notice that $V=\mathrm{Sat}_{\mathcal M}(V)\otimes_R\Frac(R)$.
	If $\mathcal M$ is torsion-free, we can view $\mathcal M$ as a subset of $M$, then $\mathrm{Sat}_{\mathcal M}(V)=\mathcal{M}\cap V$. For a submodule $\mathcal{N}$ of $\mathcal{M}$, we say that $\mathcal{N}$ is \emph{saturated} in $\mathcal{M}$ if for any $a\in R\setminus\{0\}$ and $m\in\mathcal{M}$, $am \in \mathcal{N}$ implies $m\in \mathcal{N}$, equivalently,
	\[\mathrm{Sat}_{\mathcal{M}}(\mathcal{N}\otimes_R\Frac(R)) = \mathcal{N},\]
	or equivalently, $\mathcal{M}/\mathcal{N}$ is torsion-free. 
\end{art}

\begin{art}
	Let $\overline{V}=(V,\metr)$ be an adelic vector bundle over $S$.
	Let $(\mathcal{H}^t(V))_{t\in\R}$ be the \emph{Harder-Narasimhan filtration} (see \cite[Definition~4.3.45]{chen2020arakelov} or \cite[2.4.1]{biswas2024concave}) of $V$ with respect to $\metr$. We set
	\begin{gather}\label{eq:hn norm}
		\begin{aligned}
			\metr_{\mathrm{HN}}\colon V&\to \R_{\geq0},\\
			s&\mapsto \exp\{-\sup\{t\in\R\mid s\in \mathcal{H}^t({V})\}\}.
		\end{aligned}
	\end{gather}
	Then $\overline{V_{\mathrm{HN}}}\coloneq (V,\metr_{\mathrm{HN}})$ is an adelic vector bundle over $S_0$.
\end{art}

\begin{art}\label{algebra for line bundles}
	Let $X$ be a projective variety over $K$, and let $\overline{L}=(L,\metr)\in \widehat{\Pic}_{S}(X)_\mo$ with $L$ big. Let $\mathcal{L}$ be a line bundle on a model $\pi\colon\mathcal{X}\to T$ of $X$ inducing $\overline{L}$ (In general, $\mathcal{L}$ is a model of $L^{\otimes N}$ for some integer $N\geq 1$, and then it induces the adelic line bundle $\overline{L}^{\otimes N}$. For simplicity, we assume throughout that $N=1$, so that $\mathcal{L}$ itself induces $\overline{L}$; see \cref{model metrics and model Green functions}).
	
	Let $\Omega_a\subset\Omega$ be the subset determined by an ample divisor on $T$ such that for every $v\in\Omega\setminus\Omega_a$, the base change $\mathcal{X}_v\coloneq\mathcal{X}\times_T\Spec(K_v^\circ)$ is smooth over $K_v^\circ$. Set
	\[D(\Omega_a)\coloneq T\setminus\bigcup_{\tau\in\Omega_a}\overline{\{\tau\}},\]
	where $\overline{\{\tau\}}$ denotes the Zariski closure of $\tau$ in $T$. Since $D(\Omega_a)$ is affine over $k$, we write $R_a$ for its coordinate ring.
	
	Let $E=\bigoplus_{n\geq 0}E_n=\bigoplus_{n\geq 0}H^0(X,L^{\otimes n})$ be the sectional $K$-algebra of $L$. We equip each $E_n$ with the super-norm $\metr_{\sup}=(\metr_{\sup,\omega})_{\omega\in\Omega}$ defined by \eqref{eq:sup-norm}, and set $\overline{E}=\bigoplus_{n\geq 0}\overline{E_n}=\bigoplus_{n\geq 0}(E_n,\metr_{\sup})$. For each $n$, we denote by $\metr_n^{\mathrm{HN}}$ the norm on $E_n$ defined in \eqref{eq:hn norm}. Because the norm on the direct sum $\bigoplus_{n\geq 0}(E_n,\metr_{n}^{\mathrm{HN}})$ is not multiplicative, we introduce the \emph{spectral norm} on $E_n$ by
	\begin{align*}
		\metr_{\mathrm{sp},n}\colon E_n=H^0(X,L^{\otimes n})&\to\R_{\geq0},\\
		s &\mapsto \lim_{m\to\infty}\left(\|s^{\otimes m}\|_{mn}^{\mathrm{HN}}\right)^{1/m}.
	\end{align*}
	Then $(E_n, \metr_{\mathrm{sp},n})$ is an adelic vector bundle over $S_0$, and for any $s\in E_m$ and $t\in E_n$ we have
	\[\|s\otimes t\|_{\mathrm{sp},m+n}\leq \|s\|_{\mathrm{sp},m}\|t\|_{\mathrm{sp},n}.\]
	
	For any $\lambda\in\R$, let $E_{\mathrm{HN}}^{\lambda}=\bigoplus_{n\geq 0}E_{\mathrm{HN},n}^{\lambda}$ be the graded sub-$K$-algebra of $E$ defined by
	\[E_{\mathrm{HN},0}^{\lambda} \coloneq K, \qquad E_{\mathrm{HN},n}^{\lambda}\coloneq \mathrm{Vect}_K\{s\in E_n\mid \|s\|_{\mathrm{sp},n}\leq e^{-\lambda n}\},\]
	where $\mathrm{Vect}_K(\,\cdot\,)$ denotes the $K$-linear span.
	By \cite[Proposition~4.2]{liu2024arithmetic}, there exists $p\in\N_{\geq 1}$ such that $E^0_{\mathrm{HN},p}\not=0$. For such a $p$, let $E^{(p)}=\bigoplus_{n\geq 0}E^{(p)}_{\mathrm{HN},np}$ be the graded sub-$K$-algebra of $E_{\mathrm{HN}}^{0}$ generated by $E_{\mathrm{HN},p}^{0}$. Note that $E_{\mathrm{HN},np}^{(p)}\subset E_{\mathrm{HN},np}^{0}$ for all $n$.
	
	Now set $\mathcal{X}_a\coloneq \mathcal{X}\times_{T}D(\Omega_a)$ and $\mathcal{L}_a\coloneq \mathcal{L}|_{\mathcal{X}_a}$. Let $\mathcal{E}=\bigoplus_{n\geq 0}\mathcal{E}_n=\bigoplus_{n\geq 0}H^0(\mathcal{X}_a,\mathcal{L}_a^{\otimes n})$ be the sectional $R_a$-algebra of $\mathcal{L}_a$. For any $n\in\N$ and any $p\in\N_{\geq 1}$ with $E^0_{\mathrm{HN},p}\not=0$, we denote by $\mathcal{E}_{\mathrm{HN},np}^{(p)}$ the saturation of $E^{(p)}_{\mathrm{HN},np}$ in $\mathcal{E}_{np}$, i.e.
	\[\mathcal{E}_{\mathrm{HN},np}^{(p)} \coloneq \mathcal{E}_{np}\cap E_{\mathrm{HN},np}^{(p)}\subset E_{\mathrm{HN},np}^{0}.\]
	Then $E_{\mathrm{HN},np}^{(p)}=\mathcal{E}_{\mathrm{HN},np}^{(p)}\otimes_{R_a}K$. 
	
Finally, for any $\omega\in\Omega\setminus\Omega_a$, set $\mathcal{E}_{np,\omega}\coloneq \mathcal{E}_{np}\otimes_{R_a}K_\omega^\circ$ and $\mathcal{E}_{\mathrm{HN},np,\omega}^{(p)}\coloneq \mathcal{E}_{\mathrm{HN},np}^{(p)}\otimes_{R_a}K_\omega^\circ$. Write $\metr_{\mathcal{E}_{np},\omega}$ and $\metr_{\mathcal{E}_{\mathrm{HN},np}^{(p)},\omega}$ for the norms induced by $\mathcal{E}_{np,\omega}$ and $\mathcal{E}_{\mathrm{HN},np,\omega}^{(p)}$, respectively. By construction, $\mathcal{E}_{\mathrm{HN},np}^{(p)}$ is saturated in $\mathcal{E}_{np}$, so the quotient $\mathcal{E}_{np}/\mathcal{E}_{\mathrm{HN},np}^{(p)}$ is torsion-free over $R_a$; consequently,
\[\mathcal{E}_{np,\omega}/\mathcal{E}_{\mathrm{HN},np,\omega}^{(p)}\simeq \left(\mathcal{E}_{np}/\mathcal{E}_{\mathrm{HN},np}^{(p)}\right)\otimes_{R_a}K_\omega^\circ\]
is also torsion-free. This implies that $\mathcal{E}_{\mathrm{HN},np,\omega}^{(p)}$ is saturated in $\mathcal{E}_{np,\omega}$, and therefore $\metr_{\mathcal{E}_{\mathrm{HN},np}^{(p)},\omega}$ is simply the restriction of $\metr_{\mathcal{E}_{np},\omega}$ to $E^{(p)}_{\mathrm{HN},np,\omega}$. Moreover, since $\mathcal{X}_a$ is smooth over $D(\Omega_a)$, \cite[Lemma~6.3~(iii)]{boucksom2021spaces} implies that $\metr_{\mathcal{E}_{np},\omega}$ coincides with the super-norm $\metr_{\sup,\omega}$, and hence so does $\metr_{\mathcal{E}_{\mathrm{HN},np}^{(p)},\omega}$.
\end{art}

\begin{theorem}\label{thm:fujita approximation}
	Let $X$ be a projective variety over $K$ of dimension $d=\dim(X)$, and let $\overline{L}\in \widehat{\Pic}_{S}(X)_\mo$. If $\overline{L}$ is arithmetically big, then for every $\varepsilon>0$ there exist a birational morphism $\varphi\colon X'\to X$ and a positive integer $p$, together with a decomposition 
	\[\varphi^*(\overline{L}^{\otimes p}) =\overline{A}\otimes\overline{M}\]
	satisfying:
	\begin{enumeratea}
		\item $\overline{A} \in \widehat{\Pic}_{S}(X')_\mo$ is $S$-ample and $\overline{M}\in \widehat{\Pic}_{S}(X')_\mo$ is effective;
		\item $p^{-(d+1)}\widehat{\vol}(\overline{A})\geq \widehat{\vol}(\overline{L})-\varepsilon.$
	\end{enumeratea}
\end{theorem}

\begin{proof}
	Let $\mathcal{L}$ be a line bundle on a model $\pi\colon\mathcal{X}\to T$ of $X$ inducing $\overline{L}$, and keep the notation of \cref{algebra for line bundles}. Furthermore, for a $D(\Omega_a)$-scheme $\mathcal{X}_a'$ whose generic fiber is $X$, a line bundle $\mathcal{M}_a\in \Pic(\mathcal{X}_a')$, and metrics $\metr_{M,v}$ on the pullback $M_v$ of $M\coloneq \mathcal{M}_a|_{X}$ to $X_v$ for $v\in \Omega_a$, we view the pair $\overline{\mathcal{M}_a}\coloneq (\mathcal{M}_a,(\metr_{M,v})_{v\in\Omega_a})$ as an element of $\widehat{\Pic}_S(X)$ by taking the corresponding model metrics induced by $\mathcal{M}_a$. We denote the set of such metrized line bundles by $\widehat{\Pic}_S(\mathcal{X}_a')$ for simplicity. Note that $\widehat{\Pic}_{S}(\mathcal{X}_a')\otimes_\Z\Q\subset \widehat{\Pic}_{S,\Q}(X)_\cpt^\YZ$ (see \cref{def:boundarytopologyglobal}). In what follows we set $\mathcal{X}_a\coloneq \mathcal{X}\times_TD(\Omega_a)$ and $\mathcal{L}_a\coloneq \mathcal{L}|_{\mathcal{X}_a}$, and we prove \cref{thm:fujita approximation} in three steps.
	
	\vspace{2mm} \noindent
	\textbf{Step 1.} {\it For any $\varepsilon>0$, there exists a birational morphism $\varphi_a\colon \mathcal{X}_a'\to \mathcal{X}_a$, an integer $p\in \N_{\geq 1}$, and a decomposition $\varphi_a^*\overline{L}^{\otimes p}\simeq \overline{\mathcal{G}_a}\otimes\overline{\mathcal{M}_a}$ in $\widehat{\Pic}_S(X)$ such that $\overline{\mathcal{G}_a}\in \widehat{\Pic}_S(\mathcal{X}_a')$ is $S$-nef, $\overline{\mathcal{M}_a}\in \widehat{\Pic}_S(\mathcal{X}_a')$ is effective, and
		\[p^{-(d+1)}\widehat{\vol}(\overline{\mathcal{G}_a})\geq \widehat{\vol}(\overline{L})-\varepsilon.\]}
	
	\vspace{2mm} \noindent
	We keep the notation of the proof of \cite[Theorem~5.1]{liu2024arithmetic}. Step~1 follows directly from that proof together with the following observations.
	\begin{enumerate}
		\item For any $p\in\N_{\geq 1}$ such that $E_{\mathrm{HN},p}^0\not=0$ (see \cite[Proposition~4.2]{liu2024arithmetic}), the generic fiber of 
		\[\phi_p\colon \mathcal{X}_{a,p}\coloneq\mathrm{Proj}\left(\Im\left(\bigoplus_{n\geq 0}\pi^*\mathcal{E}_{\mathrm{HN},np}^{(p)}\to\bigoplus_{n\geq 0}\mathcal{L}_a^{\otimes np}\right)\right) \longrightarrow \mathcal{X}_a\]
		over $D(\Omega_a)=\Spec(R_a)$ coincides with the blow-up $X_p$ of $X$ constructed in the proof of \cite[Theorem~5.1]{liu2024arithmetic}. Indeed, $E_{\mathrm{HN},np}^{(p)}=\mathcal{E}_{\mathrm{HN},np}^{(p)}\otimes_{R_a}K$ for all $n\geq 0$, see \cref{algebra for line bundles}.
		
		\item\label{proof:fact 2} For any $\omega\in \Omega\setminus\Omega_a$, the norm on ${E}_{\mathrm{HN},np,\omega}^{(p)}$ induced by $\mathcal{E}_{\mathrm{HN},np}^{(p)}$ coincides with the super-norm $\metr_{\sup,\omega}$, as explained in \cref{algebra for line bundles}.
		
		\item Let $\mathcal{G}_{a,p}\coloneq \OO_{\mathcal{X}_{a,p}}(1)$, and $\mathcal{M}_{a,p}$ the line bundle on $\mathcal{X}_{a,p}$ corresponding to the exceptional divisor. Then $\phi_p^*\mathcal{L}_a^{\otimes p}=\mathcal{G}_{a,p}\otimes\mathcal{M}_{a,p}$. Consider the surjective homomorphism
		\begin{align}\label{eq:quotient norm}
			\phi_p^*\pi^*\mathcal{E}_{\mathrm{HN},np}^{(p)} \longrightarrow \mathcal{G}_{a,p}^{\otimes n}.
		\end{align}  
		By \ref{proof:fact 2} and \cite[Proposition~2.3.12~(3)]{chen2020arakelov}, for any $\omega\in\Omega\setminus\Omega_a$ the quotient metric on $G_{a,p}^{\otimes n}=\mathcal{G}_{a,p}^{\otimes n}|_{X}$ induced by $\metr_{\sup,\omega}$ on $\mathcal{E}_{\mathrm{HN},np}^{(p)}$ coincides with the model metric induced by $\mathcal{G}_{a,p}^{\otimes n}$. For $\omega\in\Omega_a$, as in the proof of \cite[Theorem~5.1]{liu2024arithmetic}, we endow $G_{a,p}$ with quotient norm $\metr_{n,\omega}$ induced by \eqref{eq:quotient norm} by restriction on generic fiber. Hence $\overline{\mathcal{G}_{a,p}}=(\mathcal{G}_{a,p},(\metr_{n,\omega})_{\omega\in\Omega_a})\in \widehat{\Pic}_{S}(X_{p})$ coincides with the adelic line bundle $(G_p,\alpha_n)$ appearing in the proof of \cite[Theorem~5.1]{liu2024arithmetic}. We define $\overline{\mathcal{M}_{a,p}}\coloneq \phi_p^*\overline{L}^{\otimes p}\otimes \overline{\mathcal{G}_{a,p}}^{\otimes-1}$.
	\end{enumerate}
	As shown in \cite[Theorem~5.1]{liu2024arithmetic}, there exist $p\in\N_{\geq 1}$ and $n\in\N_{\geq1}$ such that $\overline{\mathcal{G}_{a,p}}$ and $\overline{\mathcal{M}_{a,p}}$ satisfy the properties required in Step~1.
	
		
	
	\vspace{2mm} \noindent
	\textbf{Step 2.} {\it Reduction to the case where $\overline{A}$ is $S$-nef. More precisely, for any $\varepsilon>0$ there exist a birational morphism $\varphi\colon X'\to X$, a positive integer $p$, and a decomposition $\varphi^*(\overline{L}^{\otimes p}) =\overline{A}\otimes\overline{M}$ with $\overline{A}\in\widehat{\Pic}_{S}(X')_\mo$ being $S$-nef and $\overline{M}\in\widehat{\Pic}_{S}(X')_\mo$ effective, such that
		\[p^{-(d+1)}\widehat{\vol}(\overline{A})\geq \widehat{\vol}(\overline{L})-\varepsilon.\]}
	
	\vspace{2mm} \noindent
	For any $c\in\R$, we denote by $\OO_a(c)\in \widehat{\Pic}_S(X)$ the trivial line bundle $\OO_X$ equipped with the metric $\metr_a$ defined by
	\[\|1\|_{a,\omega}=\begin{cases}
		1& \text{ if $\omega\in\Omega\setminus\Omega_a$,}\\[2pt]
		e^{-c}& \text{ if $\omega\in\Omega_a$.}
	\end{cases}\]
	Recall that $\Omega_a$ is a finite set, since it is determined by an ample divisor on $T$.
	
	Let $\varepsilon>0$. Since $\overline{L}$ is big, the continuity of volumes proved in \cite[Theorem~6.4.24]{chen2020arakelov} yields an integer $m\in\N_{\geq 1}$ such that 
	\[\overline{L_{m,\varepsilon}}\coloneq \overline{L}^{\otimes m}\otimes \OO_a\left(-\frac{\varepsilon}{2(d+1)\vol(L)\cdot\#\Omega_a}\right)\]
	is big. Here $\vol(L)$ denotes the geometric volume of the underlying line bundle $L$ on $X$. Let $\varphi_a\colon \mathcal{X}_a'\to \mathcal{X}_a$, $p\in\N_{\geq 1}$, and $\overline{\mathcal{G}_a} = (\mathcal{G}_a,(\metr_{G,\omega})_{\omega\in\Omega})$, $\overline{\mathcal{M}_a}\in \widehat{\Pic}_S(\mathcal{X}_a')$ be the data produced by Step~1 applied to $\overline{L_{m,\varepsilon}}$ and $\varepsilon/2$, so that
	\[p^{-(d+1)}\widehat{\vol}(\overline{\mathcal{G}_a})\geq \widehat{\vol}(\overline{L_{m,\varepsilon}})-\frac{\varepsilon}{2}.\]
	
	For each $\omega\in \Omega_a$, choose a model $(\mathcal{X}_\omega', \mathcal{G}_\omega')$ over $K_\omega^\circ$ of $(X, G_{a})$ such that
	\begin{align}\label{eq:compare metrics of arithnef}
		0\leq \log\frac{\metr_{G,\omega}}{\metr_{G,\omega}'}\leq \frac{\varepsilon p}{2(d+1)\vol(L)\cdot\#\Omega_a},
	\end{align}
	where $\metr_{G,\omega}'$ is the metric induced by $\mathcal{G}_\omega'$. Because $\Omega_a$ is finite, by taking a suitable birational join of $(\mathcal{X}_\omega')_{\omega\in\Omega_a}$ and $\mathcal{X}_a'$ over $\mathcal{X}$, we obtain a birational morphism $\psi\colon\mathcal{X}'\to \mathcal{X}$ of $T$-models of $X$ and a line bundle $\mathcal{G}\in \Pic(\mathcal{X}')$ modelling $G_{a}|_{X}$ such that 
	\[\metr_{G,\omega}'' =\begin{cases}
		\metr_{G,\omega} & \text{ if $\omega\in\Omega\setminus\Omega_a$,}\\[2pt]
		\metr_{G,\omega}'& \text{ if $\omega\in\Omega_a$,}
	\end{cases}\]
	where $\metr_{G}''=(\metr_{G,\omega}'')_{\omega\in\Omega}$ denotes the family of metrics induced by $\mathcal{G}$. We may assume that $\mathcal{X}'|_{D(\Omega_a)}=\mathcal{X}_a'$. Let $\varphi\colon X'\to X$ be the morphism induced by $\psi$ on generic fibers.
	
	Since $\overline{\mathcal{G}_{a}}$ is $S$-nef and $-\log\metr_{G,\omega}\leq -\log\metr_{G,\omega}''$ for any $\omega\in\Omega$, we have that $\overline{G''}\coloneq(G,\metr_G'')$ is $S$-nef (this can be deduced by \cite[Proposition~9.1.7]{chen2022hilbert}). Moreover, the choice of $m$ gives
	\[\widehat{\vol}(\overline{L_{m,\varepsilon}})\geq \widehat{\vol}(\overline{L}^{\otimes m})-(d+1)\vol(L)\cdot \frac{\varepsilon}{2(d+1)\vol(L)\cdot\#\Omega_a}\cdot \#\Omega_a = m^{d+1}\widehat{\vol}(\overline{L})-\frac{\varepsilon}{2}.\]
	Consequently,
	\begin{align*}
		(mp)^{-(d+1)}\widehat{\vol}(\overline{G''})\geq &(mp)^{-(d+1)}\widehat{\vol}(\overline{\mathcal{G}_a})\\
		\geq& m^{-(d+1)}\left(\widehat{\vol}(\overline{L_{m,\varepsilon}})-\frac{\varepsilon}{2}\right)\\
		\geq& \widehat{\vol}(\overline{L})-\frac{\varepsilon}{2m^{d+1}}-\frac{\varepsilon}{2m^{d+1}}\\
		\geq& \widehat{\vol}(\overline{L})-\varepsilon.
	\end{align*}
	Let $\overline{M}\coloneq \varphi^*\overline{L}^{\otimes mp}\otimes\overline{G''}^{\otimes-1}$. Using \eqref{eq:compare metrics of arithnef}, we may rewrite
	\begin{align*}
		\varphi^*\overline{L}^{\otimes mp}\otimes\overline{G''}^{\otimes-1}
		=&\;\varphi^*\overline{L_{m,\varepsilon}}^{\otimes p}\otimes\overline{\mathcal{G}_{a}}^{\otimes -1}\otimes \left(\overline{\mathcal{G}_{a}}\otimes \overline{G''}^{\otimes-1}\otimes \OO_a\left(\frac{\varepsilon p}{2(d+1)\vol(L)\cdot\#\Omega_a}\right)\right)\\
		=&\;\overline{\mathcal{M}_{a}}\otimes\left(\overline{\mathcal{G}_{a}}\otimes \overline{G''}^{\otimes-1}\otimes \OO_a\left(\frac{\varepsilon p}{2(d+1)\vol(L)\cdot\#\Omega_a}\right)\right).
	\end{align*}
	We claim that the second factor is effective.  Indeed, for $\omega\in\Omega_a$ we have
	\[
	-\log\|1\|_{\omega}
	= -\log\|1\|_{{G},\omega} + \log\|1\|_{G,\omega}'' + \frac{\varepsilon p}{2(d+1)\vol(L)\cdot\#\Omega_a},
	\]
	where $\|1\|_{\omega}$ denotes the metric of the second factor at $\omega$. By \eqref{eq:compare metrics of arithnef},
	\[
	0 \leq \log\frac{\|1\|_{G,\omega}}{\|1\|_{G,\omega}'} 
	= \log\|1\|_{G,\omega} - \log\|1\|_{G,\omega}'' 
	\leq \frac{\varepsilon p}{2(d+1)\vol(L)\cdot\#\Omega_a},
	\]
	whence
	\[
	-\log\|1\|_{\omega}
	= \left( -\log\|1\|_{{G},\omega} + \log\|1\|_{G,\omega}'' \right) 
	+ \frac{\varepsilon p}{2(d+1)\vol(L)\cdot\#\Omega_a}
	\geq 0.
	\]
For $\omega\notin\Omega_a$, we have $\|1\|_{G,\omega}=\|1\|_{G,\omega}''$ and $\|1\|_{a,\omega}=1$ by construction, so $-\log\|1\|_{\omega}=0$. Thus $-\log\|1\|_{\omega}\geq 0$ for all $\omega\in\Omega$, which means exactly that the second factor is effective. Since $\overline{\mathcal{M}_a}$ is effective by Step~1, it follows that $\overline{M}$ is effective. This completes Step~2.
	
	\vspace{2mm} \noindent
	\textbf{Step 3.} {\it \cref{thm:fujita approximation} holds.}
	
	\vspace{2mm} \noindent
	We follow the idea of the proof of \cite[Corollary~5.2]{liu2024arithmetic}. By \cref{S-ample}, there exists an $S$-ample line bundle $\overline{B}\in \widehat{\Pic}_S(X)_\mo$, which is big by \cite[Proposition~9.2.2]{chen2022hilbert}. Let $\varepsilon\in\R_{>0}$. By \cite[Theorem~6.4.24]{chen2020arakelov}, there is an integer $m\in\N_{\geq 2}$ such that $\overline{L}^{\otimes m}\otimes \overline{B}^{\otimes-1}$ is big and
	\[m^{-(d+1)}\widehat{\vol}\left(\overline{L}^{\otimes m}\otimes \overline{B}^{\otimes -1}\right)\geq \widehat{\vol}(\overline{L})-\frac{\varepsilon}{2}.\] 
	Applying Step~2 to $\overline{L}^{\otimes m}\otimes \overline{B}^{\otimes -1}$ and $\varepsilon/2$, we obtain a birational morphism $\varphi\colon X'\to X$, an integer $p\in\N_{\geq 1}$, and a decomposition 
	\[\varphi^*\left(\overline{L}^{\otimes mp}\otimes \overline{B}^{\otimes -p}\right)=\overline{G}\otimes \overline{M}\]
	with $\overline{G}\in\widehat{\Pic}_{S}(X')_\mo$ being $S$-nef and $\overline{M}\in\widehat{\Pic}_{S}(X')_\mo$ effective, such that
	\[p^{-(d+1)}\widehat{\vol}(\overline G)\geq \widehat{\vol}\left(\overline{L}^{\otimes m}\otimes \overline{B}^{\otimes -1}\right)-\frac{\varepsilon}{2}.\]
	Let $\overline{B'}\in\widehat{\Pic}_S(X')_\mo$ be an $S$-ample line bundle; such a line bundle exists by \cref{S-ample}. Since $\varphi^*\overline{B}$ is big, there exists $q\in \N_{\geq 1}$ such that $\varphi^*\overline{B}^{\otimes pq}\otimes\overline{B'}^{\otimes -1}$ is big. By the effectivity of big model metrics (see \cref{big implies effective}), this line bundle is effective. We then have the decomposition
	\[\varphi^*\overline{L}^{\otimes mpq} = (\overline{G}^{\otimes q}\otimes \overline{B'})\otimes \left(\overline{M}^{\otimes q}\otimes \varphi^*\overline{B}^{\otimes pq}\otimes \overline{B'}^{\otimes -1}\right).\]
	Here $\overline{A}\coloneq \overline{G}^{\otimes q}\otimes \overline{B'}$ is $S$-ample (as the tensor product of an $S$-nef and an $S$-ample line bundle), the second factor is effective, and 
	\begin{align*}
		(mpq)^{-(d+1)}\widehat{\vol}(\overline{A})\geq &(mpq)^{-(d+1)}\widehat{\vol}(\overline{G}^{\otimes q})\\
		\geq& m^{-(d+1)}\left(\widehat{\vol}(\overline{L}^{\otimes m}\otimes \overline{B}^{\otimes-1})-\frac{\varepsilon}{2}\right)\\
		\geq&\widehat{\vol}(\overline{L})-\frac{\varepsilon}{2}-\frac{\varepsilon}{2m^{d+1}}\\
		\geq&\widehat{\vol}(\overline{L})-\varepsilon.
	\end{align*}
	This proves Step~3, and hence finishes the proof of \cref{thm:fujita approximation}.
\end{proof}

\subsection{Positive intersection product}

If $\dim(T)\geq 2$, let $T_{\mathbf{c}}$ and $S_{\mathbf{c}}=(K_{\mathbf{c}}, \Omega_{\mathbf{c}},\mathcal{A}_{\mathbf{c}},\nu_{\mathbf{c}})$ be as in \cref{generic curves}. When $\dim T=1$, for consistency, the notation $T_{\mathbf c}$ and $S_{\mathbf c}$ mean $T$ and $S$, respectively. Let $\alpha=(\alpha^\#,\alpha_\#,1)\colon S_{\mathbf{c}}\to S$ be the covering defined in \cref{prop:B is covered by a curve as adelic curves}. As before, for an algebraic variety $Y$ over $K$, we denote by $Y_{\mathbf c}$ its base change via $K_{\mathbf c}/K$. Similarly for metrized line bundles.

In the rest of this section, we fix a normal quasi-projective variety $U$.

\begin{lemma}\label{comparison of arithmetic volume lemma}
	Let $\overline{L}\in\widehat{\Pic}_{S,\Q}(U)_\cpt^\YZ$. Then 
	\[\widehat{\vol}(\overline{L_{\mathbf c}})\geq \widehat{\vol}(\overline{L}).\]
\end{lemma}
\begin{proof}
	By \eqref{eq:convergence of arithmetic volume}, it suffices to consider the case where $U=X$ is projective over $K$ and $\overline{L}\in \widehat{\Pic}_{S,\Q}(X)_\mo$. In this case, $H^0(X_{\mathbf{c}},L_{\mathbf{c}})=H^0(X,L)\otimes_K K_{\mathbf{c}}$, and for any $\omega\in\Omega$, $\omega'\in \Omega_{\mathbf c}$ over $\omega$, the super-norm $\metr_{\mathbf c,\sup,\omega'}$ on $H^0(X_{\mathbf{c}},L_{\mathbf{c}})$ at $\omega'$ extends the super-norm $\metr_{\sup,\omega}$ on $H^0(X,L)$. Therefore, by \cite[Proposition~1.1.66]{chen2020arakelov}, for any subspace $W\subset  H^0(X,L)$ with a basis $s_1,\dots, s_N$, one has
	\[\|s_1\wedge\cdots\wedge s_N\|_{\sup,\omega,\det}\geq \|s_1\wedge\cdots\wedge s_N\|_{\mathbf c,\sup,\omega',\det},\]
	so that $\widehat{\deg}((W,\metr_{\sup}))\leq \widehat{\deg}((W\otimes K_{\mathbf c},\metr_{\mathbf c, \sup}))$. Taking suprema over all subspaces $W$ and passing to the limit, we obtain $\widehat{\vol}(\overline{L}) \leq \widehat{\vol}(\overline{L_{\mathbf c}})$.
\end{proof}

\begin{definition}\label{positive intersection def}
	Let $\overline{L}\in \widehat{\Pic}_{S,\Q}(U)_\cpt^\YZ$ and $\overline{M}\in \widehat{\Pic}_{S,\Q}(U)_{\nef}^\YZ$. We define the \emph{positive intersection product} of $\overline{L}$ with $\overline{M}$ as 
	\[\langle\overline{L}^d\rangle\cdot \overline{M}\coloneq \sup\limits_{(\overline{A},U',\pi)}(\overline{A}^{d}\cdot \pi^*\overline{M}\mid U')_S,\]
	where $(\overline{A},U',\pi)$ runs over all triples such that $\pi\colon U'\to U$ is a birational morphism of normal quasi-projective varieties, and $\overline{A}\in P_{S,\Q}(U')_\mo$ is a model, $S$-nef $\Q$-line bundle on some projective model $X'$ of $U'$ such that $\overline{A}\leq\pi^*\overline{L}$ (such a triple $(\overline{A},U',\pi)$ is called an \emph{admissible approximation} of $\overline{L}$).  
\end{definition}

\begin{remark}
	If $\overline{L}$ has an admissible approximation, then $\overline{L}$ is big. Conversely, by the continuity of arithmetic volume \eqref{eq:convergence of arithmetic volume} and \cref{thm:fujita approximation}, if $\overline{L}$ is big, then it has an admissible approximation. In the degenerate case where $\overline{L}$ is not big (we will not discuss this case), i.e.~$\overline{L}$ has no admissible approximation, we set $\langle\overline{L}^d\rangle\cdot \overline{M}=0$. 
\end{remark}

\begin{remark}
	By definition, it is easy to see that $\langle\overline{L}^d\rangle\cdot \overline{M}$ is stable under birational pull-backs. 
\end{remark}

\begin{remark}\label{rmk:basic properties of pip}
	By definition, it is easy to show that for any $\lambda\in\Q_{>0}$, we have 
	\begin{itemize}
		\item $\langle(\overline{L}^{\otimes \lambda})^d\rangle\cdot \overline{M} =\lambda^d\cdot (\langle\overline{L}^d\rangle\cdot \overline{M})$;
		\item $\langle\overline{L}^d\rangle\cdot (\overline{M}^{\otimes \lambda})=\lambda\cdot (\langle\overline{L}^d\rangle\cdot \overline{M})$.
	\end{itemize} 
	If $\overline{L'}\in\widehat{\Pic}_{S,\Q}(U)_\cpt^\YZ$ with $\overline{L}\leq \overline{L'}$, then $\langle\overline{L}^d\rangle\cdot \overline{M} \leq \langle\overline{L'}^d\rangle\cdot \overline{M}$. 
\end{remark}

\begin{lemma}\label{fujita type condition lemma}
	Assume that $\dim T=1$. Let $\overline{L}\in \widehat{\Pic}_{S,\Q}(U)_\cpt^\YZ$ be a big compactified $S$-metrized line bundle, and $((\overline{A_m},U_m,\pi_m))_{m\in\N_{\geq 1}}$ a sequence of admissible approximations of $\overline{L}$. If $\lim_{m\to\infty}(\overline{A_m}^{d+1}\mid U_m)_S = \widehat{\vol}(\overline{L})$,
	then for $\overline{M}\in \widehat{\Pic}_{S,\Q}(U)_{\nef}^\YZ$, one has
	\[\lim_{m\to\infty}(\overline{A_m}^{d}\cdot \overline{M}\mid U_m)_S = \langle\overline{L}^d\rangle\cdot \overline{M}.\]
\end{lemma}
\begin{proof}
	By definition of positive intersection product, one has 
	\begin{align}\label{eq:differential 1}
		\limsup_{m\to\infty}(\overline{A_m}^d\cdot \overline{M}\mid U_m)_S \leq \langle\overline{L}^d\rangle\cdot \overline{M}.
	\end{align}
	Conversely, for any $t\in\Q_{>0}$ and any $m\in\N_{\geq 1}$, by \cite[Theorem~5.2.2]{yuan2021adelic}, we have 
	\[\widehat{\vol}(\overline{L}\otimes\overline{M}^{\otimes (-t)})\geq \widehat{\vol}(\overline{A_m}\otimes\overline{M}^{\otimes (-t)}) \geq (\overline{A_m}^{d+1}\mid U_m)_S-(d+1)\cdot t\cdot (\overline{A_m}^{d}\cdot \overline M\mid U_m)_S,\]
	i.e.
	\begin{align}\label{eq:differential converse}
		\frac{(\overline{A_m}^{d+1}\mid U_m)_S-\widehat{\vol}(\overline{L}\otimes\overline{M}^{\otimes (-t)})}{t}\leq (d+1)(\overline{A_m}^{d}\cdot \overline M\mid U_m)_S.
	\end{align}
	After taking $\liminf_{m\to\infty}$ and $\lim_{t\to0^+}$ in \eqref{eq:differential converse}, by \cite[Theorem~3.13]{nijerdiff}, one has 
	\begin{align}\label{eq:differential converse1}
		(d+1)\langle\overline{L}^d\rangle\cdot \overline{M}\leq (d+1)\liminf_{m\to\infty}(\overline{A_m}^{d}\cdot \overline M\mid U_m)_S.
	\end{align}
	Then \eqref{eq:differential 1} and \eqref{eq:differential converse1} imply that $\lim_{m\to\infty}(\overline{A_m}^d\cdot \overline{M}\mid U_m)_S = \langle\overline{L}^d\rangle\cdot \overline{M}$. This completes the proof.
\end{proof}

\begin{lemma}\label{vol equal lemma}
	Let $\overline{L}\in \widehat{\Pic}_{S,\Q}(U)_\cpt^\YZ$ be a big compactified $S$-metrized line bundle. If $\widehat{\vol}(\overline{L}) = \widehat{\vol}(\overline{L_{\mathbf c}})$, then for any $\overline{M}\in \widehat{\Pic}_{S,\Q}(U)_{\nef}^\YZ$, one has
	\[\langle\overline{L}^d\rangle\cdot \overline{M} = \langle\overline{L_{\mathbf{c}}}^d\rangle\cdot \overline{M_{\mathbf c}}.\]
\end{lemma}
\begin{proof}
	Let $((\overline{A_m},U_m,\pi_m))_{m\in\N_{\geq 1}}$ be a sequence of admissible approximations of $\overline{L}$ such that $\lim_{m\to\infty}(\overline{A_m}^{d+1}\mid U_m)_S = \widehat{\vol}(\overline{L})$. Such a sequence exists by the Fujita approximation theorem over $S$, see \cref{thm:fujita approximation}. Then $((\overline{A_{m,\mathbf c}},U_{m,\mathbf{c}},\pi_{m,\mathbf c}))_{m\in\N_{\geq 1}}$ is a sequence of admissible approximations of $\overline{L_{\mathbf c}}$ and 
	\[\lim_{m\to\infty}(\overline{A_{m,\mathbf c}}^{d+1}\mid U_{m,\mathbf c})_{S_{\mathbf c}} =\lim_{m\to\infty}(\overline{A_m}^{d+1}\mid U_m)_S = \widehat{\vol}(\overline{L}) = \widehat{\vol}(\overline{L_{\mathbf c}}).\]
	Notice that every admissible approximation of $\overline{L}$ over $S$ induces one of $\overline{L_{\mathbf c}}$ over $S_{\mathbf c}$ by base change, so the supremum over $S_{\mathbf c}$ is at least that over $S$. Consequently, for any $\overline{M}\in \widehat{\Pic}_{S,\Q}(U)_{\nef}^\YZ$ and any $m\in\N_{\geq 1}$, one has
	\[ \langle\overline{L_{\mathbf{c}}}^d\rangle\cdot \overline{M_{\mathbf c}} \geq \langle\overline{L}^d\rangle\cdot \overline{M} \geq (\overline{A_m}^{d}\cdot \overline{M}\mid U_m)_S = (\overline{A_{m,\mathbf c}}^{d}\cdot \overline{M_{\mathbf c}}\mid U_{m,\mathbf c})_{S_{\mathbf c}}.\]
	Since $\dim T_{\mathbf c}=1$, Lemma \ref{fujita type condition lemma} applies to the adelic curve $S_{\mathbf c}$ and yields
	\[\langle\overline{L_{\mathbf{c}}}^d\rangle\cdot\overline{M_{\mathbf c}}=\lim_{m\to\infty} (\overline{A_{m,\mathbf c}}^{d}\cdot \overline{M_{\mathbf c}}\mid U_{m,\mathbf c})_{S_{\mathbf c}}.\]
	Hence $\langle\overline{L_{\mathbf{c}}}^d\rangle\cdot \overline{M_{\mathbf c}}= \langle\overline{L}^d\rangle\cdot \overline{M}$.
\end{proof}

\begin{art}\label{measure given by positive intersection}
	Let $\overline{L}=(L,\metr)\in\widehat{\Pic}_{S,\Q}(U)_\cpt^\YZ$ be a big compactified $S$-metrized line bundle. Assume that $\widehat{\vol}(\overline{L}) = \widehat{\vol}(\overline{L_{\mathbf c}})$. By \cref{vol equal lemma}, and as in \cite{nijerdiff}, for any $\overline{M} \in\widehat{\Pic}_{S,\Q}(U)_\integrable^\YZ$, the positive intersection product $\langle\overline{L}^d\rangle\cdot \overline{M}$ can be defined by linearity (since $\widehat{\Pic}_{S,\Q}(U)_\integrable^\YZ$ is generated by differences of nef elements), and we have 
	\[\langle\overline{L}^d\rangle\cdot \overline{M} = \langle\overline{L_{\mathbf{c}}}^d\rangle\cdot \overline{M_{\mathbf c}}.\] 	
	Let $v\in\Omega$. For any $f_{v}\in C_c(U_{v}^\an)$, we set $\overline{\OO_U(f_v)} =(\OO_U, (\metr_\omega)_{\omega\in \Omega})$ to be the $S$-metrized line bundle such that  
	\[-\log\|1\|_\omega\coloneq\begin{cases}
		f_v & \text{ if $\omega=v$;}\\[2pt]
		0& \text{ if $\omega\in\Omega\setminus\{v\}$,}
	\end{cases}\] 
	which lies in $\widehat{\Pic}_{S,\Q}(U)_{\integrable}^\YZ$. If $f_v\geq 0$, then $\overline{\OO_U(f_v)}$ is effective and hence $\langle\overline{L}^d\rangle\cdot\overline{\OO_U(f_v)}\geq 0$. Therefore the functional $f_v\mapsto \frac{\langle\overline{L}^d\rangle\cdot \overline{\OO_U(f_v)} }{\vol(\widetilde L)}$ is positive, where $\widetilde{L}$ is the image of $\overline L$ in $\widetilde{\Pic}_\Q(U)_\cpt$, and $\vol(\widetilde L)$ is the geometric volume of the compactified line bundle $L$ (see \cite[5.2.2]{yuan2021adelic} or \cite[3.1.3]{biswas2024concave}). By the Riesz representation theorem for measures, it defines a Borel measure $\mu_{\overline{L}, v}$ on $U_{v}^\an$.
\end{art}

\subsection{Equidistribution for big case}

\begin{definition}\label{def:small point wrt big}
	Let $\overline{L}=(L,\metr)\in\widehat{\Pic}_{S,\Q}(U)_\cpt^\YZ$, and $\widetilde{L}$  the image of $\overline L$ in $\widetilde{\Pic}_\Q(U)_\cpt$. Assume that $\overline{L}$ is big. We say that $(x_m)_{m\in I}\subset U(\overline{K})$ is \emph{small with respect to $\overline{L}$} if
	\[\lim\limits_{m\in I}h_{\overline{L}}(x_m) = \frac{\widehat{\vol}(\overline{L})}{(d+1)\vol(\widetilde L)}.\]

\end{definition}

\begin{lemma}\label{small for x_c lemma}
	Let $\overline{L}\in\widehat{\Pic}_{S,\Q}(U)_\cpt^\YZ$. Assume that $\overline{L}$ is big. Let $(x_m)_{m\in I}\subset U(\overline{K})$. Assume that $(x_m)_{m\in I}\subset U(\overline{K})$ is a generic net and small with respect to $\overline{L}$. Then 
	\begin{enumerate1}
		\item\label{small for x_c lemma 1} $\widehat{\vol}(\overline{L}) = \widehat{\vol}(\overline{L_{\mathbf c}})$;
		\item\label{small for x_c lemma 2} the pull-back $(x_{m,\mathbf{c}})_{m\in I}\subset U_{\mathbf c}(\overline{K_{\mathbf c}})$ is a generic net and small with respect to $\overline{L_{\mathbf c}}$.
	\end{enumerate1}
\end{lemma}
\begin{proof}
	By \cref{lemma:generic stable under base changes}, $(x_{m,\mathbf{c}})_{m\in I}$ is a generic net in $U$. Then
	$$\limsup\limits_{m\in I}h_{\overline{L_{\mathbf{c}}}}(x_{m,\mathbf c})\geq \zeta_{\mathrm{ess}}(\overline{L_{\mathbf{c}}})\geq  \frac{\widehat{\vol}(\overline{L_{\mathbf{c}}})}{(d+1)\vol(L_{\mathbf{c}})}.$$
	The last inequality is from \cite[Theorem~4.5.3]{biswas2024concave} (it holds if we replace $\widehat{\vol}_\chi^\mathrm{num}(\overline{L_{\mathbf{c}}})$ by $\widehat{\vol}(\overline{L_{\mathbf{c}}})$). Since $\vol(L_{\mathbf{c}}) =\vol(L)$ and $\widehat{\vol}(\overline{L_{\mathbf{c}}}) \geq \widehat{\vol}(\overline{L})$ (by \cref{comparison of arithmetic volume lemma}), we have 
	\[\lim\limits_{m\in I}h_{\overline{L}}(x_m)=\lim\limits_{m\in I}h_{\overline{L_{\mathbf{c}}}}(x_{m,\mathbf c})=\limsup\limits_{m\in I}h_{\overline{L_{\mathbf{c}}}}(x_{m,\mathbf c})\geq \frac{\widehat{\vol}(\overline{L_{\mathbf{c}}})}{(d+1)\vol(L_{\mathbf{c}})}\geq \frac{\widehat{\vol}(\overline{L})}{(d+1)\vol(L)}\]
	where the first equality follows from \cref{prop:intersection number after base changes}. By our assumption that $(x_m)_{m\in I}$ is small with respect to $\overline{L}$, i.e. $\lim\limits_{m\in I}h_{\overline{L}}(x_m)= \frac{\widehat{\vol}(\overline{L})}{(d+1)\vol(\widetilde L)}$, one has 
	\[\lim\limits_{m\in I}h_{\overline{L_{\mathbf{c}}}}(x_{m,\mathbf c})= \frac{\widehat{\vol}(\overline{L_{\mathbf{c}}})}{(d+1)\vol(L_{\mathbf{c}})}.\]
	This completes the proof of the lemma.
\end{proof}

\begin{theorem}
	\label{theorem:equidsitributionforfinitenergy}
Let $\overline{L}=(L,\metr)\in \widehat{\Pic}_{S,\Q}(U)^{\YZ}_{\cpt}$. Assume that $\overline{L}$ is big.  
Let $(x_m)_{m\in I}$ be a generic net of points which is small with respect to $\overline{L}$. Then  $\widehat{\vol}(\overline{L}) = \widehat{\vol}(\overline{L}_{\mathbf c})$. Moreover, for any $v\in\Omega$ and any $f_v\in C_c(U_v^\an),$ 
\begin{align}\label{eq:equidistribution}
	\lim_{m\in I} \frac{\nu(v)}{\#O(x_m)}\sum_{x^\sigma_m\in O(x_m)}f_v(x^\sigma_m)=\frac{\langle \overline{L}^d\rangle\cdot \overline{\OO_U(f_v)}}{\vol(\widetilde{L})},
\end{align}
where $\langle \overline{L}^d\rangle\cdot \overline{\OO_U(f_v)}$ is well-defined by \cref{measure given by positive intersection} and $\widetilde{L}$ is the image of $\overline L$ in $\widetilde{\Pic}_\Q(U)_\cpt$.
\end{theorem}
\begin{proof}
	By \cite[Theorem~6.5.3]{biswas2024concave}, the theorem holds when $\dim T=1$.
	
	When $\dim T\geq 2$, by \cref{small for x_c lemma}~\ref{small for x_c lemma 2}, the pull-back $(x_{m,\mathbf{c}})_{m\in I}\subset U_{\mathbf c}(\overline{K_{\mathbf c}})$ is a generic net and small with respect to $\overline{L_{\mathbf c}}$. 
	For any $\omega\in \alpha_\#^{-1}(v)$, let $\pi_{\omega}\colon U_{\mathbf{c},\omega}\to U_v$ be the canonical morphism. Since $\dim(T_{\mathbf{c}})=1$, the one-dimensional case established above yields
	\begin{align}\label{eq:equidistribution over generic curve big}
		\lim_{m\in I}\frac{1}{\#O(x_{m,\mathbf{c}})}\sum_{x_{m,\mathbf{c}}^\sigma\in O(x_{m,\mathbf{c}})}\pi^*_\omega(f_v)(x_{m,\mathbf{c}}^{\sigma})=\frac{\langle \overline{L_{\mathbf c}}^d\rangle\cdot \overline{\OO_{U_{\mathbf c}}(\pi_{\omega}^*(f_{v}))}}{\vol(\widetilde{L_{\mathbf c}})}.
	\end{align}
	where $\widetilde{L_{\mathbf c}}$ is the image of $\overline{L_{\mathbf c}}$ in $\widetilde{\Pic}_\Q(U_{\mathbf c})_\cpt$.
	By \cite[Proposition~2.3]{burgos2019the} (or \cref{prop:intersection on subvariety}) and \eqref{eq:equidistribution over generic curve big}, we have 
	\begin{align*}
		\lim_{m\in I}(\overline{\OO_{U}(f_{v})}_{\mathbf c}\mid x_{m,\mathbf{c}})_{S_{\mathbf{c}}}
		=&\lim_{m\in I}\sum_{\omega\in\alpha_\#^{-1}(v)}\frac{\nu_{\mathbf{c}}(\omega)}{\#O(x_{m,\mathbf{c}})}\sum_{x_{m,\mathbf{c}}^\sigma\in O(x_{m,\mathbf{c}})}\pi^*_\omega(f_v)(x_{m,\mathbf{c}}^{\sigma}) \\
		=& \sum_{\omega\in \alpha_\#^{-1}(v)}\frac{\langle \overline{L_{\mathbf c}}^d\rangle\cdot \overline{\OO_{U_{\mathbf c}}(\pi_{\omega}^*(f_{v}))}}{\vol(\widetilde{L_{\mathbf c}})}\\
		=&\frac{\langle \overline{L_{\mathbf c}}^d\rangle\cdot \overline{\OO_{U}(f_{v})}_{\mathbf{c}}}{\vol(\widetilde{L_{\mathbf c}})}.
	\end{align*}
	By \cref{prop:intersection number after base changes}, \cref{small for x_c lemma}~\ref{small for x_c lemma 1}, and \cref{vol equal lemma}, the above equality implies that
	\[\lim_{m\in I}(\overline{\OO_{U}(f_{v})}\mid x_m)_{S}=\frac{\langle \overline{L}^d\rangle\cdot \overline{\OO_U(f_v)}}{\vol(\widetilde L)},\]
	so
	\begin{align}\label{eq:equid for points big 1}
		\lim_{m\in I}\frac{1}{\#O(x_m)}\sum_{x_m^{\sigma}\in O(x_m)}f_v(x_m^{\sigma})= \frac{\langle \overline{L}^d\rangle\cdot \overline{\OO_U(f_v)}}{\vol(\widetilde L)}.
	\end{align} 
	Together with the first part of the proof, this establishes the theorem.
\end{proof}

\end{document}